\documentclass[11pt]{elsarticle}

\journal{Journal of Computational Physics}
\usepackage[margin=2.5cm]{geometry}

\makeatletter
\providecommand{\doi}[1]{%
	\begingroup
	\let\bibinfo\@secondoftwo
	\urlstyle{rm}%
	\href{http://dx.doi.org/#1}{%
		doi:\discretionary{}{}{}%
		\nolinkurl{#1}%
	}%
	\endgroup
}
\makeatother

\usepackage{graphicx}
\usepackage{dcolumn}
\usepackage{bm}
\usepackage{amsmath,esint}
\usepackage{amsmath}
\usepackage{amsfonts}
\usepackage{amssymb}
\usepackage{graphicx}
\usepackage{amsmath}
\usepackage{amsthm}
\usepackage{eucal}
\usepackage{amssymb}
\usepackage{mathrsfs}
\usepackage{float}
\usepackage{hyperref}
\usepackage{multirow}
\usepackage{booktabs}
\usepackage{graphicx}
\usepackage{subfigure}
\usepackage{dashrule}
\usepackage{adjustbox}
\usepackage[dvipsnames]{xcolor}
\usepackage{graphicx}
\usepackage{dcolumn}
\usepackage{bm}
\usepackage{natbib}
\usepackage{amssymb}
\usepackage{amsthm}
\usepackage{mathtools}
\usepackage{empheq}
\usepackage[most]{tcolorbox}
\usepackage{mathtools}
\usepackage{cleveref}
\usepackage{soul}

\usepackage{tikz}
\usetikzlibrary{shapes,arrows,chains}
\usetikzlibrary{calc}
\usetikzlibrary{positioning}

\usepackage{tgpagella}
\usepackage{newpxmath}

\usepackage{longtable}
\usepackage{booktabs}
\usepackage{multirow}
\usepackage{siunitx}

\renewcommand{\vec}[1]{\mathbf{#1}}

\usepackage{xspace}
\let\storeBeta=\beta
\renewcommand\beta{\relax\ifmmode{\storeBeta}\else{$\storeBeta$}\fi\xspace}

\let\storeAlpha=\alpha
\renewcommand\alpha{\relax\ifmmode{\storeAlpha}\else{$\storeAlpha$}\fi\xspace}

\newcommand{\abs}[1]{\left| #1 \right|} 
\renewcommand{\d}[2]{\frac{\mathrm{d} #1}{\mathrm{d} #2}} 
\newcommand{\pd}[2]{\frac{\partial #1}{\partial #2}} 
 
\let\baraccent=\= 
\renewcommand{\=}[1]{\stackrel{#1}{=}} 

\newcount\colveccount
\newcommand*\colvec[1]{
	\global\colveccount#1
	\begin{pmatrix}
		\colvecnext
	}
	\def\colvecnext#1{
		#1
		\global\advance\colveccount-1
		\ifnum\colveccount>0
		\\
		\expandafter\colvecnext
		\else
	\end{pmatrix}
	\fi
}

\newcommand{\norm}[1]{\left\lVert#1\right\rVert}

\def\tphi{\widetilde{\phi}}
\def\tmu{\widetilde{\mu}}
\def\tp{\widetilde{p}}
\def\tvi{\widetilde{v}_i}
\def\tvj{\widetilde{v}_j}
\def\tJi{\widetilde{J}_i}
\def\tJj{\widetilde{J}_j}
\def\tpsi{\widetilde{\psi}}
\def\tvecv{\widetilde{\vec{v}}}

\hypersetup{
	colorlinks,
	linkcolor={blue!50!black},
	citecolor={blue!50!black},
	urlcolor={blue!80!black},
	anchorcolor = {blue!80!black},
	filecolor = {blue!80!black},
	menucolor = {blue!80!black},
	runcolor = {blue!80!black}
}

\newtheorem{lemma}{Lemma}
\newtheorem{proposition}{Proposition}
\newtheorem{theorem}{Theorem}
\newtheorem{definition}{Definition}
\newtheorem{remark}{Remark}
\newtheorem{corollary}{Corollary}
\newtheorem{claim}{Claim}

\usepackage{pgfplots}
\pgfplotsset{
	table/search path={Figures},
}
	\usetikzlibrary{arrows,shapes}
	\pgfplotsset{compat=1.3}
	\pgfplotsset{
		discard if/.style 2 args={
			x filter/.code={
				\edef\tempa{\thisrow{#1}}
				\edef\tempb{#2}
				\ifx\tempa\tempb
				\def\pgfmathresult{inf}
				\fi
			}
		},
		discard if not/.style 2 args={
			x filter/.code={
				\edef\tempa{\thisrow{#1}}
				\edef\tempb{#2}
				\ifx\tempa\tempb
				\else
				\def\pgfmathresult{inf}
				\fi
			}
		}
	}
\usetikzlibrary{plotmarks}
\usetikzlibrary{spy}

\newcommand{\dendro}{\textsc{Dendro}}

\usepackage{xcolor}
\usepackage{listings}

\usepackage{xparse}

\NewDocumentCommand{\codeword}{v}{%
	\texttt{\textcolor{blue}{#1}}%
}

\definecolor{codegreen}{rgb}{0,0.6,0}
\definecolor{codegray}{rgb}{0.5,0.5,0.5}
\definecolor{codepurple}{rgb}{0.58,0,0.82}
\definecolor{backcolour}{rgb}{0.95,0.95,0.92}

\lstdefinestyle{mystyle}{
	backgroundcolor=\color{backcolour},   
	commentstyle=\color{codegreen},
	keywordstyle=\color{magenta},
	stringstyle=\color{codepurple},
	basicstyle=\ttfamily\footnotesize,
	breakatwhitespace=false,         
	breaklines=true,                 
	captionpos=b,                    
	keepspaces=true,                  
	showspaces=false,                
	showstringspaces=false,
	showtabs=false,                  
	tabsize=2
}

\usepackage{array}
\newcolumntype{C}[1]{>{\centering\arraybackslash}p{#1}}
\usepackage{subfigure}
\usetikzlibrary{arrows,automata}

\usepackage[disable]{todonotes}
\makeatletter
\if@todonotes@disabled

\else

\fi
\makeatother

\makeatletter
\if@todonotes@disabled

\else

\fi
\makeatother

\begin{document}

\begin{frontmatter}

\title{Simulating two-phase flows with thermodynamically consistent energy stable Cahn-Hilliard Navier-Stokes equations on parallel adaptive octree based meshes}

\author[isuMechEAddress]{Makrand A. Khanwale}
\ead{khanwale@iastate.edu}

\author[isuMechEAddress]{Alec D. Lofquist\fnref{AlecFootnote}}
\ead{alces14@gmail.com}

\author[utahAddress]{Hari Sundar}
\ead{hari@cs.utah.edu}

\author[isuMathAddress]{James A. Rossmanith\corref{correspondingAuthor}}
\ead{rossmani@iastate.edu}

\author[isuMechEAddress]{Baskar~Ganapathysubramanian\corref{correspondingAuthor}}
\ead{baskarg@iastate.edu}

\cortext[correspondingAuthor]{Corresponding authors}
\fntext[AlecFootnote]{Presently at Affirm Inc., San Francisco, CA, USA}

\address[isuMechEAddress]{Department of Mechanical Engineering, Iowa State University, Iowa, USA 50011}
\address[isuMathAddress]{Department of Mathematics, Iowa State University, Iowa, USA 50011}
\address[utahAddress]{School of Computing, The University of Utah, Salt Lake City, Utah, USA 84112}

\begin{abstract}
	We report on simulations of two-phase flows with deforming interfaces at various density contrasts by solving thermodynamically consistent  Cahn-Hilliard Navier-Stokes equations. An (essentially) unconditionally energy-stable Crank-Nicolson-type time integration scheme is used. Detailed proofs of energy stability of the semi-discrete scheme and for the existence of solutions of the advective-diffusive Cahn-Hilliard operator are provided. We discretize spatial terms with a conforming continuous Galerkin finite element method in conjunction with a residual-based variational multi-scale (VMS) approach in order to provide pressure stabilization. We deploy this approach on a massively parallel numerical implementation using fast octree-based adaptive meshes. A detailed scaling analysis of the solver is presented.  Numerical experiments showing convergence and validation with experimental results from the literature are presented for a large range of density ratios.     
\end{abstract}


\end{frontmatter}



\section{Introduction}
\label{sec:introduction}
    Accurate description of the dynamics of the interface in two-phase flows is essential from two perspectives: (1) accurate resolution of interfacial shapes and (2) accurate calculations of the four-way interaction of the coupling between dispersed and continuous phases. These two perspectives significantly influence modeling strategies. For example, the former becomes important in the context of simulating equilibrium shapes of bubbles and droplets (for instance, in designing micro-fluidic devices for effective bio-separations and related material science applications).  The latter becomes important for understanding the fundamental coupling of energies due to the motions of the dispersed phase, for example in bubbly flows.  While important, modeling two phase flows with a resolved description of the interfacial dynamics is challenging because of discontinuities due to surface tension and topological changes of the interface.  

	A standard approach to representing interfacial phenomena is to use jump boundary conditions, which requires interface-fitted meshes \citep{Notz2004}.  Although theoretically promising, this approach is non-trivial and impractical for large topological changes in the interfaces, especially in 3D.  An alternative description of the interface is to smear the sharp discontinuity to a numerically resolvable length scale. There are many flavors of this approach; e.g., the popular level set methods \citep{Osher2006} and front tracking approaches \cite{Unverdi1992}. In these methods a tracking variable (or an indicator field) is used to track the interface (usually on a fixed grid). If one selects a physical property like density as an indicator function, and approximates the forcing due to the motion of the interface as the product of the gradient of the indicator function over the interface and the curvature of the interface, then we get the continuum surface models \citep{Brackbill1992}. Each of these approaches has relative merits and demerits, and we refer the interested reader to the detailed discussion in \citet{Prosperetti2009}. 
	
	Phase field methods are another class of approaches to implicitly track interfaces. They offer some advantages including mass conservation, thermodynamic consistency and a natural way to account for external effects. The underlying idea of phase field models is similar to the level set methods, i.e. to use a smooth scalar field (phase field) to track the interface on a fixed grid. In phase field methods, an advective Cahn-Hilliard equation is used to track the motion of the (smeared or diffuse) interface.  Compared to the level set advection equation, the advective Cahn-Hilliard equation has an added diffusive term that is inherent to the thermodynamic description of the interface.  This diffusion term is analogous to numerical diffusion, which stabilizes the numerical schemes and improves mass conservation\footnote{We also show that this diffusive term helps in proving existence result for the phase field.}.  The other advantage of using Cahn-Hilliard based phase field models is that the surface tension is represented via a free energy-based description with well established footing in thermodynamics \citep{Jacqmin1996,Jacqmin2000} (see \citet{Anderson1998} and references therein for detailed discussion). 
	
	In all of these models a set of momentum equations are coupled with the interface tracking equation.  Typically, a single set of momentum equations are solved for an ``averaged mixture velocity'' with variable density and viscosity (which are inferred from the phase field). Even for incompressible fluids, the solenoidality (divergence-free) of the averaged mixture velocity depends on the type of averaging (mass- or volume- averaging).  Volume averaging usually results in solenoidal mixture velocity, while mass averaging results in a non-solenoidal mixture velocity leading to the so called quasi-incompressible models (see \citet{Guo2017,Shokrpour2018}, and references therein for the development of mass averaged models). The solenoidality of the mixture velocity is a useful feature while constructing numerical schemes\footnote{This property however is only true under strictly isothermal conditions.}. 
	
	In the literature there are many versions of the Cahn-Hilliard Navier-Stokes (CHNS) coupled models \citep{Hohenberg1977,Jacqmin2000,Villanueva2006,Dong2014,Xie2015}. Several of these models do not ensure thermodynamic consistency (i.e. ensure second law is followed)~\citep{Jacqmin2000,Villanueva2006,Xie2015}, nor are they compatible under high density and viscosity ratios of the two phases (model-H by \citet{Hohenberg1977} was not compatible with unequal density ratios). Thus, while generally useful, there are no guarantees that such models work for high density and viscosity contrasts and remain predictive under long simulation horizons. The original CHNS model was developed for modeling binary fluids with equal densities and viscosities, which are the so-called model-H equations by~\citet{Hohenberg1977} and later extended for unequal densities and viscosities.  Recasting these equations in a thermodynamically consistent manner was first attempted by \citet{Gurtin1996}, and the most recent contribution came from \citet{Abels2012}, who derive a thermodynamically consistent model with a solenoidal mixture velocity.  In this paper we choose the model proposed by \citet{Abels2012}, which ensures that the system follows an energy law consistent with the second law of thermodynamics and which does not assume equal densities and viscosities for the two fluids. 
	
	
	We identify three issues that have to be considered when designing numerical approaches for solving the CHNS equations. The first issue is that we would like to have a scheme that is provably energy-stable under a generously large time-step. This endows several promising traits to the numerical approach, including the ability to use larger time steps when marching towards a steady state solution (or towards a long time horizon). The second issue is the necessity of resolving the interfacial length scales for accurate capture of interface dynamics \cite{Xie2015}. This becomes especially challenging during  topological transitions (e.g. filaments, pinch-off points) with intricate changes over small length scales. In order to maintain computational efficiency,
	we require adaptive meshing strategies in order to resolve the interface properly. The need for adaptive meshing is especially important in 3D, where the degrees of freedom rise rapidly with
	the grid spacing. 
	The third issue is the spatial discretization of the CHNS model, considering that the solenoidality of the velocity (i.e. the incompressibility constraint) requires satisfaction of the discrete inf-sup condition (see section 3.3 of \citep{Volker2016} for details). We specifically desire a conforming Galerkin finite element approach for which efficient parallel h-refinement strategies are straightforward and available (for instance \citep{Burstedde2011,Sundar2007,Sundar2008}). These three issues serve as the motivation for the current work. Specifically, our contributions are as follows.
	\begin{enumerate}
		\item \textbf{Energy stability:} We develop a time integration scheme that maintains energy stability for a large range of time steps while also satisfying mass conservation.  
		\item \textbf{Conforming finite elements via stabilization:} We develop a variational multiscale based treatment of the equations that allows us to use conforming Galerkin finite elements. 
		\item \textbf{Parallel adaptive meshing:} We implement the resulting numerical methods
		with a fast, massively parallel, adaptive meshing strategy based on octree meshes for resolving the length scales of the interface dynamics. 
	\end{enumerate}

	\textbf{Energy stability:} \citet{Kim2004a} reported one of the earliest studies on energy stable schemes for a CHNS model with equal densities\footnote{However, this model was not thermodynamically consistent.}. \citet{Feng2006} (and then \citet{Han2015}) followed with a comprehensive analysis of this model reporting energy laws and other bounds on the numerical solutions. \citet{Shen2010a,Shen2010b} extended this analysis for a CHNS model with unequal densities. Subsequently, \citet{Chen2016} reported analysis on the stability of time integration schemes along with a finite difference adaptive strategy for a thermodynamically consistent CHNS system.  In other recent work, a Scalar Auxiliary Variable (SAV) approach has been proposed to construct energy stable schemes~\citep{Shen2018,Zhu2019}. \citet{Guo2017} recently reported a detailed analysis for a mass averaged mixture velocity CHNS system.  
	In \cref{subsec:time_scheme} of the present work we develop an implicit time scheme (similar to Crank-Nicolson) that is energy stable for large time steps, while also discretely mass conserving. 
	The benefit of such a time integration scheme is that it does not require the storage of more than one previous time step, while still providing accuracy and ensuring energy stability.  We prove that the time scheme is (essentially) unconditionally energy stable.  We also subsequently prove in \cref{subsec:semidiscrete_var_prob} the existence of solutions of the time scheme.

	\textbf{Conforming finite elements via stabilization:} In order to easily leverage parallel adaptive meshing tools it is helpful to have conforming finite elements. Most of the studies cited above used mixed element methods (LBB stable pairs of elements) to discretize the momentum equations in the coupled CHNS system. The distinct discrete spaces for pressure and velocities ensure local enforcement of solenoidality and satisfaction of the discrete inf-sup condition (also called the saddle point problem). Alternatively, the saddle point problem can be resolved using stabilization (popularly known as \textit{grad-div} stabilization), which enables using conforming finite elements. 
	Variational multi-scale methods (VMS) provide a principled approach to derive such stabilization. They rely on a projection based decomposition of velocity and pressure fields into coarse and fine scale components following the ansatz of large eddy simulations \citep{Hughes2000}. There are multiple flavors of VMS models based on the choice of decomposition and how the fine scales are approximated. We refer interested readers to a recent and excellent review by \citet{Ahmed2017}. In this work, we develop a formulation based on the Residual Based Variational Multi-scale Method (RBVMS) \cite{Bazilevs2007} with conforming Galerkin finite elements in \cref{subsec:space_scheme}.
	
	\textbf{Parallel adaptive meshing:} While the concept of adaptive space partitions is not novel, developing such methods for large distributed systems presents significant challenges. The challenge is to adaptively resolve the mesh~\citep{Coupez2013,Hachem2013,Hachem2016}, while ensuring appropriate load balancing across the computing cluster.  A promising approach is to use structured meshes (especially based on octrees)~\citep{Khokhlov1998}, where the spatial structure of the elements is leveraged to design efficient data exchange and communication, thus resulting in fast parallel algorithms. In this work we use the octree based library~\dendro~which is well-established for distributed octree-based (structured) meshing algorithms. \dendro~includes novel bottom-up octree construction algorithms that require only local computation followed by a single distributed sort~\citep{Sundar2007,Sundar2008}. \dendro~also implements a 2:1 balancing algorithm\footnote{Enforcing that adjacent octants differ by at most a factor of 2 in size is a condition often enforced during meshing to make subsequent numerical calculations convenient.} that by preemptively communicating information between processes avoids synchronizations and has a provably lower communication cost. \dendro, is freely available  and has been used by several research groups across the world as the meshing scheme for a variety of methods such as finite element computations, fast multipole methods, fast Gauss transforms, and for a range of applications from cardiac biomechanics to direct numerical simulation of blood flow~\citep{Sundar2008}.  We detail adaptive meshing and scalability of our framework in \cref{sec:octree_mesh} and \cref{sec:scaling}, respectively.  
	

\section{Governing equations}
\label{sec:governing_equations}
Consider the bounded domain $\Omega \subset \mathbb{R}^n$, where $n = 2$ or $3$, containing two immiscible fluids, and consider the time interval $[0, T]$. Let $\rho_{+}$ ($\eta_{+}$ ) and $\rho_{-}$ ($\eta_{-}$) denote the specific density (viscosity) of the fluids, respectively. We define a phase field, $\phi$, that tracks the fluids, i.e. takes a value of $+1$ and $-1$ in domains occupied by each of the fluids, respectively. $\phi$ varies continuously across the interface between $+1$ and $-1$.  The non-dimensional density\footnote{Our non-dimensional form uses the specific density/viscosity of fluid 1 as the non-dimensionalising density/viscosity.} is given by $\rho(\phi) = \alpha\phi + \beta$, where $\alpha = \frac{\rho_+ - \;\rho_-}{2\rho_+}$ and $\beta = \frac{\rho_+ + \;\rho_-}{2\rho_+}$. Similarly, non-dimensional viscosity is given by $\eta(\phi) = \gamma\phi + \xi$, where $\gamma = \frac{\eta_+ - \;\eta_-}{2\eta_+}$ and $\xi = \frac{\eta_+ + \;\eta_-}{2\eta_+}$. The governing equations in their non-dimensional form are as follows:
\begin{align}
\begin{split}
	\text{Momentum Eqns:} & \quad \pd{\left(\rho(\phi) v_i\right)}{t} + \pd{\left(\rho(\phi)v_iv_j\right)}{x_j} + \frac{1}{Pe}\pd{\left(J_jv_i\right)}{x_j} +\frac{Cn}{We} \pd{}{x_j}\left({\pd{\phi}{x_i}\pd{\phi}{x_j}}\right) \\
	& \quad \quad \quad + 
	\frac{1}{We}\pd{p}{x_i} - \frac{1}{Re}\pd{}{x_j}\left({\eta(\phi)\pd{v_i}{x_j}}\right) - \frac{\rho(\phi)\hat{{g_i}}}{Fr} = 0, \label{eqn:nav_stokes} 
	\end{split} \\
	\text{Thermo Consistency:} & \quad J_i = \frac{\left(\rho_- - \rho_+\right)}{2} \, m(\phi) 
	\, \pd{\mu}{x_i},\\
	\text{Solenoidality:} & \quad \pd{v_i}{x_i} = 0, \label{eqn:solenoidality}\\
	\text{Continuity:} & \quad \pd{\rho(\phi)}{t} + \pd{\left(\rho(\phi)v_i\right)}{x_i}+
	\frac{1}{Pe} \pd{J_i}{x_i} = 0, \label{eqn:cont}\\
	\text{Chemical Potential:} & \quad \mu = \psi'(\phi) - Cn^2 \pd{}{x_i}\left({\pd{\phi}{x_i}}\right)  ,\label{eqn:mu_eqn} 
	\\ 
	\text{Cahn-Hilliard Eqn:} & \quad \pd{\phi}{t} + \pd{\left(v_i \phi\right)}{x_i} - \frac{1}{PeCn} \pd{}{x_i}\left({\pd{\left(m(\phi)\mu\right)}{x_i}}\right) = 0. 
	\label{eqn:phi_eqn}
\end{align}
In the above equations, $\vec{v}$ is the volume averaged mixture velocity\footnote{We use Einstein notation throughout the manuscript. In this notation $v_i$ represents the $i^{\text{th}}$ component of the vector $\vec{v}$, and any repeated index is implicitly summed over.}, $p$ is the volume averaged pressure, $\phi$ is the phase field (interface tracking variable), and $\mu$ is the chemical potential.  The mobility, $m(\phi)$, is assumed to be a constant with a value of one.  The non-dimensional parameters are as follows: Peclet: $Pe = \frac{u_{r} L_{r}^2}{m\sigma}$; Reynolds: $Re = \frac{u_{r} L_{r}}{\nu_{r}}$; Weber: $We = \frac{\rho_{r}u_{r}^2 L_{r}}{\sigma}$; Cahn: $Cn = \frac{\varepsilon}{L_{r}}$; and Froude: $Fr = \frac{u_{r}^2}{gL_{r}}$. $u_{r}$ and $L_r$ denote the reference velocity and length, respectively. $\hat{\vec{g}}$ is a unit vector defined as $\left(0, -1, 0\right)$ denoting the direction of gravity. $\varepsilon$ is the interfacial thickness.  $\psi(\phi(\vec{x}))$ is a known free-energy function.  For the detailed analysis and importance of the thermodynamic consistency term see~\citet{Zhu2019therm,Zhu2020}.
We use the polynomial form of the free energy density defined as follows:
\begin{align}
		\psi(\phi) = \frac{1}{4}\left( \phi^2 - 1 \right)^2 \qquad \text{and} \qquad
		\psi'(\phi) = \phi^3 - \phi.
\end{align}

The system of equations \cref{eqn:nav_stokes} -- \cref{eqn:phi_eqn} has a dissipative energy law given by: 
\begin{equation}
\d{E_{\text{tot}}}{t} = -\frac{1}{Re}  \int_{\Omega} \frac{\eta(\phi)}{2} \norm{\nabla \vec{v}}_F^2 \mathrm{d}\vec{x} - \frac{Cn}{We} \int_{\Omega}m(\phi) \norm{\nabla \mu}^2  \mathrm{d}\vec{x},
\end{equation}
where the total energy is 
\begin{equation}
E_{\text{tot}}(\vec{v},\phi, t) = \int_{\Omega}\frac{1}{2}\rho \norm{\vec{v}}^2 \mathrm{d}\vec{x} + \frac{1}{CnWe}\int_{\Omega} \left(\psi(\phi) + \frac{Cn^2}{2} \norm{\nabla\phi}^2 + \frac{1}{Fr} \rho(\phi) y \right) \mathrm{d}\vec{x}.
\label{eqn:energy_functional}
\end{equation}
The norms used in the above expression are the Euclidean vector norm and the Frobenius matrix norm:
\begin{equation}
\norm{\vec{v}}^2 := \sum_i \abs{v_i}^2 \qquad \text{and} \qquad 
\norm{\nabla\vec{v}}^2_F := \sum_i \sum_j \abs{\frac{\partial v_i}{\partial x_j}}^2.
\end{equation}

\begin{remark}Realistically, the thickness of the interface (parametrized by the Cahn number) is usually in the nanometer range. Resolving this scale is computationally intractable, as all the other scales in the problem are much larger. Therefore, a standard ansatz that diffuse interface models follow is that the solution tends to the real physics in the limit of $Cn \rightarrow 0$. Usually, one starts from a coarse Cahn number and decreases it until the simulated dynamics is independent of the Cahn number. However, the choice of Cahn number, $Cn$, determines the Peclet number: $Pe$. The Peclet number, $Pe = \frac{u_{r} L_{r}^2}{m\sigma}$, is the ratio of the advection timescale to the time scale of the diffuse interface to relax to an equilibrium $\tanh$ profile (which is a purely computational construct).  \citet{Magaletti2013} reported a careful asymptotic analysis of these timescales and suggests the scaling: $1/Pe = \alpha Cn^2$. We use this scaling with $\alpha = 3$.
\end{remark}

\begin{remark} The volume averaged mixture velocity ($\vec{v}$) is solenoidal (see \cref{eqn:solenoidality}), but momentum ($\rho \vec{v}$) is not (see \cref{eqn:cont}).  Equation \eqref{eqn:cont} is the mass conservation law, and technically the solenoidality of the mixture velocity has nothing to do with mass conservation law, but it is a convenient feature of the model. We make this distinction because in the context of incompressible Newtonian single phase flow, mass conservation reduces to solenoidality of the velocity field (the d'Alembert condition), which is not the case here.
\end{remark}


\section{Numerical method and its properties}
\label{sec:numerical_tecniques}
We seek a Crank-Nicolson type time-stepping scheme for the set of equations given by \cref{eqn:nav_stokes} -- \cref{eqn:phi_eqn}. Such a method will provide accuracy and stability for large time-steps with storage of only one previous time-step. Additionally, using this implicit time scheme allows us to prove existence of solutions in the semi-discrete sense for the Cahn-Hilliard equation.

Let $\delta t$ be a time-step; let any time be given by $t^k := k \delta t$; and 
let us define the following time-averages: 

\begin{equation}
	\widetilde{\vec{v}}^{k} := \frac{\vec{v}^{k} + \vec{v}^{k+1}}{2}, \quad \tp^{k} := \frac{{p}^{k+1}+{p}^{k}}{2},
	\quad \tphi^{k} := \frac{{\phi}^{k} + {\phi}^{k+1}}{2}, \quad \text{and} \quad \tmu^{k} := \frac{{\mu}^{k} + {\mu}^{k+1}}{2},
\end{equation}
and the following potential function evaluations:
\begin{equation}
\label{eqn:psi_ave_def}
	\tpsi := \psi\left( \tphi^k \right) \qquad \text{and} \qquad
	\tpsi' := \psi'\left( \tphi^k \right).
\end{equation}
With these definitions, the time-discretized scheme can be written as follows: 
\begin{align}
\begin{split}
	\text{Momentum Eqns:} & \quad \rho\left(\phi^{k+1}\right)\frac{\left(v^{k+1}_i - v^k_i\right)}{\delta t} + \rho\left(\phi^{k+1}\right) \, \tvj^{k} \, \pd{\tvi^{k}}{x_j} + 
	\frac{1}{Pe} \tJj^{k} \, \pd{\tvi^{k}}{x_j}   \\ 
	& \quad  + \frac{Cn}{We} \pd{}{x_j}\left({\pd{\tphi^{k}}{x_i}\pd{\tphi^{k}}{x_j}}\right) +
	\frac{1}{We}\pd{\tp^{k}}{x_i} - \frac{1}{Re}\pd{}{x_j}\left(\eta(\phi^{k+1}){\pd{\tvi^{k}}{x_j}}\right)\\
	& \quad -\frac{\rho\left(\phi^{k+1}\right)\hat{g_i}}{Fr} = 0,\label{eqn:disc_nav_stokes_semi}
	\end{split} \\ 
	\text{Thermo Consistency:} & \quad \tJi^{k} = \frac{\left(\rho_- - \rho_+\right)}{2}
	 \, \pd{\tmu^{k}}{x_i},\label{eqn:disc_time_ns_semi}\\
	\text{Solenoidality:} & \quad \pd{\tvi^{k}}{x_i} = 0, \label{eqn:disc_cont_semi} \\
	\text{Continuity:} & \quad \frac{\left(\rho\left(\phi^{k+1}\right) - \rho\left(\phi^{k}\right)\right)}{\delta t} + \pd{ \left(\rho\left(\tphi^{k}\right)\tvi^{k}\right) }{x_i} + \frac{1}{Pe} \pd{\tJi^{k}}{x_i}= 0, \label{eqn:disc_consv_semi} \\
	\text{Chemical Potential:} & \quad \tmu^{k} = \tpsi' - Cn^2 \pd{}{x_i}\left({\pd{\tphi^{k}}{x_i}}\right), \label{eqn:disc_mu_eqn_semi}\\
	\text{Cahn-Hilliard Eqn:} & \quad \frac{\left(\phi^{k+1} - \phi^k\right)}{\delta t} + \pd{\left(\tvi^{k} \, \tphi^{k}\right)}{x_i} - \frac{1}{PeCn} \pd{}{x_i}\left({\pd{\tmu^{k}}{x_i}}\right) = 0,
	\label{eqn:disc_time_phi_eqn_semi}
\end{align}
with boundary conditions 
$\pd{\tmu}{x_i} \, \hat{n}_i = 0$, $\pd{\tphi}{x_i} \, \hat{n}_i = 0$, where $\hat{\vec{n}}$ is the outward pointing normal to the boundary $\partial \Omega$, and $\widetilde{\vec{v}}^{k} = \vec{0}$ on $\partial \Omega$.
In the definition below we use the notation that $\vec{v} \in \vec{H}^1_{0} \implies \vec{v}=\vec{0}$
on $\partial \Omega$.
Note that we have chosen to write the momentum equation, \cref{eqn:disc_nav_stokes_semi}, in convective form by combining the conservative form, \cref{eqn:nav_stokes}, and the continuity equation,
\cref{eqn:cont}.

\subsection{Fully discrete scheme}
\label{subsection:fully_discrete_scheme}
The fully discrete method proposed in this work is a continuous Galerkin (cG(1)) spatial discretization of \cref{eqn:disc_nav_stokes_semi} -- \cref{eqn:disc_time_phi_eqn_semi}. The fully discrete method
is based on the variational form of \cref{eqn:disc_nav_stokes_semi} -- \cref{eqn:disc_time_phi_eqn_semi},
which we define below.
%
\begin{definition}
	Let $(\cdot,\cdot)$ be the standard $L^2$ inner product. The time-discretized variational
	problem can stated as follows: find $\vec{v}^{k+1}(\vec{x}) \in \vec{H}_0^1(\Omega)$, $p^{k+1}(\vec{x})$, $\phi^{k+1}(\vec{x})$, $\mu^{k+1}(\vec{x})$ $\in {H}^1(\Omega)$ such that
	\begin{align}
	\begin{split}
	\text{Momentum Eqns:} & \quad \left(w_i, \, \rho\left(\phi^{k+1}\right)\frac{\left(v^{k+1}_i - v^k_i\right)}{\delta t}\right) + \left(w_i, \, \rho\left(\phi^{k+1} \right) \, \tvj^{k} \, \pd{\tvi^{k}}{x_j}\right) \\ & \quad +
	\frac{1}{Pe}\left(w_i, \, \tJj^{k} \, \pd{\tvi^{k}}{x_j}\right) -
	\frac{Cn}{We} \left(\pd{w_i}{x_j}, \, {\pd{\tphi^{k}}{x_i}\pd{\tphi^{k}}{x_j}}\right)  -
	\frac{1}{We}\left(\pd{w_i}{x_i}, \, {\tp^{k}} \right) \\ & \quad + \frac{1}{Re}\left(\pd{w_i}{x_j}, \,\eta\left(\phi^{k+1}\right){\pd{\tvi^{k}}{x_j}} \right) -
	\left(w_i,\frac{\rho\left(\phi^{k+1}\right)\hat{g_i}}{Fr}\right) = 0, 
	\label{eqn:nav_stokes_var_semi_disc}
	\end{split} \\ 
	\text{Thermo Consistency:} & \quad \tJi^{k} = \frac{\left(\rho_- - \rho_+ \right)}{2} \, \pd{\tmu^{k}}{x_i}, \label{eqn:thermo_consistency_semi_disc} \\
	\text{Solenoidality:} & \quad \left(q, \, \pd{\tvi^{k}}{x_i}\right) = 0, \label{eqn:cont_var_semi_disc} \\
	\text{Chemical Potential:} & \quad 
	-\left(q,\tmu^{k}\right) + \left(q, \tpsi' \right) + Cn^2 \left(\pd{q}{x_i}, \, {\pd{\tphi^{k}}{x_i}}\right)   = 0, \label{eqn:mu_eqn_var_semi_disc}\\
	\text{Cahn-Hilliard Eqn:} & \quad \left(q, \frac{\phi^{k+1} - \phi^k}{\delta t} \right) -
	 \left(\pd{q}{x_i}, \, \tvi^{k} \tphi^{k} \right) + \frac{1}{PeCn} \left(\pd{q}{x_i}, \, {\pd{\tmu^{k}}{x_i}} \right) = 0,
	\label{eqn:phi_eqn_var_semi_disc}
	\end{align}
	$\forall \vec{w} \in \vec{H}^1_0(\Omega)$, $\forall q \in H^1(\Omega)$, given $\vec{v}^{k} \in \vec{H}_0^1(\Omega)$, and $\phi^{k},\mu^{k} \in H^1(\Omega)$.
	\label{def:variational_form_sem_disc}
\end{definition}

We solve the cG(1) approximated version of variational problem \cref{eqn:nav_stokes_var_semi_disc} -- \cref{eqn:phi_eqn_var_semi_disc} using a block iteration technique, i.e., we treat the Navier-Stokes equations and the Cahn-Hilliard equations as two distinct sub-problems. Thus, two non-linear solvers are stacked together inside the time loop. These non-linear solvers are solved self-consistently until the change (error between current non-linear solve and previous non-linear solve) in the respective solutions is less than a set tolerance within every time step.  See \cref{fig:flowchart_block} for a  flowchart of the approach. We emphasize that a block iterative approach allows us to make the coupling variables from one equation constant in the other during each respective non-linear solve.  For example, for the momentum equation, all the terms depending on $\phi$ (which is solved in the Cahn-Hilliard sub problem) are known. Similarly the mixture velocity used in the Cahn-Hilliard equation solve is known.  

\begin{remark} While $\phi \in [-1, 1]$ in the original equations, there is a possibility of excursions of $\phi$ outside these bounds due to numerical errors. While this does not adversely affect the $\phi$ evolution (i.e. the CH equation), it may cause non-positivity of the mixture density $\rho(\phi)$ and the mixture viscosity $\eta(\phi)$, which directly depend on $\phi$. This causes drift of the bulk phase density from the true specific density of that phase, with some locations exhibiting negative density (or viscosity). This effect is especially possible for very high density ratio between the two fluids, like in the case of a water-air system ($1:10^{-3}$). A simple fix for this issue is by saturation scaling, i.e., pulling back the value of $\phi$ only for the calculation of density and viscosity. We therefore define $\phi^*$ that is only used for the calculation of mixture density and viscosity, where $\phi^*$ is given by: 

\begin{equation}
\phi^* := 
\begin{cases}
\phi, &\quad \text{if} \;\; \abs{\phi} \leq 1, \\
\mathrm{sign}(\phi), &\quad \text{otherwise.} 
\end{cases}
\label{eqn:phi_for_density}
\end{equation} 

\end{remark}

\begin{figure}
\centering			
			\begin{tikzpicture}[%
			scale=1,transform shape,
			>=latex,              
			start chain=going right,    
			node distance=25mm and 25mm, 
			every join/.style={norm},   
			]
			\scriptsize
			
			\tikzset{
				start/.style={rectangle, draw, text width=3cm, text centered, rounded corners, minimum height=2ex},
				base/.style={draw, on chain, on grid, align=center, minimum height=2ex},
				proc/.style={base, rectangle, inner sep=4pt}, 
				inout/.style={base,trapezium,trapezium left angle=70,trapezium right 
					angle=-70}, 
				decision/.style={base, diamond, aspect = 4.0}, 
				norm/.style={->, draw},
			}
			
			\node [start,text width=10em] (b0) {Known fields at some timestep $k$:  $\vec{v}^{k}$, $p^{k}$, $\phi^{k}$, $\mu^{k}$};
			\node [proc,below =of b0,join,text width=30em] (b1) {First iteration Navier-Stokes: $\text{block}_{\text{iter}} = 0$ \\ Solve Navier-Stokes and update the fields: \\ $\vec{v}^{k+1(0)} \leftarrow \vec{v}^k$, \quad $p^{k+1(0)} \leftarrow p^k$};
			\node [proc,below =of b1,yshift=0pt,xshift=0pt,text width=30em] (b2) {First iteration Cahn-Hilliard: $\text{block}_{\text{iter}} = 0$ \\ Solve Cahn-Hilliard and update the fields: \\
			 $\phi^{k+1(0)} \leftarrow \phi^k$, \quad $\mu^{k+1(0)} \leftarrow \mu^k$};
			\node [proc,below =of b2,yshift=0pt,xshift=0pt,text width=10em] (b3) {$\text{block}_{\text{iter}} = \text{block}_{\text{iter}} + 1$};
			\node [proc,below =of b3,yshift=0pt,xshift=0pt,text width=30em] (b4) {$\ell^\text{th}$ block iteration Navier-Stokes: $\ell=\text{block}_{\text{iter}}$ \\ Solve Navier-Stokes and update the fields$:\vec{v}^{k+1(\ell)} \leftarrow \vec{v}^{k+1(\ell-1)}$, \quad $p^{k+1(\ell)} \leftarrow p^{k+1(\ell-1)}$};
			\node [proc,below =of b4,yshift=0pt,xshift=0pt,text width=30em] (b5) {$\ell^\text{th}$ block Cahn-Hilliard: $\ell = \text{block}_{\text{iter}}$ \\ Solve Cahn-Hilliard and update the fields$:\phi^{k+1(\ell)} \leftarrow \phi^{k+1(\ell-1)}$, \quad $\mu^{k+1 (\ell)} \leftarrow \mu^{k+1 (\ell-1)}$};
			\node[decision,below =of b5,yshift=0pt,xshift=0pt,text width=20em](b6) {if $\text{block}_{\text{iter}}>1$ and $\max\norm{\vec{u}^{k+1(\ell)}- \vec{u}^{k+1(\ell-1)}} > \text{block}_\text{tol}$, where $\vec{u}^{k+1}$ is a vector containing $\vec{v}, p, \phi, \mu$};
			\node [start,below =of b6,yshift=0pt,xshift=0pt,text width=16em] (b7) {Solution at current timestep $k+1$:  $\vec{v}^{k+1} = \vec{v}^{k+1(\ell)}$, $p^{k+1} = p^{k+1(\ell)}$, $\phi^{k+1} = \phi^{k+1(\ell)}$, $\mu^{k+1} = \mu^{k+1(\ell)}$};

			\draw [->] (b0)--(b1);		
			\draw [->] (b1)--(b2);
			\draw [->] (b2)--(b3);
			\draw [->] (b3)--(b4);
			\draw [->] (b4)--(b5);
			\draw [->] (b5) -- (b6);
			\draw [->] (b6.west) -| ++(-1.5,0) node[anchor=east] {YES} |- (b3.west);
			\draw [->] (b6) -- node[anchor=east] {NO}(b7);
			
			\end{tikzpicture}
\caption{Flowchart for the block iteration technique as described
in \cref{subsection:fully_discrete_scheme}
.}
\label{fig:flowchart_block}		
\end{figure}
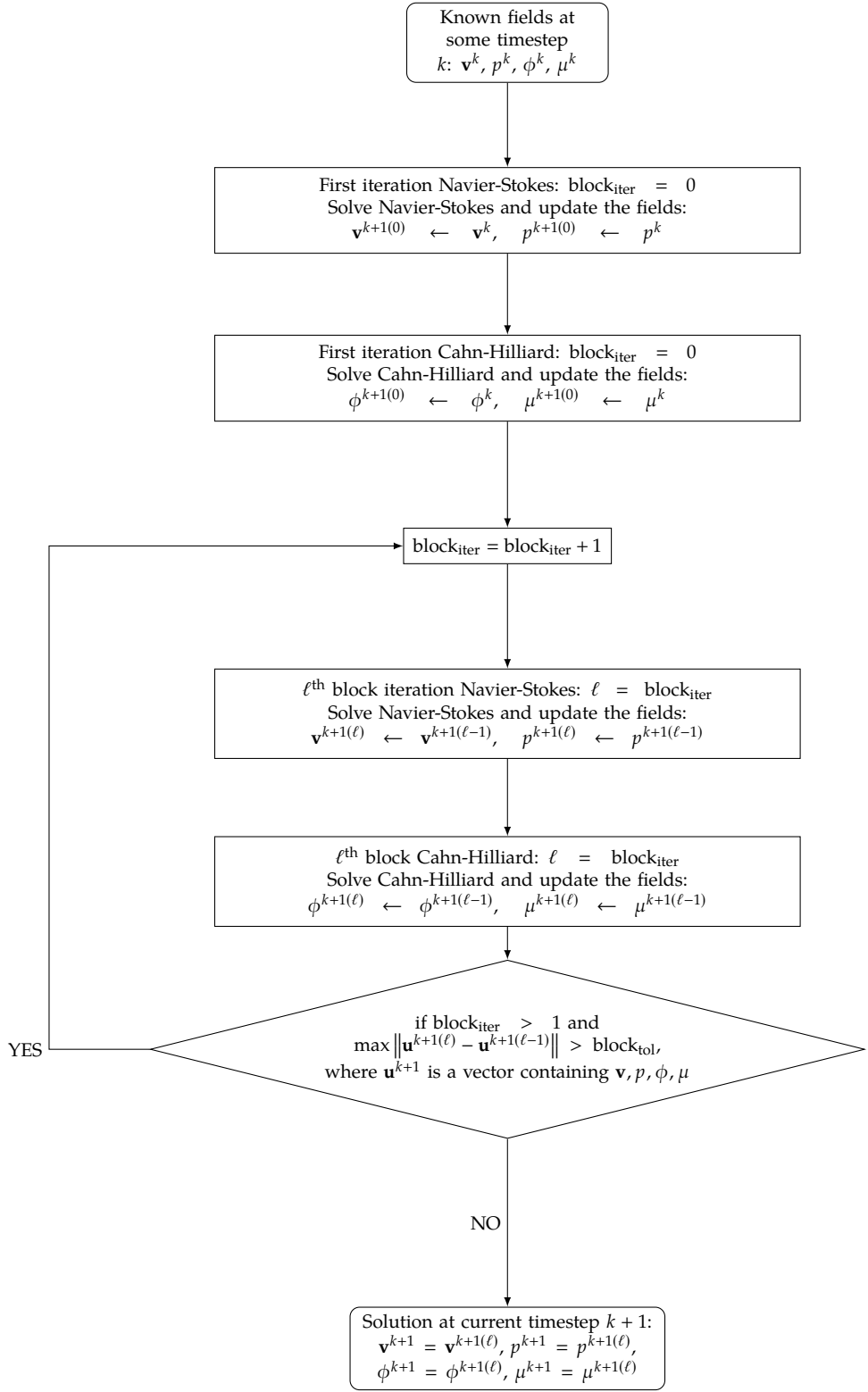

\begin{remark}
It is important to note here that we are using the block iteration technique.  Therefore, $\phi$ and $\mu$ are known when solving momentum equations and $v_i$ is known when solving the advective Cahn-Hilliard equation. The theorems and proofs we present in the subsequent
subsections all assume that we are using the block iterative technique. However, it is not difficult to extend these theorems and proofs for the case of a fully coupled implementation; the theorems of unconditional stability and existence presented here hold even in the fully coupled case. 
\end{remark}

\subsection{Energy stability of the time-stepping scheme}
\label{subsec:time_scheme}
In this subsection we give a rigorous proof of the energy-stability of the
time-stepping scheme as described above. We begin with a result
about mass conservation.

\begin{proposition}[Mass conservation]
	The scheme given by \cref{eqn:nav_stokes_var_semi_disc} -- \cref{eqn:phi_eqn_var_semi_disc} with the following boundary conditions: 
	
	\begin{equation}
	\pd{\tmu}{x_i} \hat{n}_i\Bigl\rvert_{\partial \Omega} = 0, \quad \pd{\tphi}{x_i} \hat{n}_i\Bigl\rvert_{\partial \Omega} = 0, \quad \widetilde{\vec{v}}^{k}\Bigl\rvert_{\partial \Omega} = \vec{0},
	\end{equation}
	where $\hat{\vec{n}}$ is the outward pointing normal to the boundary $\partial \Omega$,
	is globally mass conservative: 
	
	\begin{equation}
	\int_{\Omega} \phi^{k+1} \, \mathrm{d}\vec{x}  = \int_{\Omega} \phi^{k} \, \mathrm{d}\vec{x}.
	\end{equation}
	\label{prop:mass_conservation}
\end{proposition}

 This is a well known result shown previously in literature \citep{Guo2017,Feng2006}.  The proof involves selecting the test function as 1.0 $\in H^1(\Omega)$ in the variational form of  \cref{eqn:nav_stokes_var_semi_disc} -- \cref{eqn:phi_eqn_var_semi_disc} (see \cref{def:variational_form_sem_disc}) and proving the integral of the time derivative to be zero. Since this is a well-known result,
we do not provide the proof here. We verify the claim numerically in the results section for a wide variety of canonical problems.

\begin{lemma}[Weak equivalence of forcing]
	The forcing term due to Cahn-Hilliard in the momentum equation, 
\cref{eqn:nav_stokes_var_semi_disc}, with the test function $w_i = \delta t \, \tvi^k$, can be written equivalently as
	\begin{equation}
	\frac{Cn}{We}\left( \pd{}{x_j}\left({\pd{\tphi^k}{x_i}\pd{\tphi^{k}}{x_j}}\right), \delta t \, \tvi^k\right) = \frac{\delta t}{WeCn}\left({\tphi}^k\pd{\tmu^{k}}{x_i},\tvi^k\right),
	\label{eqn:lemma_weak_equiv_forcing}
	\end{equation}
	$\forall \;\; \tphi^k$, $\tmu^{k} \in H^1(\Omega)$, and $\forall \;\; \widetilde{\vec{v}}^k \in  \vec{H}_{0}^1(\Omega)$, where $\vec{v}^k, \vec{v}^{k+1}, p^k, p^{k+1}, \phi^k, \phi^{k+1}, \mu^{k},\mu^{k+1}, $ satisfy \cref{eqn:nav_stokes_var_semi_disc} -- \cref{eqn:phi_eqn_var_semi_disc}.
	\label{lem:forcing_equivalent}
\end{lemma}
\begin{proof}
	To prove this we will manipulate $\pd{}{x_j}\left({\pd{\tphi^{k}}{x_i}\pd{\tphi^{k}}{x_j}}\right)$ using vector calculus. Using the product rule we have:
	\begin{equation}
	\begin{split}
	\pd{}{x_j}\left({\pd{\tphi^{k}}{x_i}\pd{\tphi^{k}}{x_j}}\right) &= \pd{\tphi^{k}}{x_i}\left(\pd{}{x_j}\left({\pd{\tphi^{k}}{x_j}}\right)\right) + \pd{\tphi^{k}}{x_j}\left(\pd{}{x_j}\left({\pd{\tphi^{k}}{x_i}}\right)\right) \\
	& =\pd{\tphi^{k}}{x_i}\left(\pd{}{x_j}\left({\pd{\tphi^{k}}{x_j}}\right)\right) +\frac{1}{2}\pd{}{x_i}\left(  \pd{\tphi^{k}}{x_j} \pd{\tphi^{k}}{x_j} \right),
	\label{eqn:forcing_expan}
	\end{split}
	\end{equation}
	where the second equality follows from proposition \cref{prop:forcing_aside} in \ref{sec:appendix1}.
	We manipulate this expression to write it in terms of $\mu$:
	\begin{equation}
	\begin{split}
	\pd{}{x_j}\left({\pd{\tphi^{k}}{x_i}\pd{\tphi^{k}}{x_j}}\right) &= \pd{\tphi^{k}}{x_i}\left(\pd{}{x_j}\left({\pd{\tphi^{k}}{x_j}}\right)\right) +\frac{1}{2}\pd{}{x_i}\left(  \pd{\tphi^{k}}{x_j} \pd{\tphi^{k}}{x_j} \right)  + \frac{1}{Cn^2} \, \tpsi' \, \pd{\tphi^{k}}{x_i} - \frac{1}{Cn^2} \, \tpsi' \, \pd{\tphi^{k}}{x_i},\\ 
	 &= \pd{\tphi^{k}}{x_i}\left(\pd{}{x_j}\left({\pd{\tphi^{k}}{x_j}}\right) -  \frac{1}{Cn^2} \, \tpsi' \,  \right) +\frac{1}{2}\pd{}{x_i}\left(  \pd{\tphi^{k}}{x_j} \pd{\tphi^{k}}{x_j} \right)  + \frac{1}{Cn^2}  \, \tpsi' \,  \pd{\tphi^{k}}{x_i}.
	\label{eqn:forcing_expan_f}
	\end{split}
	\end{equation}
	 The expression in the parenthesis in the first term can be replaced
	using the chemical potential equation \eqref{eqn:disc_mu_eqn_semi}, which leads to:
	\begin{equation}
	\pd{}{x_j}\left({\pd{\tphi^{k}}{x_i}\pd{\tphi^{k}}{x_j}}\right) = -\pd{\tphi^{k}}{x_i}\frac{\tmu^k}{Cn^2} +\frac{1}{2}\pd{}{x_i}\left(  \pd{\tphi^{k}}{x_j} \pd{\tphi^{k}}{x_j} \right)  + \frac{1}{Cn^2} \, \tpsi' \, \pd{\tphi^{k}}{x_i}. 
	\label{eqn:forcing_expan_mu}
	\end{equation}
	Using the product and chain rules we obtain:
	\begin{equation}
	\begin{split}
	\pd{}{x_j}\left({\pd{\tphi^{k}}{x_i}\pd{\tphi^{k}}{x_j}}\right) &= 
	  \frac{\tphi^k}{Cn^2} \pd{\tmu^{k}}{x_i} - \frac{1}{Cn^2}\pd{\left(\tmu^{k}\tphi^{k}\right)}{x_i} +\frac{1}{2}\pd{}{x_i}\left(  \pd{\tphi^{k}}{x_j} \pd{\tphi^{k}}{x_j} \right)  + \frac{1}{Cn^2}\pd{\widetilde{\psi}}{x_i} \\
	&= \frac{\tphi^k}{Cn^2} \pd{\tmu^{k}}{x_i} - \pd{}{x_i}\left(\frac{\tmu^{k} \, \tphi^{k}}{Cn^2}  -
	 \frac{1}{2} \pd{\tphi^{k}}{x_j} \pd{\tphi^{k}}{x_j} - \frac{\tpsi}{Cn^2}\right).
	\label{eqn:forcing_expan_mu_int_parts}
	\end{split}
	\end{equation}
Next we substitute this simplification into the appropriate inner product term in \cref{eqn:phi_eqn_var_semi_disc}:
	\begin{equation}
	\begin{split}
	&\left( \pd{}{x_j}\left({\pd{\tphi^{k}}{x_i}\pd{\tphi^{k}}{x_j}}\right), \delta t \, \tvi^k\right) = 
	\delta t \left(\frac{\tphi^k}{Cn^2} \pd{\tmu^{k}}{x_i} - \pd{}{x_i}\left( \frac{\tmu^{k} \, \tphi^{k}}{Cn^2}  -
	 \frac{1}{2} \pd{\tphi^{k}}{x_j} \pd{\tphi^{k}}{x_j} - \frac{\tpsi}{Cn^2} \right) , \tvi^k\right) \\
	& \qquad \qquad  = 
	\delta t \left(\frac{\tphi^k}{Cn^2} \pd{\tmu^{k}}{x_i},\tvi^k \right) - \delta t \left(\pd{}{x_i}\left( \frac{\tmu^{k} \, \tphi^{k}}{Cn^2}  -
	 \frac{1}{2} \pd{\tphi^{k}}{x_j} \pd{\tphi^{k}}{x_j} - \frac{\tpsi}{Cn^2} \right) , \tvi^k\right) \\
	& \qquad \qquad = 
	\frac{\delta t}{Cn^2} \left({\tphi^k} \pd{\tmu^{k}}{x_i}, \tvi^k\right) + \delta t \left(
	\frac{\tmu^{k} \, \tphi^{k}}{Cn^2}  -
	 \frac{1}{2} \pd{\tphi^{k}}{x_j} \pd{\tphi^{k}}{x_j} - \frac{\tpsi}{Cn^2} , \, \pd{\tvi^k}{x_i}\right).
	\label{eqn:equilvalent_forcing} 
	\end{split}   	       	
	\end{equation}
	The last term vanishes due to the solenoidality of the velocity field, \cref{eqn:disc_time_phi_eqn_semi}; and therefore,
	after multiplying by $Cn/We$, we achieve the desired result.
\end{proof}

\begin{corollary}[Strong equivalence of forcing]
	If we have the following equivalence in the weak sense:
	\begin{equation}
	\frac{Cn}{We}\left( \pd{}{x_j}\left({\pd{\tphi^k}{x_i}\pd{\tphi^{k}}{x_j}}\right), \delta t  \, \tvi^k\right) = \frac{\delta t}{WeCn}\left(\tphi^k \pd{\tmu^{k}}{x_i},\tvi^k\right),
	\end{equation}
	$\forall \;\; \tphi^k$, $\tmu^{k} \in  H^1(\Omega)$, and $\forall \;\; \widetilde{\vec{v}}^k \in  \vec{H}_{0}^1(\Omega)$, where $\widetilde{\vec{v}}^k, \widetilde{\vec{v}}^{k+1}, p^k, p^{k+1}, \phi^k, \phi^{k+1}, \mu^{k},\mu^{k+1}$ satisfy \cref{eqn:nav_stokes_var_semi_disc} -- \cref{eqn:phi_eqn_var_semi_disc}, then 
	the following equivalence also holds in the strong sense:
	\begin{equation}
	\frac{Cn}{We}\pd{}{x_j}\left({\pd{\tphi^k}{x_i}\pd{\tphi^{k}}{x_j}}\right) = \frac{1}{WeCn}\tphi^k\pd{\tmu^{k}}{x_i},
	\end{equation} 
	if $\tphi^k$, $\tmu^{k} \in H^1(\Omega) \bigcap C^{\infty}_{c}(\Omega)$, and $\widetilde{\vec{v}}^k \in \vec{H}_{0}^1(\Omega)\bigcap \vec{C}^{\infty}_{c}(\Omega)$.
\end{corollary}

There are numerous papers in the literature which uses the $\tphi^k\pd{\tmu^{k}}{x_i}$ as the surface tension forcing \citep{Chen2016,Han2015,Jacqmin2000,Yue2010}.  The above corollary shows that the form of stress used in the CHNS model presented in this system is equivalent to the more popular $\tphi^k\pd{\tmu^{k}}{x_i}$ in the weak sense.

\begin{remark}
	The advection term in \cref{eqn:nav_stokes} can be defined in the skew-symmetric form as (see lemma 6.10 of section 6.1.2 of \citep{Volker2016} for details):
	\begin{equation}
	B(v_i, v_j) \coloneqq v_j \pd{v_i}{x_j} + \frac{1}{2}v_i \pd{v_j}{x_j}.
	\end{equation} 
	Using the solenoidality of the mixture velocity \cref{eqn:cont} we have that
	\begin{align}
	B_1(v_i, v_j) = \rho v_j \pd{v_i}{x_j} \qquad \text{and} \qquad B_2(v_i, v_j) = J_j\pd{v_i}{x_j}.
	\end{align}
	The skew symmetric form induces a trilinear form when weakened; for three general vectors $u_i$, $v_i$, $w_i$ $\in$ $\vec{H}_{0}^1(\Omega)$ we have\footnote{Here the subscript 0 for the Sobolev space $\vec{H}_{0}^1(\Omega)$ represents zero velocities on the boundary in the trace sense.}:
	\begin{align}
	b_1(u_i, v_j, w_i) &= \Bigl(B_1(v_i, v_j),w_i \Bigr) = \frac{1}{2}  \left(v_j \pd{v_i}{x_j}, w_i\right) - \frac{1}{2} \left(v_j \pd{w_i}{x_j}, v_i\right) , \\
	b_2(u_i, J_j, w_i) &= \Bigl(B_1(v_i, J_j),w_i\Bigr) = \frac{1}{2}  \left(J_j \pd{v_i}{x_j}, w_i\right) - \frac{1}{2}\left(J_j \pd{w_i}{x_j}, v_i\right) .
	\end{align}
Then for our case in the momentum equations, consider the situation where we have $\vec{J}, \vec{v} \in \vec{H}_{0}^1(\Omega)$, and we are working towards energy estimates, which entails taking an inner product of momentum equation with $\vec{v}$ to get an energy functional (to obtain the second order moment). In which case we have for both the non-linear terms in momentum equations: 
\begin{align}
	b_1(v_i, v_j, v_i) &= \Bigl(B_1(v_i, v_j),v_i\Bigr) = \frac{1}{2}  \left(\rho v_j \pd{v_i}{x_j}, v_i\right) - \frac{1}{2}\left(\rho v_j \pd{v_i}{x_j}, v_i\right)  = 0,\label{eqn:trilinear_zero_vel}\\
	b_2(v_i, J_j, v_i) &= \Bigl(B_2(v_i, J_j),v_i\Bigr) = \frac{1}{2}  \left(J_j \pd{v_i}{x_j}, v_i\right) - \frac{1}{2}\left(J_j \pd{v_i}{x_j}, v_i\right) = 0\label{eqn:trilinear_zero_massflux}.
\end{align}
This makes physical sense from the point-of-view of energy balance, since the aforementioned non-linear terms do not act as sinks or sources; instead, they provide the mechanism for redistribution of energy to various length scales. 
\label{rem:skew-symmetry}
\end{remark}

\begin{lemma} 
The variational advection term from the Cahn-Hilliard contribution in the momentum equation, 
\cref{eqn:nav_stokes_var_semi_disc}, can be written as follows:  
	\begin{equation}
	\begin{split}
	\frac{\delta t}{WeCn}\left(\tphi^k\tvi^k,\pd{\tmu^{k}}{x_i}\right) & = 
	\frac{1}{2}\int_{\Omega}\left( \rho\left(\phi^{k+1}\right)\norm{{\vec{v}}^{k+1}}^2 - \rho\left(\phi^{k}\right)\norm{\vec{v}^{k}}^2 \right)  \, d\vec{x}  
	- \frac{ \delta t}{Re} \norm{\sqrt{\eta\left(\phi^{k+1}\right)}\nabla \widetilde{\vec{v}}^k}_{L^2}^2 \\ 
	& \quad - \frac{1}{Fr}\left(y,  \,  \rho\left(\phi^{k+1}\right) - \rho\left(\phi^{k}\right)  \right),
	\end{split}
	\label{eqn:cahn_hilliard_term_lemma}
	\end{equation}
	$\forall \;\; \tphi^k$, $\phi^{k+1}$, $\tmu^{k} \in  H^1(\Omega)$, and $\forall \;\; {\vec{v}}^k, {\vec{v}}^{k+1} \in \vec{H}_{0}^1(\Omega)$, where  ${\vec{v}}^k, {\vec{v}}^{k+1}, p^k, p^{k+1}, \phi^k, \phi^{k+1}, \mu^{k},\mu^{k+1}$ satisfy \cref{eqn:nav_stokes_var_semi_disc} -- \cref{eqn:phi_eqn_var_semi_disc},
	and
	\begin{equation}
	\begin{split}
	    \norm{\vec{v}}^2 := \sum_i \abs{v_i}^2, \\
	    \qquad \norm{\sqrt{\eta\left(\phi^{k+1}\right)}\nabla \widetilde{\vec{v}}^k}_{L^2}^2 &:= 
	\int_{\Omega} \sqrt{\eta\left(\phi^{k+1}\right)} \sum_i \sum_j \abs{\frac{\partial \tvi^k}{\partial x_j}}^2 \, d\vec{x} = 
	\int_{\Omega} \sqrt{\eta\left(\phi^{k+1}\right)}\norm{\nabla\vec{v}}^2_F \, d\vec{x}.
	\end{split}
	\end{equation}
	\label{lem:advection_estimate}
\end{lemma}
\begin{proof}
	We start by taking the $L^2$ inner product of momentum equation \eqref{eqn:nav_stokes_var_semi_disc} with $\delta t \, \tvi^k$:
	\begin{equation}
	\begin{split}
	\left(\rho\left(\phi^{k+1}\right)\frac{\left(v^{k+1}_i - v^k_i\right)}{\delta t}, \delta t \, \tvi^k\right) &+ \left(\rho\left(\phi^{k+1}\right) \, \tvj^{k} \, \pd{\tvi^{k}}{x_j}, \delta t \, \tvi^k\right) \\ &+ 
	\frac{1}{Pe}\left(\tJj^{k}\pd{\tvi^{k}}{x_j}, \delta t \, \tvi^k\right)
	+ \frac{Cn}{We}\left( \pd{}{x_j}\left({\pd{\tphi^{k}}{x_i} \, \pd{\tphi^{k}}{x_j}}\right), \delta t  \,\tvi^k\right) \\ &+ 
	\frac{1}{We}\left(\pd{\tp^{k}}{x_i}, \delta t \, \tvi^k\right) - \frac{1}{Re}\left(\pd{}{x_j}\left(\eta\left(\phi^{k+1}\right){\pd{\tvi^{k}}{x_j}}\right), \delta t \, \tvi^k\right) \\ &-
	\frac{1}{Fr}\left(\rho\left(\phi^{k+1}\right)\hat{g_i}, \delta t \, \tvi^k\right) = 0. 
	\label{eqn:disc_nav_stokes_inner_prod}
	\end{split}
	\end{equation}
	Notice that the second and third terms are in a trilinear form so from \cref{eqn:trilinear_zero_vel} and \cref{eqn:trilinear_zero_massflux} they go to zero and we have:
	\begin{align}
	\begin{split}
	& \quad \left(\rho\left(\phi^{k+1}\right)\left(v^{k+1}_i - v^k_i\right), \tvi^k\right) + 
	\frac{Cn}{We}\left( \pd{}{x_j}\left({\pd{\tphi^{k}}{x_i}\pd{\tphi^{k}}{x_j}}\right), \delta t \, \tvi^k\right) \\ & \qquad \qquad \qquad + 
	\frac{1}{We}\left(\pd{\tp^{k}}{x_i}, \delta t \, \tvi^k\right)  
	- \frac{1}{Re}\left(\pd{}{x_k}\left(\eta\left(\phi^{k+1}\right){\pd{\tvi^{k}}{x_k}}\right), \delta t \,\tvi^k\right) \\ & \qquad \qquad \qquad -
	 \frac{1}{Fr}\left(\hat{g_i}, \delta t \, \rho\left(\phi^{k+1}\right)\tvi^k\right)= 0,
	\end{split}\\  
	\begin{split}     	
	&\implies \frac{1}{2}\int_{\Omega}\left( \rho\left(\phi^{k+1}\right)\norm{{\vec{v}}^{k+1}}^2 -  \rho\left(\phi^{k}\right)\norm{\vec{v}^{k}}^2 \right)  \, d\vec{x}\\ & \qquad \qquad \qquad
	+ \frac{Cn}{We}\left( \pd{}{x_j}\left({\pd{\tphi^{k}}{x_i}\pd{\tphi^{k}}{x_j}}\right), \delta t \,\tvi^k\right) \\ & \qquad \qquad \qquad + 
	\frac{1}{We}\left(\pd{\tp^{k}}{x_i}, \delta t \, \tvi^k\right) 
	- \frac{1}{Re} \left(\pd{}{x_j}\left(\eta\left(\phi^{k+1}\right){\pd{\tvi^{k}}{x_j}}\right), \delta t \, \tvi^k\right) \\ & \qquad \qquad \qquad -
	 \frac{1}{Fr}\left(\hat{g_i}, \delta t \, \rho\left(\phi^{k+1}\right)\tvi^k\right)= 0,
	\end{split}
	\label{eqn:disc_nav_stokes_inner_prod_no_advec}
	\end{align}
	where we made use of the fact that $\tvi^k = (v_i^{k+1}+v_i^k)/2$ and~\cref{lem:evol_estimate}~from \ref{sec:appendix1}.
	We can now use solenoidality of the velocity field to get rid of the pressure term.  We can do this by weakening the pressure term:
	\begin{align}
	\begin{split}	
	& \frac{1}{2}\int_{\Omega}\left( \rho\left(\phi^{k+1}\right)\norm{{\vec{v}}^{k+1}}^2 - \rho\left(\phi^{k}\right)\norm{\vec{v}^{k}}^2 \right)  \, d\vec{x}
	+ \frac{Cn}{We}\left( \pd{}{x_j}\left({\pd{\tphi^{k}}{x_i}\pd{\tphi^{k}}{x_j}}\right), \delta t \,\tvi^k\right) \\ & \qquad  
	-\frac{\delta t}{We}\left(\tp^{k}, \pd{\tvi^k}{x_i}\right)  
	-\frac{1}{Re} \left(\pd{}{x_j}\left(\eta\left(\phi^{k+1}\right){\pd{\tvi^{k}}{x_j}}\right), \delta t \, \tvi^k\right) \\
	& \qquad   -\frac{1}{Fr}\left(\hat{g_i}, \delta t \, \rho\left(\phi^{k+1}\right)\tvi^k\right) = 0, \\
	\end{split} \\
	\displaybreak[0]
	\begin{split}
	\implies \quad  & \frac{1}{2}\int_{\Omega}\left( \rho\left(\phi^{k+1}\right)\norm{{\vec{v}}^{k+1}}^2 - \rho\left(\phi^{k}\right)\norm{\vec{v}^{k}}^2 \right)  \, d\vec{x}
	+\frac{Cn}{We}\left( \pd{}{x_j}\left({\pd{\tphi^{k}}{x_i}\pd{\tphi^{k}}{x_j}}\right), \delta t \, \tvi^k\right)  \\
	&  \qquad   -\frac{1}{Re} \left(\pd{}{x_j}\left(\eta(\phi^{k+1}){\pd{\tvi^{k}}{x_j}}\right), \delta t \, \tvi^k\right) 
	 -\frac{1}{Fr}\left(\hat{g_i}, \delta t \, \rho\left(\phi^{k+1}\right)\tvi^k\right) = 0,\label{eqn:disc_nav_stokes_inner_prod_no_pres} 
	\end{split} \\
	\begin{split}
	\implies \quad  & \frac{1}{2}\int_{\Omega} \left(\rho\left(\phi^{k+1}\right)\norm{{\vec{v}}^{k+1}}^2 - \rho\left(\phi^{k}\right)\norm{\vec{v}^{k}}^2 \right)  \, d\vec{x}
	+\frac{Cn}{We}\left( \pd{}{x_j}\left({\pd{\tphi^{k}}{x_i}\pd{\tphi^{k}}{x_j}}\right), \delta t \,\tvi^k\right) \\
	&  \qquad    + \frac{\delta t}{Re} \left(\left( \sqrt{\eta(\phi^{k+1})} {\pd{\tvi^{k}}{x_j}}\right), \left(\sqrt{\eta(\phi^{k+1})}{\pd{\tvi^{k}}{x_j}}\right)\right) 
	 -\frac{1}{Fr}\left(\hat{g_i}, \delta t \, \rho\left(\phi^{k+1}\right)\tvi^k\right) = 0,\label{eqn:disc_nav_stokes_inner_prod_no_pres_weak}
	\end{split} \\
	\begin{split}
	\implies \quad  & \frac{1}{2}\int_{\Omega}\left( \rho\left(\phi^{k+1}\right)\norm{{\vec{v}}^{k+1}}^2 - \rho\left(\phi^{k}\right)\norm{\vec{v}^{k}}^2 \right)  \, d\vec{x}
	+\frac{Cn}{We}\left( \pd{}{x_j}\left({\pd{\tphi^{k}}{x_i}\pd{\tphi^{k}}{x_j}}\right), \delta t \,\tvi^k\right) \\
	&  \qquad   + \frac{\delta t}{Re} \norm{\sqrt{\eta(\phi^{k+1})} \nabla \widetilde{\vec{v}}^k}^2_{L^2}
	 -\frac{1}{Fr}\left(\hat{g_i}, \delta t \, \rho\left(\phi^{k+1}\right)\tvi^k\right) = 0.
	\label{eqn:disc_nav_stokes_inner_prod_no_pres_weak_norm}
	\end{split}
	\end{align}
	Next we invoke \cref{lem:forcing_equivalent} and write \cref{eqn:disc_nav_stokes_inner_prod_no_pres_weak_norm} as:
	\begin{align}
	\begin{split}
	\frac{1}{2}\int_{\Omega}\left( \rho\left(\phi^{k+1}\right)\norm{{\vec{v}}^{k+1}}^2 - \rho\left(\phi^{k}\right)\norm{\vec{v}^{k}}^2 \right)  \, d\vec{x} 
	&+ \frac{\delta t}{WeCn}\left(\tphi^k\tvi^k,\pd{\tmu^{k}}{x_i}\right) 
	+ \frac{\delta t}{Re} \norm{\sqrt{\eta(\phi^{k+1})} \nabla \widetilde{\vec{v}}^k}^2_{L^2} \\
	&- \frac{1}{Fr}\left(\hat{g_i}, \delta t \, \rho\left(\phi^{k+1}\right)\tvi^k\right) = 0.
	\end{split}
	\label{eqn:disc_nav_stokes_eqv_forcing_ch_without_grav} 
	\end{align}
	Next we simplify the gravity term.
	Notice that
	\begin{align}
		-\frac{1}{Fr}\left(\hat{g_i}, \delta t \, \rho\left(\phi^{k+1}\right)\tvi^k\right) = -\frac{1}{Fr}\left(\pd{\left(-y\right)}{x_i}, \delta t \, \rho\left(\phi^{k+1}\right)\tvi^k\right) = 
		-\frac{1}{Fr}\left(y, \delta t \, \pd{ \left(\rho\left(\phi^{k+1}\right)\tvi^k\right) }{x_i}\right),
	\label{eqn:lemma2_estimate_without_gravSimple}
	\end{align}
	where $y = x_2$ and $\hat{\vec{g}} = (0, -1, 0)$. Here we invoke that ${\tvecv}^{k+1} \in \vec{H}_{0}^1(\Omega)$ so the boundary terms go to zero while doing integration by parts. 
	Let $C_1 = \frac{\left(\rho_- - \rho_+\right)}{2} \, m(\phi)$, then using the continuity equation,
	\cref{eqn:cont}, and the definition of $J_i$ we obtain:
	\begin{align}
	\begin{split}
	&\frac{1}{Fr}\left(y, \delta t \, \pd{ \left(\rho\left(\phi^{k+1}\right)\tvi^k\right) }{x_i}\right) = -\frac{1}{Fr}\left(y,  \,  \rho\left(\phi^{k+1}\right) - \rho\left(\phi^{k}\right)  \right)  
	-\frac{\delta t \, C_1}{Fr\,Pe}\left(y,  \, \pd{}{x_i} \left( \pd{\tmu^{k}}{x_i} \right)\right)\\
	&= -\frac{1}{Fr}\left(y,  \,  \rho\left(\phi^{k+1}\right) - \rho\left(\phi^{k}\right) \right)   
	+\frac{\delta t \, C_1}{Fr\,Pe}\left(\pd{y}{x_i} ,  \, \pd{\tmu^{k}}{x_i} \right)\\
	&= -\frac{1}{Fr}\left(y,  \,  \rho\left(\phi^{k+1}\right) - \rho\left(\phi^{k}\right)  \right)  
	- \frac{\delta t \, C_1}{Fr\,Pe}\left( \pd{}{x_i}\left(\pd{y}{x_i}\right) ,  \, \tmu^{k} \right) 
	 + \frac{\delta t \, C_1}{Fr\,Pe} \int_{\mathrm{d}\Omega}\tmu^{k}\left(\pd{y}{x_i}\right) \hat{n_i} \mathrm{d}\vec{x}\\
	&= -\frac{1}{Fr}\left(y,  \,  \rho\left(\phi^{k+1}\right) - \rho\left(\phi^{k}\right) \right)  
	- \frac{\delta t \, C_1}{Fr\,Pe}\left( \pd{}{x_i}\left(\pd{y}{x_i}\right) ,  \, \tmu^{k} \right) 
	  + \frac{\delta t \, C_1}{Fr\,Pe} \int_{\mathrm{d}\Omega}\tmu^{k}\hat{g_i} \hat{n_i} \mathrm{d}\vec{x},
	\end{split}
	\end{align} 
	where $\hat{n_i}$ is outward pointing normal to the boundary of the domain $\Omega$.  
	\begin{remark}
		We will assume that 
		\begin{equation}
		\frac{\delta t C_1}{Fr\,Pe} \int_{\mathrm{d}\Omega}\tmu^{k}\hat{g_i} \hat{n_i} \mathrm{d}\vec{x} = 0,
		\end{equation}
		which is true as long as there is no three-phase contact line on any boundary on which $\hat{n_i} \hat{g_i}$ is non-zero.
	\end{remark} 
	Under the above assumption we can write 
	\begin{equation}
	-\frac{1}{Fr}\left(\hat{g_i}, \delta t \, \rho\left(\phi^{k+1}\right)\tvi^k\right) =  
	-\frac{1}{Fr}\left(y, \delta t \, \pd{ \left(\rho\left(\phi^{k+1}\right)\tvi^k\right) }{x_i}\right) = 
	\frac{1}{Fr}\left(y,  \,  \rho\left(\phi^{k+1}\right) - \rho\left(\phi^{k}\right) \right).
	\end{equation}
	Combining this last result with \cref{eqn:disc_nav_stokes_eqv_forcing_ch_without_grav} yields the desired result:
	\begin{align}
	\begin{split}
	\frac{1}{2}\int_{\Omega}\left( \rho\left(\phi^{k+1}\right)\norm{{\vec{v}}^{k+1}}^2 - \rho\left(\phi^{k}\right)\norm{\vec{v}^{k}}^2 \right)  \, d\vec{x} + & \,
	\frac{\delta t}{WeCn}\left(\tphi^k\tvi^k,\pd{\tmu^{k}}{x_i}\right) 
	+ \frac{\delta t}{Re} \norm{\sqrt{\eta(\phi^{k+1})} \nabla \widetilde{\vec{v}}}^2_{L^2} \\
	+& \, \frac{1}{Fr}\left(y,  \,  \rho\left(\phi^{k+1}\right) - \rho\left(\phi^{k}\right)  \right) = 0.
	\end{split}
	\label{eqn:disc_nav_stokes_eqv_forcing_ch} 
	\end{align}	
\end{proof}

\begin{proposition}
	The following identity holds: 
	\begin{equation}
	 \left({\psi^\prime(\tphi^k)}, \phi^{k+1} - \phi^k\right) = \left(\psi(\phi^{k+1}) - \psi(\phi^{k}), 1\right) -  \left(\frac{\psi^{\mathrm{\prime\prime\prime}}(\lambda)}{24},\left(\phi^{k+1} - \phi^k\right)^3\right),
	 \label{eq:cubic_term}
	 \end{equation}
	for some $\lambda$ between $\phi^k$ and $\phi^{k+1}$.
	\label{prop:free_en_taylor}
\end{proposition}
\begin{proof}
    Recall that $\tphi^k = (\phi^{k+1}+\phi^k)/2$. From Taylor series we note the following:
    \[
    \psi^\prime\left( \frac{\phi^{k+1}+\phi^k}{2} \right) - \frac{\psi(\phi^{k+1}) - \psi(\phi^{k})}{\phi^{k+1} - \phi^k}
    = -\frac{\psi'''(\lambda)}{24}  \left( \phi^{k+1} - \phi^k \right)^2,
    \]
    for some $\lambda$ between $\phi^k$ and $\phi^{k+1}$. Computing the inner product of this expression
    with $\phi^{k+1} - \phi^k$
    and slightly re-arranging yields the desired result.
\end{proof}

\begin{claim}[Estimate of the correction]
	The following estimate holds:
	\begin{align}
	\left| \left( \frac{\psi^{\mathrm{\prime\prime\prime}}(\lambda)}{24},\left(\phi^{k+1} - \phi^k\right)^3 \right) \right| \le C_m L^3\delta t^3 \norm{\frac{\psi^{\mathrm{\prime\prime\prime}}(\lambda)}{24}}_{L^\infty(\Omega)},
	\label{eqn:estimate_corr_cl}
	\end{align}
	where $L$ is a Lipschitz constant and $C_m$ is the volume of the physical domain:
	\[
		\left| \phi^{k+1} - \phi^k \right| \le L \, \delta t \qquad \text{and} \qquad
		C_m := \int_{\Omega} \mathrm{d}\vec{x}.
	\]
	\label{claim:correction_estimate}
\end{claim}
\begin{proof}
    We start with the error term in \cref{prop:free_en_taylor} and obtain the following upper bound:
	\begin{align}
	\begin{split}
	\left| \left(\frac{\psi^{\mathrm{\prime\prime\prime}}(\lambda)}{24},\left(\phi^{k+1} - \phi^k\right)^3\right) \right|
	&= \abs{\int_{\Omega} \frac{\psi^{\mathrm{\prime\prime\prime}}(\lambda)}{24}\left(\phi^{k+1} - \phi^k\right)^3 \mathrm{d}\vec{x}}\\
	&\le \norm{\frac{\psi^{\mathrm{\prime\prime\prime}}(\lambda)}{24}}_{L^\infty(\Omega)} \int_{\Omega} \abs{\phi^{k+1} - \phi^k}^3\mathrm{d}\vec{x}.
	\end{split}
	\end{align}
	Using the Lipschitz continuity of $\phi$ we arrive at the desired result: \cref{eqn:estimate_corr_cl}.
\end{proof}

 We are now in a position to prove energy stability.  The argument we present here is based on the fact that the energy functional given by  \cref{eqn:energy_functional} is decreasing as the discrete solution is evolving in time. This represents the strict adherence to the second law of thermodynamics at the semi-discrete level.  Therefore, if the difference between the energy functional between two time steps is negative, then we have achieved energy stability. We prove energy stability
 in the following theorem. 
		
\begin{theorem}[Energy stability]
		The time discretization of the Cahn-Hilliard Navier-Stokes (CHNS) equations as described by  \cref{eqn:nav_stokes_var_semi_disc} -- \cref{eqn:phi_eqn_var_semi_disc} is energy stable and satisfies the following energy law:
		\begin{equation}
		\begin{split}
		E_{tot}\left(\vec{v}^{k+1},\phi^{k+1} \right) - E_{tot}\left(\vec{v}^k,\phi^k\right) &= \frac{-\delta t}{Re} \norm{\sqrt{\eta(\phi^{k+1})}\widetilde{\vec{v}}^k}^2_{L^2}
		 - \frac{\delta t}{Pe Cn^2We} \norm{\nabla \widetilde{\mu}^k}^2_{L^2} \\ &+ \frac{1}{WeCn}\left(\frac{\psi^{\mathrm{\prime\prime\prime}}(\lambda)}{24},\left(\phi^{k+1} - \phi^k\right)^3\right),
		\label{energy_law}
		\end{split}
		\end{equation}	
		provided the following time-step restriction is observed:
		\begin{equation}
				0 \le \delta t \leq 
		\left(\frac{\frac{1}{Re}\left(\norm{\sqrt{\eta(\phi^{k+1})}\nabla \widetilde{\vec{v}}^k}_{L^2}^2\right) + 
			\frac{1}{Pe Cn^2We}\norm{\nabla \widetilde{\mu}^k}_{L^2}^2}
		{\frac{C_m L^3}{WeCn} \norm{\frac{\psi^{\mathrm{\prime\prime\prime}}(\lambda)}{24}}_{L^\infty(\Omega)}}\right)^{\frac{1}{2}}. \label{eqn:timestep_condition_thr}
		\end{equation}
		where, 
		\begin{equation}
		E_{\text{tot}}(\vec{v}^k,\phi^k) = \int_{\Omega}\frac{1}{2}\rho\left(\phi^k\right) \norm{\vec{v}^k}^2 \mathrm{d}\vec{x} + \frac{1}{CnWe}\int_{\Omega} \left(\psi(\phi^k) + \frac{Cn^2}{2} \norm{\nabla\phi^k}^2 + \frac{1}{Fr} \rho(\phi^k) y \right) \mathrm{d}\vec{x}.
		\label{eqn:energy_functional_disc}
		\end{equation}	
\label{th:energy_stability}
\end{theorem}   
	
\begin{proof}
		We begin with taking the $L^2$ inner product of \cref{eqn:disc_time_phi_eqn_semi} with $\delta t \, \tmu^k$: 
		\begin{align}
		 \left(\phi^{k+1} - \phi^k, \tmu^k\right) = - \left(\pd{\left(\tvi^{k} \tphi^{k}\right)}{x_i}, \delta t \, \tmu^k\right) + \frac{\delta t}{PeCn} \left(\pd{}{x_i}\left({\pd{\tmu^{k}}{x_i}}\right) , \, \tmu^k  \right),
		\label{eq:weak_phi_eqn}
		\end{align}
		and integrate-by-parts on the right-hand side:
		\begin{equation}
		\begin{split}
		\left(\phi^{k+1} - \phi^k, \tmu^k\right) &= 
		  \left( \tvi^{k} \tphi^{k} , \delta t \, \pd{\tmu^k}{x_i} \right) - \frac{\delta t}{PeCn} \norm{\nabla \widetilde{\mu}^k}^2_{L^2}.	
		\label{eqn:phi_eqn_simple}
		\end{split}
		\end{equation}
		We now work with the $\mu$ equation by taking the $L^2$ inner product of \cref{eqn:disc_mu_eqn_semi} with $\phi^{k+1} - \phi^k$:
		\begin{equation}
		\left(\tmu^{k} , \phi^{k+1} - \phi^k\right) = \left(\tpsi', \phi^{k+1} - \phi^k\right) - Cn^2 \left(\pd{}{x_i}\left({\pd{\tphi^{k}}{x_i}}\right), \phi^{k+1} - \phi^k\right), 
		\end{equation}
		where $\tpsi'$ is defined by \cref{eqn:psi_ave_def}, and integrate-by-parts on the last term:
		\begin{align}
		 \left(\tmu^{k} , \phi^{k+1} - \phi^k\right) = \left(\tpsi', \phi^{k+1} - \phi^k\right) + \frac{Cn^2}{2} \left(\norm{\nabla \phi^{k+1}}^2_{L^2} - \norm{\nabla \phi^{k}}^2_{L^2}\right),
		\label{eqn:mu_eq_unsimplified}
		\end{align}
		where we also used the fact that $\tphi^{k} = (\phi^{k+1} + \phi^k)/2$.
		The first term on right-hand side of \cref{eqn:mu_eq_unsimplified} can be simplified further using            
		\cref{prop:free_en_taylor}:
		\begin{equation}
		\begin{split}
		\left(\tmu^{k} , \phi^{k+1} - \phi^k\right) &= \left(\psi(\phi^{k+1}) - \psi(\phi^{k}), 1\right) - \left(\frac{\psi^{\prime\prime\prime}(\lambda)}{24},\left(\phi^{k+1} - \phi^k\right)^3\right)  \\
		&+ \frac{Cn^2}{2} \left(\norm{\nabla \phi^{k+1}}^2_{L^2} - \norm{\nabla \phi^{k}}^2_{L^2}\right).
		\label{eqn:mu_eq_simple}
		\end{split}
		\end{equation}
		 Now, combining \cref{eqn:mu_eq_simple} and \cref{eqn:phi_eqn_simple} we have: 
		\begin{equation}
		\begin{split}
		 \left(\psi(\phi^{k+1}) - \psi(\phi^{k}), 1\right) &- 
		\left(\frac{\psi^{\prime\prime\prime}(\lambda)}{24},\left(\phi^{k+1} - \phi^k\right)^3\right) +  
		\frac{Cn^2}{2} \left(\norm{\nabla \phi^{k+1}}^2_{L^2} - \norm{\nabla \phi^{k}}^2_{L^2}\right) \\ &= 
		\left(\tvi^{k} \tphi^{k}, \delta t \pd{\tmu^k}{x_i}\right) - \frac{\delta t}{PeCn} \norm{\nabla \widetilde{\vec{\mu}}^k}^2_{L^2}.
		\end{split}
		\label{eqn:ch_eq_simple}
		\end{equation}
		Next we divide \cref{eqn:ch_eq_simple} by $We Cn$ and from \cref{lem:advection_estimate} 
		we can replace the first term on the right-hand side by \cref{eqn:cahn_hilliard_term_lemma}:
		\begin{equation}
		\begin{split}
		& \frac{1}{2}\int_{\Omega}\left( \rho\left(\phi^{k+1}\right)\norm{{\vec{v}}^{k+1}}^2 - \rho\left(\phi^{k}\right)\norm{\vec{v}^{k}}^2 \right)  \, d\vec{x}
		+ \frac{\delta t}{Re} \left(\norm{ \sqrt{\eta(\phi^{k+1})} \nabla \widetilde{\vec{v}}^k}^2_{L^2}\right) \\ & \qquad \qquad + 
		\frac{1}{WeCn}\left(\psi(\phi^{k+1}) - \psi(\phi^{k}), 1\right) 
		- \frac{1}{WeCn}\left(\frac{\psi^{\prime\prime\prime}(\lambda)}{24},\left(\phi^{k+1} - \phi^k\right)^3\right) \\
		& \qquad \qquad + \frac{Cn}{2We} \left(\norm{\nabla \phi^{k+1}}^2_{L^2} - \norm{\nabla \phi^{k}}^2_{L^2}\right) 
		+ 
		\frac{\delta t}{PeCn^2We} \norm{\nabla \tmu^{k}}^2_{L^2} \\ & \qquad \qquad + \frac{1}{Fr}\left(y,  \,  \rho\left(\phi^{k+1}\right) - \rho\left(\phi^{k}\right) \right)  =0.
        \end{split}
         \label{eqn:disc_nav_stokes_eqv_forcing_com}
		\end{equation}
		Simplifying and using the definition of the energy functional, \cref{eqn:energy_functional}, 
		we obtain the energy law:
		\begin{equation}
		\begin{split}
			E_{\text{tot}}\left(\vec{v}^{k+1},\phi^{k+1}\right) - E_{\text{tot}}\left(\vec{v}^k,\phi^k\right)  &= 
			\frac{-\delta t}{Re}\norm{ \sqrt{\eta(\phi^{k+1})} \nabla \widetilde{\vec{v}}^k}_{L^2}^2 - 
			\frac{\delta t}{Pe Cn^2We} \norm{\nabla \tmu^{k}}_{L^2}^2 \\
			&+\frac{1}{WeCn}\left(\frac{\psi^{\mathrm{\prime\prime\prime}}(\lambda)}{24},\left(\phi^{k+1} - \phi^k\right)^3\right).
			\end{split}
		\end{equation}
		In order for this energy to be non-increasing in forward time, we require the following:
		\begin{equation}
		\frac{\delta t}{Re} \norm{\sqrt{\eta(\phi^{k+1})} \nabla \widetilde{\vec{v}}^k}_{L^2}^2 + 
		\frac{\delta t}{Pe Cn^2We} \norm{\nabla \tmu^{k}}_{L^2}^2 \geq
		\frac{1}{WeCn}\left(\frac{\psi^{\mathrm{\prime\prime\prime}}(\lambda)}{24},\left(\phi^{k+1} - \phi^k\right)^3\right).
		\label{eqn:stab_condition}
		\end{equation}
		Using the estimate from \cref{claim:correction_estimate} we can guarantee this inequality
		provided that:
		\begin{equation}
		\frac{\delta t}{Re} \norm{\sqrt{\eta(\phi^{k+1})} \nabla \widetilde{\vec{v}}^k}_{L^2}^2 + 
		\frac{\delta t}{Pe Cn^2We} \norm{\nabla \tmu^{k}}_{L^2}^2 \geq
		\frac{1}{WeCn} \norm{\frac{\psi^{\mathrm{\prime\prime\prime}}(\lambda)}{24}}_{L^\infty(\Omega)}
		C_m L^3 \delta t^3.
		\end{equation}
		This condition can be turned into a condition on the maximum energy-stable time-step:
		\begin{equation}
		 0 \le \delta t \le \left( \frac{\frac{1}{Re} \norm{\sqrt{\eta(\phi^{k+1})} \nabla \widetilde{\vec{v}}^k}_{L^2}^2 + 
		\frac{1}{Pe Cn^2We} \norm{\nabla \tmu^{k}}_{L^2}^2}{\frac{C_m L^3}{WeCn} \norm{\frac{\psi^{\mathrm{\prime\prime\prime}}(\lambda)}{24}}_{L^\infty(\Omega)}} \right)^{\frac{1}{2}},
		\label{eqn:timestep_condition}
		\end{equation}
		which proves the theorem.
\end{proof}
\begin{remark}
	It is important to note that condition \cref{eqn:timestep_condition} is a very weak condition (satisfied for most $\delta t$), as all the quantities in the condition are order one quantities. The bounds presented for $\left(\frac{\psi^{\mathrm{\prime\prime\prime}}(\lambda)}{24},\left(\phi^{k+1} - \phi^k\right)^3\right)$ are the absolute worst case scenarios, which in practice would
	only rarely be achieved.  Therefore, though we cannot claim unconditional stability for the scheme, we can say that the scheme is energy stable for large range of $\delta t$ values and allows us to take large time steps.  It is common practice in the literature to approximate the free energy functional such that it has a form which will not result in the cubic term in \cref{eq:cubic_term} (see \citep{Shen2010, Shen2010a,Shen2010b} for examples), which results in an unconditionally stable scheme.  In the estimate we have presented we do not make any approximations on the form of free energy functional, which results in a slightly tighter restriction on the time step restriction.  
\end{remark}

\subsection{Solvability of the discrete-in-time, continuous-in-space CHNS system}
\label{subsec:semidiscrete_var_prob}
In this subsection we establish the solvability of system of equations \cref{eqn:nav_stokes_var_semi_disc} -- \cref{eqn:phi_eqn_var_semi_disc}.  
We follow the basic strategy used by \citet{Han2015}, which, after adaptation to the specific Cahn-Hilliard Navier-Stokes system considered in this work, can be summarized as follows:
\begin{itemize}
\item Show that \cref{eqn:mu_eqn_var_semi_disc} has the following property: given $\widetilde{\mu}^k$, 
then $\tphi^{k}$ is uniquely determined;
\item Show that \cref{eqn:nav_stokes_var_semi_disc} -- \cref{eqn:cont_var_semi_disc} have the following property: 
given $\widetilde{\mu}^k$, and hence $\phi^{k+1}$ from
\cref{eqn:mu_eqn_var_semi_disc} as stated above, then $\widetilde{\vec{v}}^k$ and $\widetilde{p}^{k}$ are uniquely determined;
\item This establishes $\tphi^{k}$, $\widetilde{\vec{v}}^k$, and $\widetilde{p}^{k}$ as uniquely determined by
$\widetilde{\mu}^k$; with this knowledge in hand, we can now view the remaining equation, \cref{eqn:phi_eqn_var_semi_disc}, 
as a scalar equation for $\widetilde{\mu}^k$;
\item Show that there exists a solution, $\widetilde{\mu}^k$, to  \cref{eqn:phi_eqn_var_semi_disc},
with $\tphi^{k}$ and $\widetilde{\vec{v}}^k$ understood
to be functions of $\widetilde{\mu}^k$ via \cref{eqn:nav_stokes_var_semi_disc} --  \cref{eqn:mu_eqn_var_semi_disc}.
\end{itemize}

The key to the above argumentation is the Browder-Minty theorem (e.g., see theorem 9.14 \citet{Ciarlet2013}) and the main theorem on pseudo-monotone operators due to Brezis (see theorem 27.A \citet{Zeidler1985IIb}), both of which we reproduce below for completeness.  
\begin{theorem}{(Browder-Minty (1963))}
	Let $X$ be a real, reflexive, Banach space and let $T:X \longrightarrow X^*$ be a monotone, coercive, continuous, and bounded operator, then for any $b \in X^*$, there exists a solution to 
	\begin{equation}
	T(u) = b.
	\end{equation}
	further, if $T:X \longrightarrow X^*$ is strictly monotone, then the solution $u$ is unique
	\label{thrm:browder_minty}  		
\end{theorem}

\begin{theorem}{(Brezis (1968))}
	Let $X$ be a real, reflexive, Banach space and let $T:X \longrightarrow X^*$ be a pseudo-monotone, coercive, continuous, and bounded operator, then for any $b \in X^*$, there exists a solution to 
	\begin{equation}
	T(u) = b.
	\end{equation}
	\label{thrm:pseudo-monotone}  		
\end{theorem}
%
%
\begin{lemma}[Solvability of \eqref{eqn:mu_eqn_var_semi_disc}]
	Given $\tmu^{k} \in H^1(\Omega)$ and $\phi^{k} \in H^1(\Omega)$, there exists a unique
	solution $\tphi^{k} \in H^1(\Omega)$ to \cref{eqn:mu_eqn_var_semi_disc}.
	This establishes the solution operator:
	\[
		\tphi^{k} \left( \tmu^{k} \right): \tmu^{k} \mapsto \tphi^{k}.
    \]
	\label{lemma:mu_unique}
\end{lemma}
Proof of the above lemma follows from \cref{thrm:browder_minty}, where the continuity and boundedness of the solution operator follows from the fact that \cref{eqn:mu_eqn_var_semi_disc} is an elliptic semi-linear equation.  The detailed proof is omitted here for brevity.


\begin{lemma}[Solvability of \eqref{eqn:nav_stokes_var_semi_disc} -- \eqref{eqn:cont_var_semi_disc}]
	Given $\tmu^{k} \in H^1(\Omega)$,
 	 $\vec{v}^k \in \vec{H}^1_0(\Omega)$, and $\vec{\phi}^k \in H^1(\Omega)$,
	 there exists a unique solution $\widetilde{\vec{v}}^{k} \in \vec{H}^1_0(\Omega)$ 
	 and $\widetilde{p}^k \in H^1(\Omega)$ to \cref{eqn:nav_stokes_var_semi_disc} -- 
	 \cref{eqn:cont_var_semi_disc}. 
	 This establishes the solution operator:
	\[
		\widetilde{\vec{v}}^{k} \left( \tmu^{k} \right): \tmu^{k} \mapsto \widetilde{\vec{v}}^{k}.
    \]
	\label{lemma:v_unique}
\end{lemma}
Proof of this lemma follows from the generalized Lax-Milgram theorem under a suitable inf-sup condition. 
We again omit the details of this proof for brevity, and instead refer the interested
ready to \citet{Volker2016} for a detailed explanation.
  
\begin{remark}
	In the fully discrete setting one needs to satisfy the discrete inf-sup condition (which is a modified coercivity condition) to prove uniqueness for the fully discrete analog of \cref{eqn:mu_eqn_var_semi_disc}.  However, we use the variational multi-scale technique, described below in \cref{subsec:space_scheme}, which circumvents the need of a discrete inf-sup condition.  The variational multi-scale technique allows for the use of classical Lax-Milgram to prove uniqueness as one can prove classical definition of coercivity directly in this case. 
\end{remark}


We now prove solvability for the full time-discretized Cahn-Hilliard Navier-Stokes system by 
showing that  \cref{eqn:phi_eqn_var_semi_disc} has a solution, $\widetilde{\mu}^k$,
with $\phi^{k+1}$ and $\widetilde{\vec{v}}^k$ understood
to be functions of $\widetilde{\mu}^k$ via \cref{eqn:nav_stokes_var_semi_disc} --  \cref{eqn:mu_eqn_var_semi_disc}.  We aim to show that all the conditions of \cref{thrm:pseudo-monotone} are satisfied;
\cref{thrm:pseudo-monotone} is a generalization of \cref{thrm:browder_minty} for operators that are a summation of a higher order monotone operator and a strongly continuous lower order operator. To this end, it is important to note that proving strong continuity for our lower order operators in \cref{eqn:phi_eqn_var_semi_disc} is difficult.  However, \citet{Liu2011} showed an equivalent condition called local monotonocity which is easier to prove.

\begin{lemma}
	Given an operator $A:X \longrightarrow X^*$, where $X$ is a real, reflexive Banach space, i.e. $H^1(\Omega)$ in our case.
	If $A$ has the following properties: 
	\begin{itemize}
		\item[(H1)] (Hemicontinuity)  The map $t \mapsto \left\langle A(\xi_1 + t\;\xi_2), q \right\rangle$ is continuous on $\mathbb{R}$;
		\item[(H2)] (Local monotonicity) The following inequality
		holds: 
					\begin{equation}
					\left\langle A(\xi_1) + A(\xi_2), \;\xi_1 - \xi_2 \right\rangle_{H^1(\Omega)} \leq \left( C + \upsilon(\xi_1) + \gamma(\xi_2) \right) \norm{\xi_1 - \xi_2}_{H^1(\Omega)}^2,
					\label{eqn:local_monotinicity}
					\end{equation}
			where $\upsilon(\xi_1)$ and $\gamma(\xi_2)$ are bounded, measurable functions in $H^1(\Omega)$;
	\end{itemize}
then $A$ is pseudo-monotone operator. 
\label{lem:pseudo-monotone-equiv}
\end{lemma}
\citet{Liu2011} proves \cref{lem:pseudo-monotone-equiv} in a more general setting for Banach spaces which are compactly embedded in Hilbert spaces.  

\begin{definition}
	Consider the following shorthand notation:
	\begin{equation}
	\mu := \tmu^k, \quad
	\phi\left(\mu \right) := \tphi^k \left( \tmu^k \right),
	\quad \text{and} \quad
	v_i\left(\mu \right) := \tvi^k\left(\tmu^k \right). 
	\end{equation}
	Then, given $\phi^{k} \in {H}^1(\Omega)$, from \cref{eqn:phi_eqn_var_semi_disc} we establish the following solution operator:
	\begin{equation}
	\left\langle T(\mu), q \right\rangle :=  
	2 \left(  \phi\left(\mu\right) - \phi^k, \, q \right) -
	\frac{\delta t}{2} \left( v_i\left(\mu \right)  \phi\left(\mu \right), \, \pd{q}{x_i} \right) + \frac{\delta t}{PeCn} \left( {\pd{\mu}{x_i}}, \, \pd{q}{x_i} \right),
	\label{eqn:operator_phi}
	\end{equation}
	for all $q \in H^1(\Omega)$.
	\label{defn:operator_defn_exist}
\end{definition}
\begin{lemma}
	If we assume the solution operator $\phi(\mu)$ to be Lipshitz in $\mu$, then the first term from the operator in \cref{defn:operator_defn_exist}, i.e., 
	$2 \left(  \phi\left(\mu\right) - \phi^k, \, q \right)$ satisfies (H1) and (H2) from \cref{lem:pseudo-monotone-equiv} and is therefore pseudo-monotone.
	\label{lemm:time_pseudo}
\end{lemma}

\begin{proof}
	(H1) follows from the continuity of the solution operator from \cref{lemma:mu_unique}. We proceed to check local monotonicity. 
	Using the Lipschitz continuity of $\phi$ and standard
	inequality, one can show that
	\begin{equation}
	\begin{split}
	2 \left(  \phi\left(\xi_1\right) - \phi\left(\xi_2\right) , \, \xi_1 - \xi_2 \right) &\leq 
	2 L^2 \norm{ \xi_1 - \xi_2 }_{H^1}^2
	\biggl[ \norm{\phi\left(\xi_1\right)}_{L^4}^4 + \norm{\phi\left(\xi_2\right)}_{L^4}^4\\
	&+ 2\norm{\phi\left(\xi_1\right)}_{L^6}^3\norm{\phi\left(\xi_2\right)}_{L^2} + 3\norm{\phi\left(\xi_1\right)}_{L^4}^2\norm{\phi\left(\xi_2\right)}_{L^4}^2 \\
	&+ 2\norm{\phi\left(\xi_2\right)}_{L^6}^3\norm{\phi\left(\xi_1\right)}_{L^2} + \frac{1}{2}
	\left( 1 + Cn^2 \right) \biggr],
	\end{split}
	\end{equation}
	where $L$ is the Lipschitz constant for $\phi$.
	This proves local monotonicity (H2).  In the interest of brevity we do not provide the detailed steps here. We refer interested readers to \citet{Liu2011} which presents similar estimates.
\end{proof}

\begin{lemma}
	If we assume the solution operators $\phi(\mu)$ 
	and $\vec{v}(\mu)$ to be Lipshitz in $\mu$, then the second term from the operator in \cref{defn:operator_defn_exist}, i.e., 
	$\frac{\delta t}{2} \left( v_i\left(\mu \right)  \phi\left(\mu \right), \, \pd{q}{x_i} \right)$ satisfies (H1) and (H2) from \cref{lem:pseudo-monotone-equiv} and is therefore pseudo-monotone.
\label{lemm:advect_pseudo}
\end{lemma}
\begin{proof}
	(H1) follows from the continuity of the solution operators from \cref{lemma:mu_unique} and \cref{lemma:v_unique}. We proceed to check local monotonicity. 
	Using Lipshitz continuity and standard inequalities
	one can show that
	\begin{align}
	\begin{split}
	\left( v_i\left(\xi_1 \right)  \phi\left(\xi_1 \right) -  v_i\left(\xi_2 \right)  \phi\left(\xi_2 \right), \, \pd{\left(\xi_1 - \xi_2\right)}{x_i} \right)
	& \leq L\norm{\left(\xi_1 - \xi_2\right) }_{H^1}^2 \biggl[ \norm{ \vec{v}\left(\xi_1 \right) }_{L^\infty} + \norm{ \phi\left(\xi_2 \right) }_{L^\infty}  \biggr],
	\end{split}
	\end{align}
	where $L$ is the maximum Lipschitz constant for $\phi$
	and $\vec{v}$.
	This proves local monotonicity (H2).
\end{proof}
 
\begin{theorem}[Solvability of \eqref{eqn:phi_eqn_var_semi_disc}]
       Given $\phi^{k} \in {H}^1(\Omega)$, there exists a solution, $\mu \in H^1(\Omega)$, to \cref{eqn:phi_eqn_var_semi_disc} in the sense of the solution operator defined in \cref{defn:operator_defn_exist} that satisfies the conditions of the existence theorem of pseudo-monotone operators (see \cref{thrm:pseudo-monotone}
       above).		
\end{theorem}
\begin{proof}
%
%
%
	%
	%
	%
	We now proceed by proving that $T(\mu)$ satisfies the conditions of the Browder-Minty theorem.
	\begin{enumerate}
	\item {\bf Continuous and bounded.} We compute the absolute value of
	\cref{eqn:operator_phi} and use standard inequalities:
	\begin{equation}
	\begin{split}
	\left|\left\langle T(\mu), \, q \right\rangle\right| &\leq C_1\norm{q}_{L^2}\left(\norm{\phi(\mu)}_{L^2} +  \norm{\phi^{k}}_{L^2}\right) \\ &+ 
	C_2\delta t\norm{\nabla {q}}_{L^2}\norm{\vec{v}(\mu)}_{L^2}\norm{\phi(\mu)}_{L^2}  + 
	\frac{C_3 \delta t}{PeCn} \norm{\nabla {q}}_{L^2}\norm{\nabla {\mu}}_{L^2}.
	\end{split}
	\end{equation}
	Using the fact that $\norm{\cdot}_{L^2} \le \norm{\cdot}_{H^1}$ and combining all
	the constants yields:
	\begin{equation}
	\left|\left\langle T(\mu), \, q \right\rangle\right| \leq C(\delta t) \,  
	\norm{q}_{H^1} \, \left[ \norm{\phi^{k}}_{L^2} + \bigl(1 + \norm{\vec{v}(\mu)}_{L^2} \bigr) \norm{\phi(\mu)}_{L^2}   + \norm{\nabla \mu}_{L^2}\right].
	\end{equation}
	Therefore, operator $T$ is bounded as a consequence of the boundedness of $\phi(\mu)$ 
	and $\vec{v}(\mu)$ from \cref{lemma:mu_unique} and \cref{lemma:v_unique}, respectively.
	Continuity of $T$ follows from a similar argument.
	\item {\bf Pseudo-monotonocity.}
	We begin with the following expression:
	\begin{equation}
	\begin{split}
	\left\langle T(\xi_1) - T(\xi_2), \xi_1 - \xi_2 \right\rangle &=  
	2\left( \phi\left(\xi_1\right) - \phi\left(\xi_2\right), \, \xi_1 - \xi_2 \right) \\ 
	&- \frac{\delta t}{2} \left( \vec{v}\left(\xi_1 \right) \phi\left(\xi_1 \right) - 
	\vec{v}\left(\xi_2 \right) \phi\left(\xi_2 \right), \, \nabla \left(\xi_1 - \xi_2\right) \right) \\
	&+ \frac{\delta t}{PeCn} \bigl\| \nabla \left(\xi_1 - \xi_2\right) \bigr\|^2_{L^2},
	\end{split}
	\end{equation}
	for all $\xi_1, \xi_2 \in H^1(\Omega)$.
	Note that the third term on the right hand side is strictly monotone.  The first and second term on the right hand side is shown to be pseudo-monotone in \cref{lemm:time_pseudo} and \cref{lemm:advect_pseudo} respectively. Further, the summation of a pseudo-monotone operator and a monotone operator is also pseudo-monotone (Proposition 27.6 of \citet{Zeidler1985IIb}). 	
	Which implies $T$ is a pseudo-monotone operator. 	
	\item {\bf Coercivity.}: Coercivity in this context is written as 
	\begin{equation}
	\frac{\left\langle T(u),u \right\rangle}{\norm{u}_{H^1}} \longrightarrow \infty, \;\; as \;\;\norm{u}_{H^1} \longrightarrow \infty.
	\end{equation}
	We check whether this condition is satisfied. 
	We start from  		 
	\begin{equation}
	\left\langle T(\mu), \mu \right\rangle = 2\left(\left(\phi(\mu) - \phi^k\right), \mu\right) - \frac{\delta t}{2}\left( v_i(\mu) \phi(\mu), \pd{\mu}{x_i}\right) + \frac{\delta t}{PeCn} \left(\pd{\mu}{x_i},\pd{\mu}{x_i}  \right).
	\end{equation}
	The first term can be bounded by taking the test function $q = \phi^{k+1} - \phi^k$ in \cref{eqn:mu_eqn_var_semi_disc} and using the fact that $\phi^k$, $\phi^{k+1}$ are strictly bounded between $-1$ and $1$:
	\begin{align}
	\begin{split}
	\left( \phi^{k+1} - \phi^k, \mu\right) &\geq 
	\frac{1}{8}\int_{\Omega}(\phi^{k+1})^4 \mathrm{d}x - \frac{3}{4}\int_{\Omega}(\phi^{k+1})^2 \mathrm{d}x + \frac{Cn^2}{2}\int_{\Omega}\sum_i \left|\pd{\phi^{k+1}}{x_i}\right|^2 \mathrm{d}x \\
	& \quad + \frac{3}{4}\int_{\Omega}(\phi^{k})^2 \mathrm{d}x + \frac{1}{8}\int_{\Omega}(\phi^{k})^4 + \mathrm{d}x - \frac{Cn^2}{2}\int_{\Omega}\sum_i \left|\pd{\phi^{k}}{x_i}\right|^2 \mathrm{d}x\\ 
	&\geq \frac{1}{16}\norm{\phi^{k+1}}_{L^4}^4 + \frac{Cn^2}{2}\norm{\nabla\phi^{k+1}}_{L^2}^2 - C\left(\norm{\phi^k}^4_{H^1} + \norm{{\phi}^k}^2_{H^1} + 1\right).
	\end{split}
	\end{align}
	Using the estimate from \cref{lem:advection_estimate}, 
	the second term on the right can be written as
	\begin{equation}
	\begin{split}
	-\delta t\left( v_i \phi, \pd{\mu}{x_i}\right) \geq C_3\biggl[\norm{\vec{v}^{k+1}}^2_{L^2} - \norm{\vec{v}^k}^2_{L^2} + \norm{\sqrt{\eta(\phi^{k+1})} \; \nabla\vec{v}}^2_{L^2}   
	 +   \left(y,  \, \rho^{k+1} - \rho^k \right)  \biggr].
	\end{split} 
	\end{equation}
	Collecting all inequalities, we get
	\begin{align}
	\begin{split}
	\left\langle T(\mu), \mu \right\rangle &\geq 
	 C \norm{\nabla\mu}_{L^2}^2 + \frac{1}{16} \norm{\phi^{k+1}}_{L^4}^4 + \frac{Cn^2}{2}\norm{\nabla\phi^{k+1}}_{L^2}^2 + C_5\norm{\vec{v}^{k+1}}^2_{L^2}\\
	&\quad + C_3\left(\norm{\sqrt{\eta(\phi^{k+1})} \; \nabla\vec{v}}^2_{L^2} + \left(y,  \,\rho\left(\phi^{k+1}\right)\right)\right) -  C_b.
	\end{split}
	\label{eq:coercivity_grad_mu}
	\end{align} 	
	Note, that the all the terms which depend on $\mu$ in \cref{eq:coercivity_grad_mu} are positive, and all the fields at the previous timestep (known functions) are constant, and are absorbed in $C_b$.  To show coercivity we need to use Poincar\'e inequality to get the inequality in terms of $\mu$ instead of $\nabla\mu$; thus, 
	\begin{equation}
	\left\langle T(\mu), \mu \right\rangle \geq C_a \norm{\mu}_{H^1}^{n} - C_b,
	\label{eq:coercivity_mu}
	\end{equation}
	where the constants are adjusted for the inequalities used and $n = 4/3$ which implied coercivity.  Here we do not write other positive terms ($\frac{Cn^2}{2}\norm{\nabla\phi^{k+1}}_{L^2}^2 + C_5\norm{\vec{v}^{k+1}}^2_{L^2}$.. so on) which are $\mu$ dependent to show the inequality in a clearer form, as they are positive and the coercivity condition is still satisfied with them included.  
\end{enumerate}
	\begin{remark}
		To go from \cref{eq:coercivity_grad_mu} to \cref{eq:coercivity_mu}, we use a particular form of the Poincar\'e inequality for $\mu$ given in \citep{Oden2012}:
		\begin{align}
		\begin{split}
		& \norm{f}_{H^1(\Omega)}^2 \leq \left(C^{\prime} + 1\right)  \norm{\nabla f}_{L^2(\Omega)}^2  + C^{\prime} \left(\int_{\Omega} f \, d\vec{x}\right)^2,
		\end{split}
		\label{eqn:poincare}
		\end{align}
	 $\forall f \in  H^1(\Omega)$. To use this inequality for $\mu$, one needs to bound the average $\int_{\Omega} \mu \, d\vec{x}$ and use it in conjunction with \cref{eq:coercivity_grad_mu}.  Using standard inequalities one can show that $\norm{\phi^{k+1}}_{L^4}^4$ bounds $\left(\int_{\Omega} \mu \, d\vec{x}\right)^{4/3}$.  Therefore, the second term in \cref{eq:coercivity_grad_mu} can be used to replace the average term in \cref{eqn:poincare}, which gives the exponent n = 4/3 for $\norm{\mu}_{H^1}^{n}$.  The proof is similar to that shown in \citep{Han2015}, so we do not reproduce it in detail here. 
	\end{remark}
	
	We proved all the required conditions for existence theorem for pseudo-monotone operators, this implies that there exists a solution $\mu^{\prime}$ such that $\left\langle T(\mu^{\prime}), w \right\rangle = 0$, $\forall w \in H^1(\Omega)$. Consequently, there $\tmu^{k}$ is a solution to \cref{eqn:operator_phi}. The same $\tmu^k(\vec{x})$ is the source function for the mapping in \cref{lemma:mu_unique,lemma:v_unique} which provide unique solutions $\phi^{k+1}(\vec{x})$, and $\vec{v}^{k+1}(\vec{x})$ respectively. 
\end{proof}

\begin{remark}
	Contrary to strictly monotone operator in \citet{Han2015}, the operator for our fully implicit time-scheme is pseudo-monotone. The strictly monotone operator along with the Browder-Minty theorem (\cref{thrm:browder_minty}) gives uniqueness of solutions.  In our case we are using a generalization of Browder-Minty theorem to pseudo-monotone operators which only gives existence of solutions.  Rigorously proving uniqueness in our case is not trivial, but for practical situations we do not see any problems. 
\end{remark}

\subsection{Spatial discretization and variational multi-scale approach}
\label{subsec:space_scheme}
In this work all the unknown variables, ($\phi$, $\mu$, $\vec{v}$, and $p$), are discretized in space using the standard piecewise linear continuous Galerkin or cG(1) finite element method.
 It is  well-known in the literature that numerical instabilities occur when  solving
the incompressible Navier-Stokes with a numerical method that uses the same polynomial order for
both the velocity and the pressure. These instabilities are due to fact that equal polynomial order
representations of $p$ and $\vec{v}$ will not satisfy the inf-sup condition (Ladyzhenskaya-Babuska-Brezzi condition, see page 31 in \citet{Volker2016}). In order to overcome this difficulty, additional numerical
stabilization needs to be introduced. One of the most popular stabilization techniques for this problem is the SUPG-PSPG approach: 
streamline-upwind/Petrov-Galerkin (SUPG) \cite{article:BrooksHughes1982} and pressure-stabilizing/Petrov-Galerkin (PSPG) \cite{article:TezMitRayShi92}.  A generalization of these approaches is the variational multi-scale approach \citep{hughes2018multiscale}. In this work we make use of the variational multi-scale (VMS) approach proposed in the context of large-eddy simulations (LES) \citep{Hughes1995}.  This approach has the advantage that it provides a stabilization mechanism such that the inf-sup stability condition is converted to a coercivity condition, while also providing a natural leeway into modeling high Reynolds number flows in the context of LES \citep{Bazilevs2007}. 

The philosophy of VMS models follows that of LES, where we seek a direct-sum decomposition of the discrete spaces which approximate the continuous spaces.  If
 $\vec{v} \in \vec{V}$ and $p \in Q$, then we can decompose these spaces as follows: 
\begin{eqnarray}
\vec{V} = \overline{\vec{V}} \oplus \vec{V}^\prime \qquad \text{and} \qquad Q = \overline{Q} \oplus Q^\prime,
\end{eqnarray}
where $\overline{\vec{V}}$ and $\overline{Q}$ are the cG(1) subspaces of $\vec{V}$ and $Q$, respectively, and the primed versions are the complements of the cG(1) subspaces in $\vec{V}$ and $Q$, respectively.
We can write the decomposition for velocity and pressure as follows: $\vec{v} = 
\overline{\vec{v}} + \vec{v}^{\prime}$ and $p = \overline{p} + p^{\prime}$, 
where the {\it coarse scale} solution is $\overline{\vec{v}} \in \overline{\vec{V}}$, $\overline{p}\in \overline{Q}$, and the {\it fine scale} solution is $\vec{v}^{\prime} \in \vec{V}^\prime$ and $p^{\prime} \in Q^\prime$.  We define a projection operator, $\mathscr{P}:\vec{V} \rightarrow \overline{\vec{V}}$, such that
 $\overline{\vec{v}} = \mathscr{P}\{\vec{v}\}$ and $\vec{v}^\prime = \vec{v} - \mathscr{P}\{\vec{v}\}$.  A similar operator can be used for the decomposition of $p$.  Substituting this decomposition in the original variational form in definition \ref{def:variational_form_sem_disc} yields:
\begin{align}	
\begin{split}	
		\text{Momentum Eqns:}& \quad \left(w_i,\rho(\phi)\pd{ \overline{v_i}}{t}\right) + \left(w_i,
		\pd{\left(\rho(\phi) v_i^{\prime}\right)}{t}\right) +  \left(w_i,\rho(\phi)\overline{v_j}\pd{\overline{v_i}}{x_j}\right) \\ & + \left(w_i,\rho(\phi)v_j^\prime\pd{\overline{v_i}}{x_j}\right) +
		\left(w_i,\pd{\left(\rho(\phi)\overline{v_j}v_i^\prime\right)}{x_j}\right) + \left(w_i,\pd{\left(\rho(\phi)v_j^\prime v_i^\prime\right)}{x_j}\right) \\ & +  
		\frac{1}{Pe}\left(w_i, J_j\pd{\overline{v_i}}{x_j}\right) + \frac{1}{Pe}\left(w_i, \pd{\left(J_j v_i^\prime\right)}{x_j}\right) + \frac{Cn}{We} \left(w_i,\pd{}{x_j}\left({\pd{\phi}{x_i}\pd{\phi}{x_j}}\right)\right)  \\
		& + \frac{1}{We}\left(w_i,\pd{\left(\overline{p} + p^{\prime}\right)}{x_i}\right) + \frac{1}{Re}\left(\pd{w_i}{x_k},\eta(\phi)\pd{\left(\overline{v_i} + v_i^{\prime}\right)}{x_k}\right) - \left(\frac{w_i,\rho(\phi)\hat{g_i}}{Fr}\right)   \\
		&+  \left(q,\pd{\overline{v_i}}{x_i}\right) + \left(q,\pd{v_i^\prime}{x_i}\right)= 0, \label{eqn:weak_VMS_ns}
		\end{split}\\
		\text{Cahn-Hilliard Eqn:}& \quad \left(q, \pd{\phi}{t}\right) + \left(q,\pd{\left(\overline{v_i}\phi\right)}{x_i}\right)
		- \frac{1}{PeCn} \left(q, \frac{\partial^2 \left(m(\phi)\mu\right)}{\partial x_i \partial x_i}\right) = 0, \label{eqn:phi_eqn_var_decom}\\
		\text{Chemical Potential:}& \quad -\left(q,\mu\right) + \left(q, \d{\psi}{\phi}\right) - Cn^2 \left(q, \pd{}{x_i}\left({\pd{\phi}{x_i}}\right)\right)   = 0, \label{eqn:mu_eqn_var_decom}
\end{align} 
where $\vec{w},\vec{\overline{v}}, \in \mathscr{P}\vec{H}^{1}(\Omega)$, $\overline{p},\phi \in \mathscr{P}H^1(\Omega), \vec{v}^\prime \in (\mathscr{I} - \mathscr{P})\vec{H}^{1}(\Omega)$,
$p^\prime \in (\mathscr{I} - \mathscr{P})H^1(\Omega)$, and $\mu, q \in \mathscr{P}H^1(\Omega)$.
Here $\mathscr{I}$ is the identity operator and $\mathscr{P}$ is the projection operator.  We use the residual-based approximation proposed by \citet{Bazilevs2007} for fine scale components to close the equations, which is given by 
\begin{equation}
\rho(\phi) v_i^\prime = -\tau_m \mathcal{R}_m(\rho,\overline{v_i},\overline{p}) \qquad \text{and} \qquad
p^\prime = -\rho(\phi)\tau_c \mathcal{R}_c(\overline{v_i}).
\end{equation}

It is important to note that because we are using block iterative method, the momentum equations,  \cref{eqn:weak_VMS_ns}, and the Cahn-Hilliard equations, \cref{eqn:phi_eqn_var_decom} and \cref{eqn:mu_eqn_var_decom},
 are solved as two different nonlinear sub-problems. We use conforming Galerkin based finite elements and replace the continuous spaces with their discrete counterparts; notice that as we only solve for course scale components, the trial functions and the basis functions are in the same space. Then we can write a discrete variational formulation can we written as follows.
\begin{definition}
	Find $\vec{\overline{v}}^h \in \mathscr{P}\vec{H}^{1,h}(\Omega)$ and $\overline{p}^h, \phi^h, \mu^h\in\mathscr{P}H^{1,h}(\Omega)$ such that
	\begin{align}
	\begin{split}
	\text{Momentum Eqns:}&\quad 
	\left(w_i,\rho(\phi^h)\pd{\overline{v_i}^h}{t}\right) + \left(w_i,\rho(\phi^h)\overline{v_j}^h\pd{\overline{v_i}^h}{x_j}\right) - \left(w_i,\tau_m\mathcal{R}_m(\overline{v_j}^h,\overline{p}^h)\pd{\overline{v_j}^h}{x_j}\right) \\ 
	&\quad+ \left(\pd{w_i}{x_j},\overline{v_j}^h\left(\tau_m\mathcal{R}_m(\overline{v_i}^h,\overline{p}^h)\right)\right) -\left(\pd{w_i}{x_j},\frac{\tau_m^2}{\rho(\phi^h)}\mathcal{R}_m(\overline{v_j}^h,\overline{p}^h)\mathcal{R}_m(\overline{v_i}^h,\overline{p}^h)\right) \\
	&\quad  +
	\frac{1}{Pe}\left(w_i, J_j^h\pd{\overline{v_i}^h}{x_j}\right) + \frac{1}{Pe}\left(\pd{w_i}{x_j}, J_j^h \frac{\tau_m}{\rho(\phi^h)}\mathcal{R}_m(\overline{v_i}^h,\overline{p}^h)\right)  \\ 
	&\quad-\frac{Cn}{We} \left(\pd{w_i}{x_j}, \, {\pd{\phi^h}{x_i}\pd{\phi^h}{x_j}} \right) - 
	\frac{1}{We}\left(\pd{w_i}{x_i},\overline{p}^h\right) \\
	&\quad + \frac{1}{We}\left(\pd{w_i}{x_i},\rho(\phi^h)\tau_c \mathcal{R}_c(\overline{v_i}^h)\right) + \frac{1}{Re}\left(\pd{w_i}{x_k},\eta(\phi^h)\pd{\overline{v_i}^h}{x_k}\right) \\
	&\quad - \left(\frac{w_i,\rho(\phi^h)\hat{g_i}}{Fr}\right)  + \left(q,\pd{\overline{v_i}^h}{x_i}\right) - \left(\pd{q}{x_i},\frac{\tau_m}{\rho(\phi^h)}\mathcal{R}_m(\overline{v_i}^h,\overline{p}^h)\right)= 0, 
	\end{split} \\
	\text{Cahn-Hilliard Eqn:}&\quad \left(q, \pd{\phi^h}{t}\right) - \left(\pd{q}{x_i},\overline{v_i}^h\; \phi^h\right)
	+ \frac{1}{PeCn} \left(\pd{q}{x_i},{\pd{\left(m(\phi^h)\mu^h\right)}{x_i}}\right)
	= 0,\\
	\text{Chemical Potential:}&\quad -\left(q,\mu^h\right) + \left(q, \d{\psi}{\phi^h}\right) + Cn^2 \left(\pd{q}{x_i},{\pd{\phi^h}{x_i}}\right)  = 0, \label{eqn:mu_eqn_var_decom_disc}
	\end{align}
		\label{eqn:weak_VMS_disc}
where, 
\begin{align}
\tau_m &= \left( \frac{4}{\Delta t^2}  + \overline{v_i}^hG_{ij}\overline{v_j}^h + \frac{1}{(\rho(\phi^h)Pe)}\overline{v_i}^hG_{ij}\overline{J_j}^h + C_{I} \left(\frac{\eta(\phi^h)}{\rho(\phi^h)Re}\right)^2 G_{ij}G_{ij}\right)^{-1/2},\\
\tau_c &=  \frac{1}{tr(G_{ij})\tau_m}.
\end{align}
\end{definition}
Here we set $C_{I}$ for all our simulations to 6 and the residuals are given by
\begin{align}
\begin{split}
\mathcal{R}_m(\overline{v_i}^h,\overline{p}^h) &= \rho(\phi)\pd{\overline{v_i}^h}{t} + \rho(\phi)\overline{v_j}^h\pd{\overline{v_i}^h\;}{x_j} + 
\frac{1}{Pe}J_j^h\pd{\overline{v_i}^h}{x_j} + \frac{Cn}{We} \pd{}{x_j}\left({\pd{\phi^h}{x_i}\pd{\phi^h}{x_j}}\right) \\ &+ 
\frac{1}{We}\pd{\overline{p}^h}{x_i} - \frac{1}{Re}\pd{}{x_k}\left({\eta(\phi)\pd{\overline{v_i}^h}{x_k}}\right) - \frac{\rho(\phi)\hat{g}}{Fr}, 
\end{split} \\ 
\mathcal{R}_c(\overline{v_i}^h) &= \pd{\overline{v_i}^h}{x_i}.
\end{align}

Finally, we note that in the above expressions the time derivative is still continuous.  In the fully discrete numerical method we replace the time-derivatives in the momentum and phase field equations using
the trapezoidal rule in the form of the scheme presented in \cref{eqn:disc_nav_stokes_semi} -- \cref{eqn:disc_time_phi_eqn_semi}.

\subsection{Handling non-linearity}
\label{subsec:newton_iter}

The fully discretized system is a collection of two non-linear systems of algebraic equations, one corresponding to the discretized version of the momentum equations (\crefrange{eqn:nav_stokes_var_semi_disc}{eqn:cont_var_semi_disc}), the other corresponding
to the Cahn-Hilliard equations, \cref{eqn:cont_var_semi_disc,eqn:phi_eqn_var_semi_disc}. Because 
we use an implicit time-stepping strategy, an internal (within each block iteration) Newton's method is used to solve the aforementioned non-linear algebraic equations.  Newton's method for a system of equations can be written as follows:
\begin{align}
J_{ij}^{k} \; \delta U_j^{k} =& -F_i^{s,k}(U_1^{s,k}, U_2^{s,k}, \dots, U_n^{s,k}), \label{eqn:newton_linear_system}\\ 
 \quad J_{ij}^{k} :=& \pd{}{U_j}( F_i^{s,k}(U_1^{s}, U_2^{s}, \dots, U_n^{s})),
\end{align}
where  $U_j^{s, k}$ is a vector of all degrees of freedom at the $k^{\text{th}}$ time step and at the
 $s^{\text{th}}$ Newton iteration. $\delta U_j^{k}$ is a vector of the ``perturbation" in the degrees of freedom from the previous Newton iteration.  An initial guess $U_i^{0,k}$ must be provided to start the iteration.  $J_{ij}^{k}$ is a Jacobian matrix (very similar to the gradient term in the 1D root finding Newton's algorithm).  $F_i^{s,k}$ is the function of the degrees of freedom at the $s^{\text{th}}$ Newton iteration which is being minimized.  One can calculate $J_{ij}^{k}$ either numerically using finite differences or analytically.  We calculate $J_{ij}^{k}$ analytically by calculating the variations (partial differentials) of the operators with respect to the degrees of freedoms.  Using this technique, $U_j^{s, k}$ can be updated as follows until the desired tolerances are reached:
\begin{align}
U_j^{s+1, k} = U_j^{s, k} + \delta U_j^{k}.
\end{align}

In the time-steping context the solution vector at the previous time step can be used to initiate the Newton iteration at each timestep. Here \cref{eqn:newton_linear_system} is the linear system which has to be solved at each Newton iteration  on a massively parallel scale for two sets of PDEs working in a block iteration setup.  In order to handle the Newton iterations and the embedded linear solves, we make use of the  {\sc petsc}  library, which provides parallel efficient implementations of the above ideas along with a large suite of preconditioners and solvers for the linear system \citep{petsc-efficient,petsc-web-page,petsc-user-ref}.  The choice of linear solvers and preconditioner is different for different numerical experiments and more details are provided in the respective sections for those results. 


\section{Octree based finite element discretisation and remeshing}
\label{sec:octree_mesh}

While the concept of adaptive space partitions is well studied, developing such methods for applications demanding frequent refinements on large distributed systems presents significant challenges. This work builds on existing methods for performing large-scale finite element computations using octree-refined meshes. 
The octree-based framework, \dendro\, is extended to support sub-domains, primarily with the objective of supporting long channels and division of the domain based on arbitrary functions that define the geometry.
We provide a brief description on building the octree mesh in parallel and performing finite element computations. Additional details can be found in~\cite{SundarSampathBiros08}. 
\dendro\ provides the adaptive mesh refinement (AMR) and all parallel data-structures, and for this project, \dendro\ was extended to support domains that are not cuboidal in shape. We give a brief overview of the \dendro\ framework and provide details on the new contributions. The main steps in building and maintaining an adaptively refined mesh in a distributed-memory machine are described below.

\paragraph{Refinement:} The sparse grid is constructed based on the geometry. Proceeding in a top-down fashion, a cell is refined if a surface (defined by a zero level-set of a field, or a cloud of points)  passes through it. 
We also provide an additional function that tests for membership and eliminates regions outside the domain. This is necessary as by definition the octree maps to a cuboidal domain. By eliminating regions, we can support arbitrary domains, including domains with holes, such a porous media. For long channels, such a pruning of the octree mesh is preferable to stretching the domain, as it keeps the elements isotropic and results in better conditioning of the operators.

Since the refinement happens in an element-local fashion, this step is embarrassingly parallel.
The user provides a function (as a \texttt{C++} lambda function) that given coordinates, $x,y,z$ returns the distance from the surface. The eight corners of an octant are tested using this function. If all $8$ points have a positive distance (outside), then we retain this element, but do not refine further. If all $8$ points have a negative distance (inside), then this element is removed from the mesh. If some of the corners of the octant are inside and others outside, then this octant is refined. This is repeated till the desired level of refinement is achieved. 

In distributed memory, the initial top-down tree construction, also enables an efficient partitioning of the domain across an arbitrary number of processes. All processes start at the root node (i.e., the cubic bounding box for the entire domain). In order to avoid communication during the refinement stage, we opt to perform redundant computations on all processes. Starting from the root node, all processes refine (similar to the sequential code) until at least $\mathcal{O}(p)$ octants requiring further refinement are produced. Then using a weighted space-filling-curve (SFC) based partitioning, we partition the octants across all processes. Note that we do not communicate the octants as every process has a copy of the octants, and all that needs to be done at each process is to retain a subset of the current octants corresponding to its sub-domain. This allows us to have excellent scalability, as all processes perform (roughly) the same amount of work without requiring any communication. The SFC-based partitioning also ensures load balancing for subsequent stages and minimizes data-dependencies from the resulting partition. See \citet{SundarSampathAdavaniEtAl07, FernandoDuplyakinSundar17} for additional details on the tree construction and partitioning.
%
This produces a non-cuboidal octree that is refined to the geometry. For geometry that are stretched along certain directions, or domains with large voids/holes, eliminating regions keeps the overall problem sizes small, without adversely affecting the conditioning of the system. 
Refinement is followed by enforcing a constraint on the size of neighbouring elements, called 2:1 balancing. This is important to ensure that the neighborhood maps and data-structures are bounded, and also maximizes the sparsity of the assembled matrices.  

\paragraph{2:1 Balancing:} We enforce a condition in our distributed octrees that no two neighboring octants differ in size by more than a factor of two. This makes subsequent operations simpler without affecting the adaptive properties. Our balancing algorithm is similar to existing approaches for balancing octrees~\citep{bern1999parallel,BursteddeWilcoxGhattas11,SundarSampathAdavaniEtAl07} with the added aspect that it does not generate octants if the ancestor does not exist in the input. This is done to ensure that regions that were previously eliminated are not filled in. The algorithm proposed by~\citet{bern1999parallel} is easily extensible to support this case, as we simply need to skip adding balancing octants that violate the criteria.
The basic idea is to visit each element and generate balancing octants, i.e., neighbouring octants that are larger without violating the balance condition. These can be generated (locally) in an efficient manner by using an efficient ordered set data-structure like AVL-trees~\citep{Knuth1973Art}. In distributed memory, we sort the octants according to the Morton order and remove duplicate octants to obtained the balanced octree.

\paragraph{Partition:} Refinement and the subsequent 2:1 balancing of the octree can result in a non-uniform distribution of elements across the processes, leading to load imbalance. This is particularly challenging when arbitrary geometries are meshed, as this can make the mesh heavily load-imbalanced.  The Morton ordering enables us to equipartition the elements by performing a parallel scan on the number of elements on each process followed by point-to-point communication to redistribute the elements.  As we refine near the two-phase interface, it can affect the performance, as it is likely localized on a small subset of processes, this where Morton ordering comes to rescue and delivers and effective partition. The partitioning scheme is able to handle arbitrary geometries as the partition only tries to equally divide the retained elements across the processes.  The weighted partitioning, is a straightforward extension of our SFC-based partitioning that provides variable weight to the elements based on whether the element lies inside the retained domain of the arbitrary geometry or not. This allows us to more accurately estimate the work on each partition and provide better parallel load-balancing. Additional details on SFC-based partitioning and its implementation details in \dendro\ can be found in~\citet{FernandoDuplyakinSundar17}. 
 
\paragraph{Meshing:} By meshing we refer to the construction of the (numerical) data structures required for finite element computations from the (topological) octree data. \dendro ~already has efficient implementations for building the required neighborhood information and for managing overlapping domains between processors (\textit{ghost} or \textit{halo} regions). The key difference with our previous applications is the requirement to handle arbitrary geometries, as all neighbors might not be present in the mesh. This also complicates the process of applying boundary conditions. We added support for defining {\em subdomains} within \dendro. The subdomains are defined using a function that takes a coordinate $(x,y,z)$ as input and returns \texttt{true} or \texttt{false} depending on whether that coordinate is part of the subdomain or not. The subdomain leverages the core mesh data-structure and additionally defines a unique mapping for nodes that are part of the subdomain. It also keeps track of which nodes belong to subdomain boundaries. Therefore, subdomains have a small overhead and store significantly less data than the main mesh data-structure. For our target application, it is important to identify the external (domain) boundary as this dictates which elements will be retained in the domain. Therefore, the subdomain stores two \texttt{bits} to keep track of whether a node is non-boundary, or external. Additional details on the construction of the meshing-related data structures can be found in~\citet{SundarSampathBiros08}.

\paragraph{Handling hanging nodes:} While the use of quasi-structured grids such as octree-grids makes parallel meshing scalable and efficient, without sacrificing adaptivity, one challenge is to efficiently handle the resulting non-conformity. This results in so called {\em hanging nodes} occurring on faces/edges shared between unequal elements that do not represent independent degrees of freedom. In order to minimize the memory footprint and overall efficiency, the hanging nodes are not stored in \dendro. Instead, since they are constrained by the order of the elements and the non-hanging nodes on the hanging face/edge, we introduce these as temporary variables before elemental matrix assembly or matrix-vector multiplication\footnote{for Matrix-free computations} and eliminate them following the elemental operation. This is fairly straightforward given that our meshes are limited to a 2:1 balance, limiting the number of overall cases to be considered. Additional details on the handling of hanging nodes in \dendro\ can be found in~\citet{SundarSampathAdavaniEtAl07,SundarSampathBiros08}. 

\paragraph{Intergrid transfers}

An essential requirement is to adapt the spatial mesh as the interface moves across the domain. An example of the adaptive mesh refinement following the moving bubble is shown in Figure~\ref{fig:mesh_adaption}. In the distributed memory setting, this also indicates a need to repartition and rebalance the load. Every few time steps, we remesh. This is similar to the initial mesh generation and refinement, except that it is now based on the current position of the interface as well as the original geometry. 
This is followed by the 2:1 balance enforcement and meshing. Once the new mesh is generated, we transfer the velocity field from the old mesh to the new mesh using interpolation as needed. Since the intergrid transfer happens only between parent and child (for coarsening and refinement) or remains unchanged, this can be performed on the old mesh using standard polynomial interpolation, followed by a simple repartitioning based on the new mesh (Note that the use of SFCs makes this a linear shift). Additional details on efficient implementation of distributed-memory intergrid transfers across octree meshes can be found in~\citet{SundarBirosBurstedde12}.
%

\section{Numerical experiments}
\label{sec:num_exp}


\subsection{2D manufactured solutions}
\label{subsec:manfactured_soln_result}
We first compare convergence and other properties using the method of manufactured solutions. The idea of this approach  is to input a ``solution'' that satisfies solenoidality, but not necessarily the full set of evolution equations. Instead, the residual from plugging this ``solution'' into the full Cahn-Hilliard Navier-Stokes system becomes a forcing term on the right-hand side of \cref{eqn:nav_stokes_var_semi_disc} -- \cref{eqn:phi_eqn_var_semi_disc}.
We select the following ``solution'' with appropriate forcing terms: 
\begin{equation}
\begin{split}
\vec{v} &= \left( \sin(\pi x_1)\cos(\pi x_2)\sin(t), \, -\cos(\pi x_1)\sin(\pi x_2)\sin(t), \, 0 \right), \\
p &= \sin(\pi x_1)\sin(\pi x_2)\cos(t), \quad \phi = \cos(\pi x_1)\cos(\pi x_2)\sin(t),\\
\mu &= \cos(\pi x_1)\cos(\pi x_2)\sin(t).
\end{split}
\label{eq:manufac_exact}
\end{equation}
We compute numerical solutions with the following non-dimensional parameters: $Re = 10$, $We = 1$, $Cn = 1.0$, $Pe = 3.0$, and $Fr = 1.0$. The density ratio is set to be $\rho_{-}/\rho_{+} = 0.85$. We use a 2D uniform mesh with $450 \times 450$ bilinear elements (quads) for all the numerical experiments in this sub-section.  We test the numerical framework using various time-steps to check for convergence in time.  Panel (a) of \cref{fig:manufac_temporal_convergence} shows the temporal convergence of the $L^2$ errors (numerical solution compared with the manufactured solution) calculated at $t = \pi$ to allow for one complete phase with respect to time-steps. It can be clearly seen that on a log-log scale of error vs. time-step, the errors are decreasing with a slope close to two for the phase-field parameter $\phi$ which demonstrates second order convergence. For velocity the slopes taper off from close to 1.5 with decreasing time steps. We expect this tapering off at smaller time-steps due to the effect of spatial errors.  

We next conduct a spatial convergence study. We fix the time step at $\delta t = 10^{-3}$, and vary the spatial mesh resolution. Panel (b) of \cref{fig:manufac_temporal_convergence} shows the spatial convergence of $L^2$ errors (numerical solution compared with the manufactured solution) at $t = 0.2$. \Cref{tab:spatial_error_table} \todo[color=cyan]{{R1: \#1.1}} shows the errors and the rate of change of errors for varying spatial grid spacing for velocity and $\phi$; we observe second order convergence for both velocity and $\phi$.

Panel (c) of \cref{fig:manufac_temporal_convergence} shows mass conservation for an intermediate resolution simulation with $\delta t=10^{-3}$ and $175 \times 175$ elements. We plot mass drift (i.e. $\int_{\Omega} \phi(.,t)dx - \int_{\Omega}\phi(.,t=0)dx$), and we expect this value to be very close to zero as per the theoretical
prediction of \cref{prop:mass_conservation}. We observe excellent mass conservation with fluctuations of the order of $10^{-12}$ (which is close to machine precision).  

For the Navier-Stokes (NS) block we use a relative tolerance of $10^{-10}$, and for the Cahn-Hilliard (CH) block the relative tolerance is set to $10^{-10}$.  For the linear solves within each Newton iteration, for NS we use a relative tolerance of $10^{-7}$, and for CH relative tolerance is set to $10^{-7}$.  The tolerance for block iteration errors is set to $10^{-4}$ ($\text{block}_\text{tol}$ from \cref{fig:flowchart_block}).          
 
\begin{figure}
	\centering
	\begin{tikzpicture}
	\begin{loglogaxis}[width=0.45\linewidth, scaled y ticks=true,xlabel={timestep},
	ylabel={$\norm{u - u_{exact}}_{L^2(\Omega)}$},legend entries={$v_1$,$v_2$,$\phi$, $slope = 1$, $slope = 2$},
	legend style={nodes={scale=0.65, transform shape}}, legend pos=south east, title={(a)}]
	\addplot table [x={timestep},y={L2U},col sep=comma] {manufac_results.csv};
	\addplot table [x={timestep},y={L2V},col sep=comma] {manufac_results.csv};
	\addplot +[mark=triangle*] table [x={timestep},y={L2Phi},col sep=comma] {manufac_results.csv};
	\addplot +[mark=none, red, dashed] [domain=0.005:0.1]{0.1*x};
	\addplot +[mark=none, blue, dashed] [domain=0.005:0.1]{0.025*x^2};
	\end{loglogaxis}
	\end{tikzpicture}
	\begin{tikzpicture}
	\begin{loglogaxis}[width=0.45\linewidth, scaled y ticks=true,
	xlabel={Element size, $h$ (-)},
	ylabel={$\norm{u - u_{exact}}_{L^2(\Omega)}$},
	legend entries={$v_1$,$v_2$, $\phi$, $slope = 2$},legend style={nodes={scale=0.65, transform shape}}, legend pos=south east, 
	xtick = {0.005, 0.01, 0.015, 0.02},
    xticklabel={
		\pgfkeys{/pgf/fpu=true}
		\pgfmathparse{exp(\tick)}%
		\pgfmathprintnumber[fixed relative, precision=3]{\pgfmathresult}
		\pgfkeys{/pgf/fpu=false}
	},
	scaled x ticks=false
	, title={(b)}]
	\addplot table [x={h},y={L2U},col sep=comma] {manufac_results_spatial_NonInt.csv};
	\addplot table [x={h},y={L2V},col sep=comma] {manufac_results_spatial_NonInt.csv};
	\addplot +[mark=triangle*] table [x={h},y={L2Phi},col sep=comma] {manufac_results_spatial_NonInt.csv};
	\addplot +[mark=none, blue, dashed] [domain=0.005:0.02]{x^2};
	\end{loglogaxis}
	\end{tikzpicture}
	
	\begin{tikzpicture}
	\begin{axis}[width=0.45\linewidth,scaled y ticks=true, 
	xlabel={Time (-)},ylabel={$\int_{\Omega} \phi (x_i)\mathrm{d}x_i - \int_{\Omega} \phi_{0} (x_i)\mathrm{d}x_i$},legend style={nodes={scale=0.65, transform shape}}, 
	ymin=-1e-12,ymax=1e-12,
	xmin=0, xmax=6
	, title={(c)}]
	\addplot[line width=0.35mm, color=blue]table [x={time},y={TotalPhiMinusInitial},col sep=comma] {Energy_data_manufacSol_ts1e-3_150elem_conservation.csv};
	\end{axis}
	\end{tikzpicture}
	\caption{\textit{Manufactured Solution Examples} (\cref{subsec:manfactured_soln_result}): (a) Temporal convergence of the numerical scheme for the case of manufactured solutions; (b) spatial convergence of the numerical scheme for the case of manufactured solutions; (c) mass conservation for the case of manufactured solutions using $175 \times 175$ elements with time step of $10^{-3}$.} 
	\label{fig:manufac_temporal_convergence}
\end{figure}
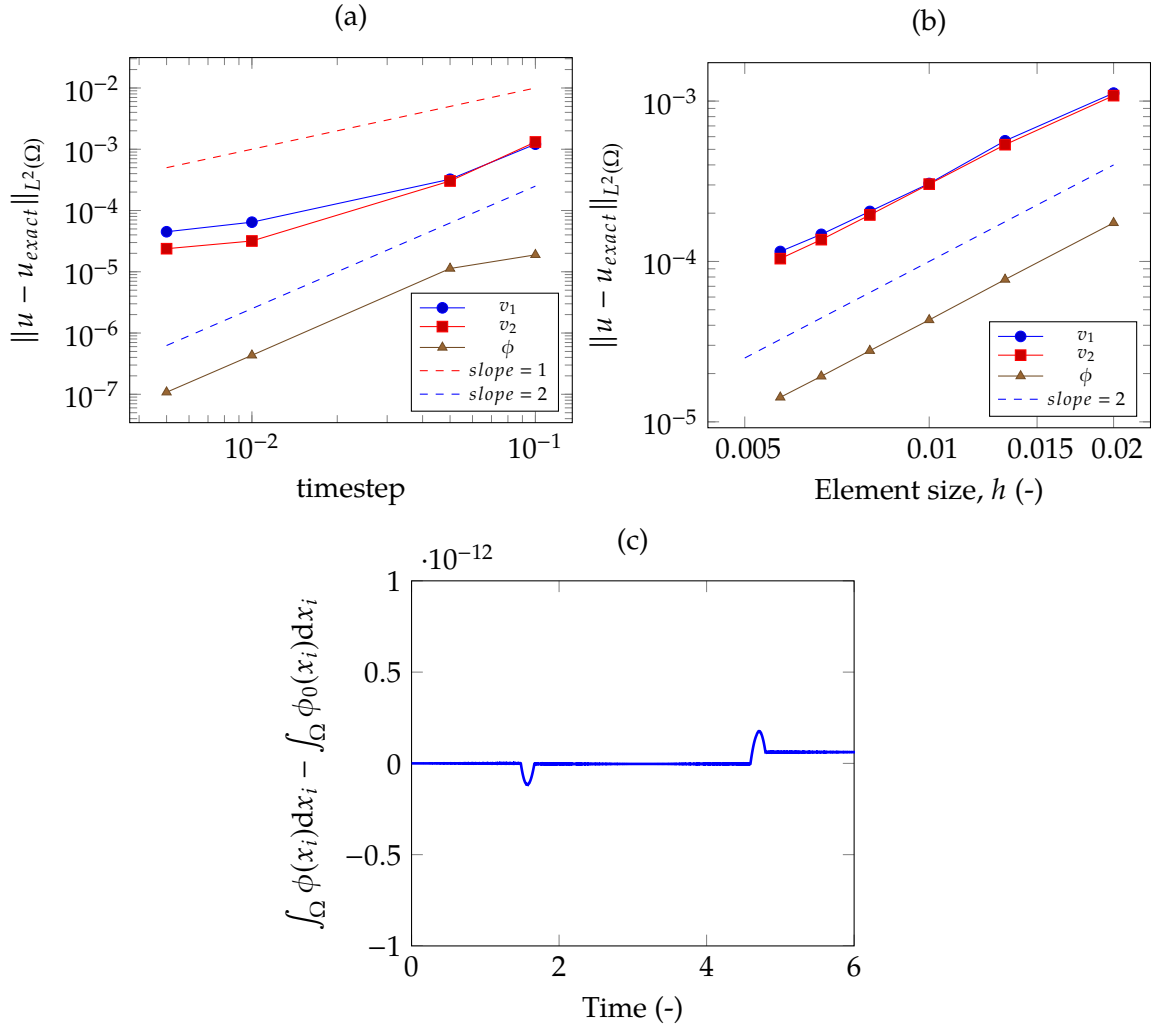

\begin{table}[H]
	\centering
	\scriptsize
	\begin{tabular}{@{}c|c|c|c|c|c|c@{}}
		\toprule
		 & \multicolumn{2}{c|}{$v_1$}   & \multicolumn{2}{|c|}{$v_2$} & \multicolumn{2}{|c}{$\phi$}   \\ 
		\cmidrule(lr){2-3}
		\cmidrule(lr){4-5}
		\cmidrule(lr){6-7}
		{Resolution ($h$)}  & $\norm{v_1 - v_{1,exact}}_{L^2(\Omega)}$  & Rate  & $\norm{v_2 - v_{2,exact}}_{L^2(\Omega)}$  & Rate & $\norm{\phi - \phi_{exact}}_{L^2(\Omega)}$ & Rate  \\
		\midrule
		\midrule
		{$1/50$}  & {$\num{1.12146E-3}$}  & {-}  & {$\num{1.08111E-3}$}  &{-} & {$\num{1.74074E-4}$}     &{-} \\
		{$1/75$}  & {$\num{5.65573E-4}$}  & {1.678}  & {$\num{5.34651E-4}$}  &{1.726} & {$\num{7.71813e-05}$}     &{1.9936} \\
		{$1/100$}  & {$\num{3.06388E-4}$}  & {2.149}  & {$\num{3.03545E-4}$}  &{1.985} & {$\num{4.31879e-05}$}     &{2.0359} \\
		{$1/125$}  & {$\num{2.05091E-4}$}  & {1.798}  & {$\num{1.94995E-4}$}  &{1.983} & {$\num{2.77703e-05}$}     &{1.979}  \\ 
		{$1/150$}  & {$\num{1.47591E-4}$}  & {1.804}  & {$\num{1.47591E-4}$}  &{1.945} & {$\num{1.92326e-05}$}     &{2.0149}  \\
		{$1/175$}  & {$\num{1.15397E-4}$}  & {1.5963}  & {$\num{1.15397E-4}$}  &{1.768} & {$\num{1.42145e-05}$}     &{1.9614}  \\ \bottomrule
	\end{tabular}
	\caption{\textit{Manufactured Solution Examples} (\cref{subsec:manfactured_soln_result}): Spatial convergence of the numerical scheme for the case of manufactured solutions with time step of $10^{-3}$.}
	\label{tab:spatial_error_table}                            
\end{table}

\subsection{Single rising bubble: 2D Benchmarks}
\label{subsec:single_rising_drop_2D}

We next illustrate the framework using a canonical case of a single air bubble rising in a quiescent channel of water. This is a well studied case, and several benchmark studies have been published~\citep{Hysing2009,Aland2012,Yuan2017}.  We start with selecting appropriate scales to non-dimensionalize the problem.  We begin with setting the Froude number ($Fr = u^2/(gD)$) to 1.0, which fixes the non-dimensional velocity scale to $u = \sqrt{gD}$, where $g$ is the gravitational acceleration, and $D$ is the diameter of the bubble.  This gives a Reynolds number of $\rho_c g^{1/2}D^{3/2}/\mu_c$, where $\rho_c$ and $\mu_c$ are the specific density and specific viscosity of the continuous fluid (i.e. water) respectively. The non-dimensional group $\rho_c g^{1/2}D^{3/2}/\mu_c$ is called the Archimedes number (Ar) and is a variant of the Reynolds number; it serves as a coefficient in front of the diffusion term in the momentum equation.  Further, the same choice of non-dimensional scales leads to a Weber number ($We = \rho_c g D^2/\sigma$).  We use the density of the continuous fluid (i.e., water) to non-dimensionalize; in this case $\rho_{+} = 1$.  Further, the density ratio is given by $\rho_{+}/\rho_{-}$.  Similarly, $\nu_{+}/\nu_{-}$ is the viscosity ratio. We present results for two test cases that are popularly reported in the benchmark studies.  

\Cref{tab:physParam_bubble_rise_2D_benchmarks} shows the parameters and the corresponding non-dimensional numbers.  The bubble is centered at $(1,1)$, and since our scaling length scale is the bubble diameter, the bubble diameter for our simulations is 1.  The domain is 
$[0,2]\times[0,4]$.  
Following the benchmark studies in the literature we choose the top and bottom wall to have no slip boundary conditions and the side walls to have boundary conditions: $v_1=0$ ($x$-velocity) and $\frac{\partial v_2}{\partial x}=0$ ($y$-velocity). We use the biCGstab (bcgs) linear solver from the PETSc suite along with the Additive Schwarz (ASM) preconditioner for the linear solves in the Newton iterations (see \cref{subsec:newton_iter}).  We use a time step of $\num{2.5e-3}$ for both the test cases. 

The convergence criterion for both the test cases is as follows. For the Navier-Stokes (NS) block we use a relative tolerance of $10^{-6}$, and for the Cahn-Hilliard (CH) block the relative tolerance is set to $10^{-8}$.  
The tolerance for block iteration errors is set to $10^{-4}$ ($\text{block}_\text{tol}$ from \cref{fig:flowchart_block}).

\begin{table}[H]
	\centering\normalsize\setlength\tabcolsep{5pt}
	\begin{tabular}{@{}|c|c|c|c|c|c|c|c|c|c|c|c|@{}}
		\toprule
		Test Case  & $\rho_{c}$  & $\rho_{b}$  & $\mu_{c}$  & $\mu_{b}$ & $\rho_{+}/\rho_{-}$ & $\nu_{+}/\nu_{-}$  & $g$ & $\sigma$ & $Ar$ & $We$ & $Fr$ \\
		\midrule
		\midrule
		{$1$}  & {1000}  & {100}  & {10}  &{1.0} & {10}     &{10} & {0.98}    & {24.5}  &  {35}     & {10} & {1.0}    \\
		{$2$}  & {1000}  & {1.0}  & {10}  &{0.1} & {1000}     &{100} & {0.98}    & {1.96}  &  {35}     & {125} & {1.0}    \\
		\bottomrule
	\end{tabular}
	\caption{Physical parameters and corresponding non-dimensional numbers for the 2D single rising drop  benchmarks considered
	in \cref{subsec:single_rising_drop_2D}.}
	\label{tab:physParam_bubble_rise_2D_benchmarks}                            
\end{table}

\subsubsection{Test case 1}
\label{subsubsec:single_rising_drop_2D_t1}
This test case considers the effect of higher surface tension and consequently less deformation of the bubble as it rises. 
We compare the bubble shape in~\cref{fig:test_case1} with benchmark quantities presented in three previous studies~\citep{Hysing2009,Aland2012,Yuan2017}.  We take $Cn=\num{5e-3}$ for this case.  Panel (a) of~\cref{fig:test_case1} shows a shape comparison against benchmark studies in the literature, and we see an excellent agreement in the shape of the bubble.  Panel (b) of~\cref{fig:test_case1} shows a comparison of centroid locations with respect to time against benchmark studies in the literature; again we see excellent agreement. We can see from the magnified inset in panel (b) of~\cref{fig:test_case1} that as we keep increasing the mesh resolution the plot approaches the benchmark studies, and we see an almost exact overlap between the benchmark and cases with $h = 2/400$ and $h = 2/600$, where $h$ is the size of the element, demonstrating spatial convergence. 

We next check whether the numerical method follows the theoretical energy stability proved in \cref{th:energy_stability}.  We present the evolution of the energy functional defined in \cref{eqn:energy_functional} for test case 1. Panel (c) of~\cref{fig:test_case1}~shows that the energy is decreasing in accordance with the energy stability condition for all three spatial resolutions of $h = 2.0/200$, $h = 2.0/400$, and $h = 2.0/600$. 

Finally we check the mass conservation. Panel (d) shows the total mass of the system minus the initial mass. At all reported spatial resolutions the change in the total mass is of the order of $10^{-8}$, even after 1600 time steps.  It is clear that the numerical method follows excellent mass conservation across long time horizons. 

\begin{figure}[H]
	\centering
	\begin{tikzpicture}
	\begin{axis}[width=0.55\linewidth,scaled y ticks=true,xlabel={$\mathrm{x}$},ylabel={$\mathrm{y}$},legend style={nodes={scale=0.65, transform shape}}, ymin=0.0, ymax=4.0, ytick distance=1.0,  xtick={0.0, 1.0, 2}, title={(a)},
	legend style={nodes={scale=0.95, transform shape}, row sep=2.5pt},
	legend entries={present study, \citet{Hysing2009}, \citet{Aland2012}},
	legend pos= north west,
	legend image post style={scale=1.0},
	unit vector ratio*=1 1 1,
	xmin=0, xmax=2, 
	legend image post style={scale=3.0},
	]
	\addplot [only marks,mark size = 0.5pt,color=blue, each nth point=5, filter discard warning=false, unbounded coords=discard] table [x={x},y={y},col sep=comma] {bubble_shape_Re35We10.csv};
	\addplot [only marks,mark size = 0.5pt,color=black,each nth point=3, filter discard warning=false, unbounded coords=discard] table [x={x},y={y},col sep=comma] {Hysing_Re35We10.csv};
	\addplot [only marks,mark size = 0.5pt,color=red,each nth point=1, filter discard warning=false, unbounded coords=discard] table [x={x},y={y},col sep=comma] {Aland_Voight_Re35We10_shape.csv};
	\end{axis}
	\end{tikzpicture}
	
	\begin{tikzpicture}[spy using outlines={rectangle, magnification=3, size=1.5cm, connect spies}]
	\begin{axis}[width=0.45\linewidth,scaled y ticks=true,xlabel={Time (-)},ylabel={Centroid},legend style={nodes={scale=0.65, transform shape}}, xmin=0, xmax=4.2, ymin=0.9, ymax=2.5, ytick distance=0.5,  xtick={0.0, 1.0, 2, 3, 4}, title={(b)},
	legend style={nodes={scale=0.95, transform shape}, row sep=2.5pt},
	legend entries={\citet{Hysing2009}, \citet{Aland2012}, \citet{Yuan2017}, $h = 2.0/200$, $h = 2.0/400$, $h = 2.0/600$},
	legend pos= north west,
	legend image post style={scale=1.0}
	]
	\addplot [line width=0.15mm, color=black] table [x={time},y={centroid},col sep=comma] {Hysing_centroid_Re35We10.csv};
	\addplot [line width=0.15mm, color=red] table [x={time},y={centroid},col sep=comma] {Aland_Voight_centroid_Re35We10.csv};
	\addplot [line width=0.15mm, color=ForestGreen] table [x={time},y={centroid},col sep=comma] {Yuan_centroid_Re35We10.csv};
	\addplot+[mark size = 0.5pt]table [x={time},y={centroid},col sep=comma, each nth point=3, filter discard warning=false, unbounded coords=discard] {centerOfMass_200_400_Re35We10.csv};	
	\addplot+[mark size = 0.5pt]table [x={time},y={centroid},col sep=comma, each nth point=3, filter discard warning=false, unbounded coords=discard] {centerOfMass_400_800_Re35We10.csv};
	\addplot+[mark size = 0.5pt]table [x={time},y={centroid},col sep=comma, each nth point=3, filter discard warning=false, unbounded coords=discard] {centerOfMass_600_1200_Re35We10.csv};
	\coordinate (a) at (axis cs:3.6,2.0);
	\end{axis}
	\spy [black] on (a) in node  at (5,1.2);
	\end{tikzpicture}
	\hskip 5pt
	\begin{tikzpicture}
	\begin{axis}[width=0.45\linewidth,scaled y ticks=true,xlabel={Time (-)},ylabel={$E_{tot}(\vec{v},\phi,t)$},legend style={nodes={scale=0.65, transform shape}}, xmin=0, xmax=4.2, xtick={0,1,2,3,4},  title={(c)},
	legend style={nodes={scale=0.95, transform shape}, row sep=2.5pt},
	legend entries={$h = 2.0/200$, $h = 2.0/400$, $h = 2.0/600$, $h = 2.0/800$},
	legend pos= north east,
	legend image post style={scale=1.0}]
	\addplot +[mark size = 1pt, each nth point=20, filter discard warning=false, unbounded coords=discard] table [x={time},y={TotalEnergy},col sep=comma] {Energy_data_200_400_Re35We10.csv};
	\addplot +[mark size = 1pt, each nth point=20, filter discard warning=false, unbounded coords=discard] table [x={time},y={TotalEnergy},col sep=comma] {Energy_data_400_800_Re35We10.csv};
	\addplot +[mark size = 1pt, each nth point=20, filter discard warning=false, unbounded coords=discard] table [x={time},y={TotalEnergy},col sep=comma] {Energy_data_600_1200_Re35We10.csv};
	\addplot +[mark size = 1pt, each nth point=20, filter discard warning=false, unbounded coords=discard] table [x={time},y={TotalEnergy},col sep=comma] {Energy_data_800_1600_Re35We10.csv};
	\end{axis}
	\end{tikzpicture}

	\begin{tikzpicture}
	\begin{axis}[width=0.45\linewidth,scaled y ticks=true,xlabel={Time (-)},ylabel={$\int_{\Omega} \phi (x_i)\mathrm{d}x_i - \int_{\Omega} \phi_{0} (x_i)\mathrm{d}x_i$},legend style={nodes={scale=0.65, transform shape}}, xmin=0, xmax=4.2, xtick={0,1,2,3,4}, title={(d)},
	legend style={nodes={scale=0.95, transform shape}, row sep=2.5pt},
	legend entries={$h = 2.0/200$, $h = 2.0/400$, $h = 2.0/600$, $h = 2.0/800$},
	legend pos= north east,
	legend image post style={scale=1.0},
	ymin=-9e-8,ymax=9e-8]
	\addplot+[mark size = 0.75pt, each nth point=20, filter discard warning=false, unbounded coords=discard]table [x={time},y={TotalPhiNorm},col sep=comma] {Energy_data_200_400_Re35We10.csv};
	\addplot+[mark size = 0.75pt, each nth point=20, filter discard warning=false, unbounded coords=discard]table [x={time},y={TotalPhiNorm},col sep=comma] {Energy_data_400_800_Re35We10.csv};
	\addplot+[mark size = 0.75pt, each nth point=20, filter discard warning=false, unbounded coords=discard]table [x={time},y={TotalPhiNorm},col sep=comma] {Energy_data_600_1200_Re35We10.csv};
	\addplot+[mark size = 0.75pt, each nth point=20, filter discard warning=false, unbounded coords=discard]table [x={time},y={TotalPhiNorm},col sep=comma] {Energy_data_800_1600_Re35We10.csv};
	\end{axis}
	\end{tikzpicture}
	\caption{ \textit{2D Single Rising Drop Test Case 1}
	(\cref{subsubsec:single_rising_drop_2D_t1}): (a) Comparison of the computed bubble shape against results
	from the literature at non-dimensional time $T = 4.2$; (b) comparison of the rise of the bubble centroid against results from the literature; (c) decay of the energy functional illustrating theorem \ref{th:energy_stability}; and (d) total mass conservation (integral of total $\phi$).}
	\label{fig:test_case1}
\end{figure}
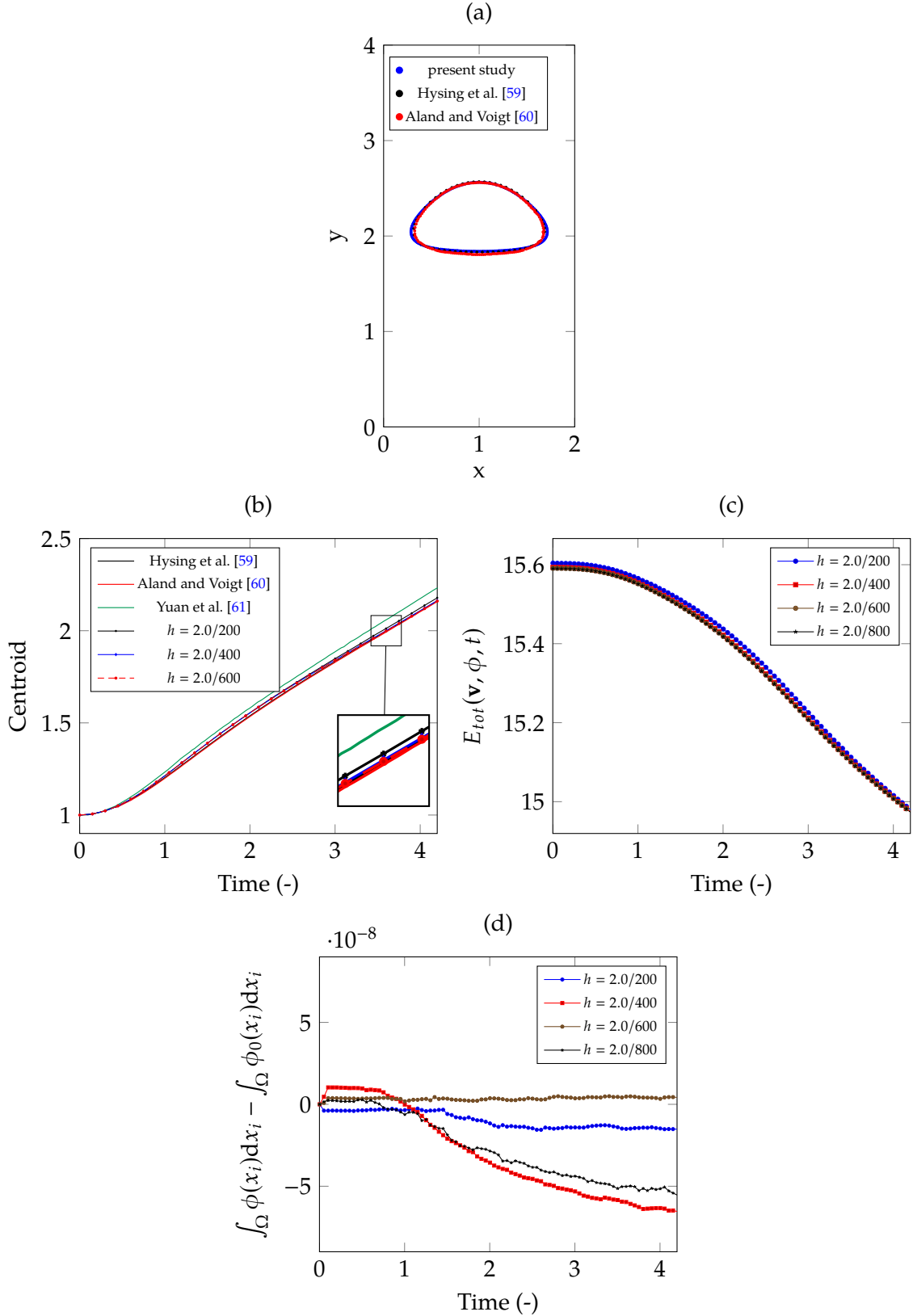

\subsubsection{Test case 2}
\label{subsubsec:single_rising_drop_2D_t2}
This test case considers a lower surface tension resulting in high deformations of the bubble as it rises. 
As before, we compare the bubble shape in~\cref{fig:test_case2} with benchmark quantities presented in three previous studies~\citep{Hysing2009,Aland2012,Yuan2017}.  Panel (a) shows the shape comparison with benchmark studies in the literature and we see an excellent agreement in the shape of the bubble. All simulations (our results and benchmarks) exhibit a skirted bubble shape. We see an excellent match in the overall shape of the bubble with some minor differences in the dynamics of the tails of the bubble. Specifically, we see that the tails of bubble in our case pinch-off to form satellite bubbles\footnote{Such instabilities require a very low $Cn$ number, as a very thin interfacial thickness is required to capture the dynamics of the thin tails of the bubble}.  We performed this simulation with a $Cn=0.0025$ and three different spatial resolutions.  We can see in panel (a) of~\cref{fig:test_case2} that our simulation captures this filament pinch off in the tails very well.  It is important to note that \citet{Aland2012, Yuan2017} did not observe these thin tails and 
pinch-off, whereas \citet{Hysing2009} did observe pinch-off of the tails and satellite bubbles.  The dynamics of this bubble tail is highly dependent on the numerical method used. 

Panel (b) of~\cref{fig:test_case2} shows comparison of centroid location with respect to time. Again we see an excellent agreement with all three previous benchmark studies.  We can see from the magnified inset in panel (b) of~\cref{fig:test_case2} that as we keep increasing the mesh resolution the plot approaches the benchmark studies and we see an almost exact overlap between the benchmark and cases with $h = 2/1000$ and $h = 2/2000$ demonstrating spatial convergence. 
Next, we report the evolution of the energy functional defined in \cref{eqn:energy_functional} for test case 2.  Panel (c) of~\cref{fig:test_case2}~shows the decay of the total energy functional in accordance with the energy stability condition for all three spatial resolutions of $h = 2.0/800$, $h = 2.0/1000$, and $h = 2.0/1200$. Finally, panel (d) of~\cref{fig:test_case2} shows the total mass of the system in comparison with the total initial mass of the system. We can see that for all spatial resolution the change in the total mass with respect to the initial total mass is of the order of \num{1e-8}. This illustrates that the numerical method satisfies mass conservation over long time horizons.

\begin{figure}[H]
	\centering
	\begin{tikzpicture}
	\begin{axis}[width=0.55\linewidth,scaled y ticks=true,xlabel={$\mathrm{x}$},ylabel={$\mathrm{y}$},legend style={nodes={scale=0.65, transform shape}}, ymin=0, ymax=4, ytick distance=1.0,  xtick={0.0, 1.0, 2}, title={(a)},
	legend style={nodes={scale=0.95, transform shape}, row sep=2.5pt},
	legend entries={present study, \citet{Hysing2009}, \citet{Aland2012}, \citet{Yuan2017}},
	legend pos= north west,
	legend image post style={scale=1.0},
	unit vector ratio*=1 1 1,
	xmin=0, xmax=2, 
	legend image post style={scale=3.0}
	]
	\addplot [only marks,mark size = 0.5pt,color=blue, each nth point=2, filter discard warning=false, unbounded coords=discard] table [x={x},y={y},col sep=comma] {bubble_shape_Re35We125.csv};
	\addplot [only marks,mark size = 0.5pt,color=black,each nth point=1, filter discard warning=false, unbounded coords=discard] table [x={x},y={y},col sep=comma] {Hysing_Re35_We125_bubble_shape.csv};
	\addplot [only marks,mark size = 0.5pt,color=red,each nth point=1, filter discard warning=false, unbounded coords=discard] table [x={x},y={y},col sep=comma] {Aland_Voight_Re35_We125_bubble_shape.csv};
	\addplot [only marks,mark size = 0.5pt,color=ForestGreen,each nth point=1, filter discard warning=false, unbounded coords=discard] table [x={x},y={y},col sep=comma] {Yuan_Re35_We125_bubble_shape.csv};
	\end{axis}
	\end{tikzpicture}
	
	\begin{tikzpicture}[spy using outlines={rectangle, magnification=3, size=1.5cm, connect spies}]
	\begin{axis}[width=0.45\linewidth,scaled y ticks=true,xlabel={Time (-)},ylabel={Centroid},legend style={nodes={scale=0.65, transform shape}}, xmin=0, xmax=4.2, ymin=0.9, ymax=2.8, ytick distance=0.5,  xtick={0.0, 1.0, 2, 3, 4}, title={(b)},
	legend style={nodes={scale=0.95, transform shape}, row sep=2.5pt},
	legend entries={\citet{Hysing2009}, \citet{Aland2012}, \citet{Yuan2017}, $h = 2.0/800$, $h = 2.0/1000$, $h = 2.0/1200$},
	legend pos= north west,
	legend image post style={scale=1.0}
	]
	\addplot [line width=0.1mm, color=black] table [x={time},y={centroid},col sep=comma] {Hysing_centroid_Re35We125.csv};
	\addplot [line width=0.1mm, color=red] table [x={time},y={centroid},col sep=comma] {aland_voight_centroid_Re35We125.csv};
	\addplot [line width=0.1mm, color=ForestGreen] table [x={time},y={centroid},col sep=comma] {Yuan_centroid_Re35We125.csv};
	\addplot+[mark size = 0.75pt]table [x={time},y={centroid},col sep=comma, each nth point=3, filter discard warning=false, unbounded coords=discard] {centerOfMass_800_1600.csv};	
	\addplot+[mark size = 0.75pt]table [x={time},y={centroid},col sep=comma, each nth point=3, filter discard warning=false, unbounded coords=discard] {centerOfMass_1000_2000.csv};
	\addplot+[mark size = 0.75pt]table [x={time},y={centroid},col sep=comma, each nth point=3, filter discard warning=false, unbounded coords=discard] {centerOfMass_1200_2400.csv};
	\coordinate (a) at (axis cs:3.6,2.05);
	\end{axis}
	\spy [black] on (a) in node  at (5,1.2);
	\end{tikzpicture}
	\hskip 5pt
	\begin{tikzpicture}
	\begin{axis}[width=0.45\linewidth,scaled y ticks=true,xlabel={Time (-)},ylabel={$E_{tot}(\vec{v},\phi,t)$},legend style={nodes={scale=0.65, transform shape}}, xmin=0, xmax=4.2, xtick={0,1,2,3,4},  title={(c)},
	legend style={nodes={scale=0.95, transform shape}, row sep=2.5pt},
	legend entries={$h = 2.0/800$, $h = 2.0/1000$, $h = 2.0/1200$},
	legend pos= north east,
	legend image post style={scale=1.0}]
	\addplot +[mark size = 0.75pt, each nth point=20, filter discard warning=false, unbounded coords=discard] table [x={time},y={TotalEnergy},col sep=comma] {Energy_data_800_1600.csv};
	\addplot +[mark size = 0.75pt, each nth point=20, filter discard warning=false, unbounded coords=discard] table [x={time},y={TotalEnergy},col sep=comma] {Energy_data_1000_2000.csv};
	\addplot +[mark size = 0.75pt, each nth point=20, filter discard warning=false, unbounded coords=discard] table [x={time},y={TotalEnergy},col sep=comma] {Energy_data_1200_2400.csv};
	\end{axis}
	\end{tikzpicture}

	\begin{tikzpicture}
	\begin{axis}[width=0.45\linewidth,scaled y ticks=true,xlabel={Time (-)},ylabel={$\int_{\Omega} \phi (x_i)\mathrm{d}x_i - \int_{\Omega} \phi_{0} (x_i)\mathrm{d}x_i$},legend style={nodes={scale=0.65, transform shape}}, xmin=0, xmax=4.2, xtick={0,1,2,3,4}, title={(d)},
	legend style={nodes={scale=0.95, transform shape}, row sep=2.5pt},
	legend entries={$h = 2.0/800$, $h = 2.0/1000$, $h = 2.0/1200$},
	legend pos= north west,
	legend image post style={scale=1.0},
	ymin=-9.5e-8,ymax=9.5e-8]
	\addplot+[mark size = 0.75pt, each nth point=20, filter discard warning=false, unbounded coords=discard]table [x={time},y={TotalPhiNorm},col sep=comma] {Energy_data_800_1600.csv};
	\addplot+[mark size = 0.75pt, each nth point=20, filter discard warning=false, unbounded coords=discard]table [x={time},y={TotalPhiNorm},col sep=comma] {Energy_data_1000_2000.csv};
	\addplot+[mark size = 0.75pt, each nth point=20, filter discard warning=false, unbounded coords=discard]table [x={time},y={TotalPhiNorm},col sep=comma] {Energy_data_1200_2400.csv};
	\end{axis}
	\end{tikzpicture}
		\caption{\textit{2D Single Rising Drop Test Case 2} 
	(\cref{subsubsec:single_rising_drop_2D_t2}): (a) Comparison of the computed bubble shape against results
	from the literature at non-dimensional time $T = 4.2$; (b) comparison of the rise of the bubble centroid against results from the literature; (c) decay of the energy functional illustrating theorem \ref{th:energy_stability}; and (d) total mass conservation (integral of total $\phi$).}
	\label{fig:test_case2}
\end{figure}


Another important aspect we report is the mass conservation of the lighter bubble. The mass of the bubble is a small fraction ($\sim 10\%$) of the total volume and is thus a more stringent test of mass conservation. We plot mass loss (mass of bubble at time t - initial mass of bubble) as a function of time in \cref{fig:BR_massLoss} for both test cases.
For test case 1 panel (a) of~\cref{fig:BR_massLoss} shows the mass loss of bubble as a function of time and we can see that for all four spatial resolutions, the change in mass of the bubble compared to initial mass of the bubble\footnote{we calculate this by integrating $\phi < 0$ over the domain} is within the bound of $\pm \num{5e-3}$, which shows that the mass loss of the bubble is less than 0.5\% over this long time horizon.  For test 2, on the other hand, mass conservation of the bubble is typically difficult to capture because of the pinch off and formation of small satellite bubbles.  Even for this challenging case (panel (b) of~\cref{fig:BR_massLoss}), the change in mass of the bubble compared to initial mass of the bubble is within the bound of $\pm \num{5e-3}$ which is again less than a 0.5\% change.  

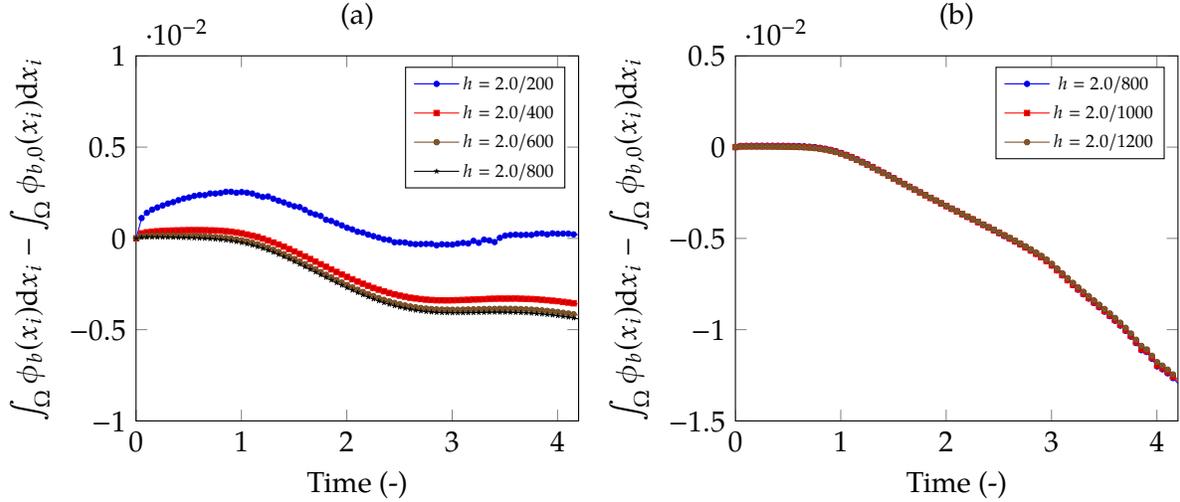
\begin{figure}[]
	\centering
	\begin{tikzpicture}
	\begin{axis}[width=0.45\linewidth,scaled y ticks=true,xlabel={Time (-)},ylabel={$\int_{\Omega} \phi_{b} (x_i)\mathrm{d}x_i - \int_{\Omega} \phi_{b,0} (x_i)\mathrm{d}x_i$},legend style={nodes={scale=0.65, transform shape}}, xmin=0, xmax=4.2, xtick={0,1,2,3,4},  title={(a)},
	legend style={nodes={scale=0.95, transform shape}, row sep=2.5pt},
	legend entries={$h = 2.0/200$, $h = 2.0/400$, $h = 2.0/600$, $h = 2.0/800$},
	legend pos= north east,
	legend image post style={scale=1.0},
	ymin=-1e-2,ymax=1e-2]
	\addplot +[mark size = 1pt, each nth point=20, filter discard warning=false, unbounded coords=discard] table [x={time},y={TotalPhiMinusOnePhaseNorm},col sep=comma] {Energy_data_200_400_Re35We10.csv};
	\addplot +[mark size = 1pt, each nth point=20, filter discard warning=false, unbounded coords=discard] table [x={time},y={TotalPhiMinusOnePhaseNorm},col sep=comma] {Energy_data_400_800_Re35We10.csv};
	\addplot +[mark size = 1pt, each nth point=20, filter discard warning=false, unbounded coords=discard] table [x={time},y={TotalPhiMinusOnePhaseNorm},col sep=comma] {Energy_data_600_1200_Re35We10.csv};
	\addplot +[mark size = 1pt, each nth point=20, filter discard warning=false, unbounded coords=discard] table [x={time},y={TotalPhiMinusOnePhaseNorm},col sep=comma] {Energy_data_800_1600_Re35We10.csv};
	\end{axis}
	\end{tikzpicture}
	\hskip 5pt
	\begin{tikzpicture}
	\begin{axis}[width=0.45\linewidth,scaled y ticks=true,xlabel={Time (-)},ylabel={$\int_{\Omega} \phi_{b} (x_i)\mathrm{d}x_i - \int_{\Omega} \phi_{b,0} (x_i)\mathrm{d}x_i$},legend style={nodes={scale=0.65, transform shape}}, xmin=0, xmax=4.2, xtick={0,1,2,3,4},  title={(b)},
	legend style={nodes={scale=0.95, transform shape}, row sep=2.5pt},
	legend entries={$h = 2.0/800$, $h = 2.0/1000$, $h = 2.0/1200$},
	legend pos= north east,
	legend image post style={scale=1.0},
	ymin=-1.5e-2,ymax=5e-3]
	\addplot +[mark size = 1pt, each nth point=20, filter discard warning=false, unbounded coords=discard] table [x={time},y={TotalPhiMinusOnePhaseNorm},col sep=comma] {Energy_data_800_1600.csv};
	\addplot +[mark size = 1pt, each nth point=20, filter discard warning=false, unbounded coords=discard] table [x={time},y={TotalPhiMinusOnePhaseNorm},col sep=comma] {Energy_data_1000_2000.csv};
	\addplot +[mark size = 1pt, each nth point=20, filter discard warning=false, unbounded coords=discard] table [x={time},y={TotalPhiMinusOnePhaseNorm},col sep=comma] {Energy_data_1200_2400.csv};
	\end{axis}
	\end{tikzpicture}
	\caption{\textit{Bubble Mass Conservation for the 2D Rising Drop Test Cases:} (a) Mass conservation of bubble (integral of $\phi < 0$) for test case 1 (\cref{subsubsec:single_rising_drop_2D_t1}); (b) mass conservation of bubble (integral of $\phi < 0$) for test case 2
	(\cref{subsubsec:single_rising_drop_2D_t2})}
	\label{fig:BR_massLoss}
\end{figure}

\subsection{Rayleigh-Taylor instability: 2D simulations}
\label{subsec:rayleigh_taylor_2D}
Performance of the framework at higher Reynolds numbers and large changes in the topology of the interface can be demonstrated by simulating the Rayleigh-Taylor instability.  While the bubble rise case is an interplay between surface tension and buoyancy, the physics of a Rayleigh-Taylor instability is dominated by buoyancy.  A lot of studies in the literature also switch off the surface tension forcing terms in the momentum equations (see \citep{Xie2015,Tryggvason1990,Li1996,Guermond2000} for examples).  Here, the choice of non-dimensional numbers ensures that surface tension effect is small (high Weber numbers).  The setup is as follows: the heavier fluid is placed on top of lighter fluid and the interface is perturbed. The heavier fluid on top penetrates into the lighter fluid and buckles, which generates instabilities.  This interface motion is very difficult to track in interface resolved simulations (like the current ones), as the changes in the topology of the interface are large and Rayleigh-Taylor instabilities generally encompass turbulent conditions which calls for resolving finer scales.  We non-dimensionalize  the problem by selecting the width of the channel as the characteristic length scale and the density of the lighter fluid as the characteristic specific density.  Just like in the case of bubble rise we use buoyancy-based scaling, setting the Froude number ($Fr = u^2/(gD)$) to 1.0, which fixes the non-dimensional velocity scale to be $u = \sqrt{gD}$, where $g$ is the gravitational acceleration, and $D$ is the width of the channel.  Using this velocity to calculate the Reynolds number, we get $Re = \rho_L g^{1/2}D^{3/2}/\mu_L$, where $\rho_L$ and $\mu_L$ are the specific density and specific viscosity of the light fluid respectively. We set the Reynolds number at 3000.  Furthermore, the same choice of non-dimensional scales leads to a Weber number ($We = \rho_c g D^2/\sigma$).  To compare our results with previous studies, we simulate with same initial conditions as presented in \citet{Tryggvason1988, Guermond2000, Ding2007, Xie2015}.  The We number is selected to be 1000, so that the effect of surface tension is small on the evolution of interface.  In this case, similar to the bubble rise case we have chosen specific density of the light fluid to non-dimensionalize, therefore $\rho_{+} = 1.0$.  Further, for the 2D simulations we choose two density ratios ($\rho_{+}/\rho_{-}$) of 0.33, and 0.1 respectively.  Similarly, $\nu_{+}/\nu_{-}$ is the viscosity ratio which is selected to be 1.0.  In the literature Atwood number ($At$) is often used to parametrize the dependence on density ratio which is given by $At = \left(\rho_{+} - \rho_{-}\right)/\left(\rho_{+} + \rho_{-}\right)$.  For the density ratios of 0.33, and 0.1, the Atwood numbers are $At = 0.5$, and $At = 0.82$ respectively. The boundary conditions we use are no-slip for velocity on all the walls and no flux conditions for $\phi$ and $\mu$. We assume a 90 degree wetting angle for both the fluids.  The simulations were performed using a time step of $10^{-4}$.  

\Cref{fig:rt2d} shows the snapshots of the interface shape as it evolves in time.  We can see the heavier fluid penetrating in the light fluid, and the light fluid rises up near the wall.  For both cases of $At = 0.5$ and $At = 0.82$ a number of previous studies have presented the location of top front and bottom front as a function of time.  Panel (a) \cref{fig:RT_2D_comparison} shows the comparison of locations of bottom and top front of the interface with previous studies~\citep{Xie2015,Tryggvason1990,Li1996,Guermond2000}.  We can clearly see that the results from the current study match the previous benchmarks very well.  Panel (b) from \cref{fig:RT_2D_comparison} shows the decay of energy functional in line with~\cref{th:energy_stability} for $At = 0.5$.  Panel (c) from \cref{fig:RT_2D_comparison} shows the change in the total mass with respect to the initial total mass. We observe that it is of the order of $10^{-13}$. Therefore, we see excellent mass conservation even with high amount of deformation of the interface over very large time horizons (over 30,000 time steps).

The convergence criterion for both the 2D Rayleigh-Taylor test cases are as follows.  For the Navier-Stokes (NS) block we use a relative tolerance of $10^{-6}$, and for the Cahn-Hilliard (CH) block the relative tolerance is set to $10^{-10}$.  For the linear solves within each Newton iteration, for NS we use a relative tolerance of $10^{-7}$, and for CH relative tolerance is set to $10^{-7}$. The tolerance for block iteration errors is set to $10^{-4}$ ($\text{block}_\text{tol}$ from \cref{fig:flowchart_block}).
 
\begin{figure}[]
	\centering
	\begin{tabular}{C{0.18\textwidth}C{0.18\textwidth}C{0.18\textwidth}C{0.18\textwidth}C{0.18\textwidth}}
		\subfigure [$t = 0.0$] {
			\includegraphics[width=\linewidth]{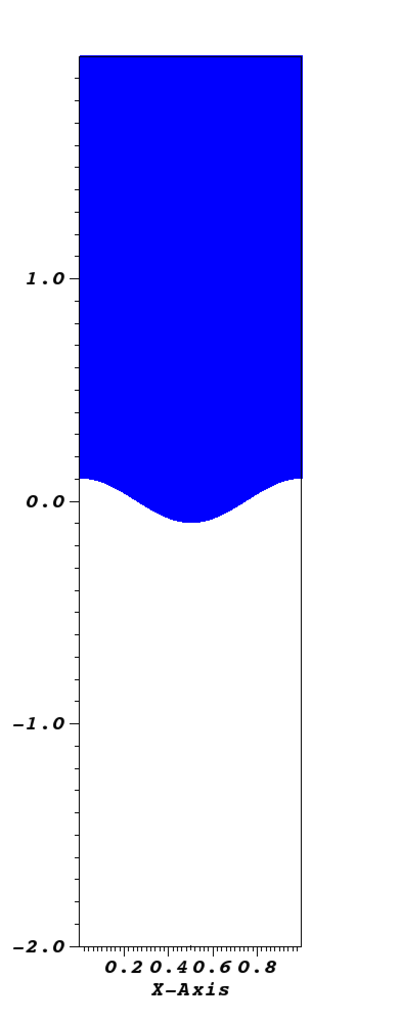}
			\label{subfig:rt_snap_1}
		} &
		\subfigure [$t = 0.8080$] {
			\includegraphics[width=\linewidth]{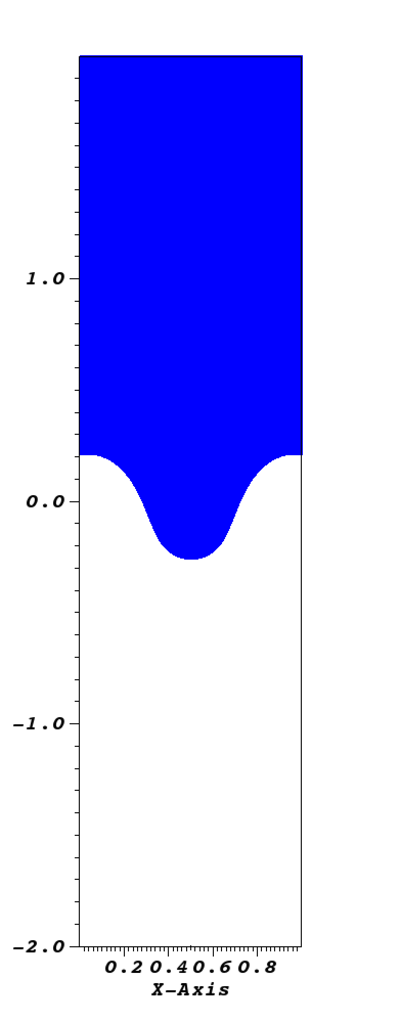}
			\label{subfig:rt_snap_2}
		} & 
		\subfigure [$t = 1.232$] {
			\includegraphics[width=\linewidth]{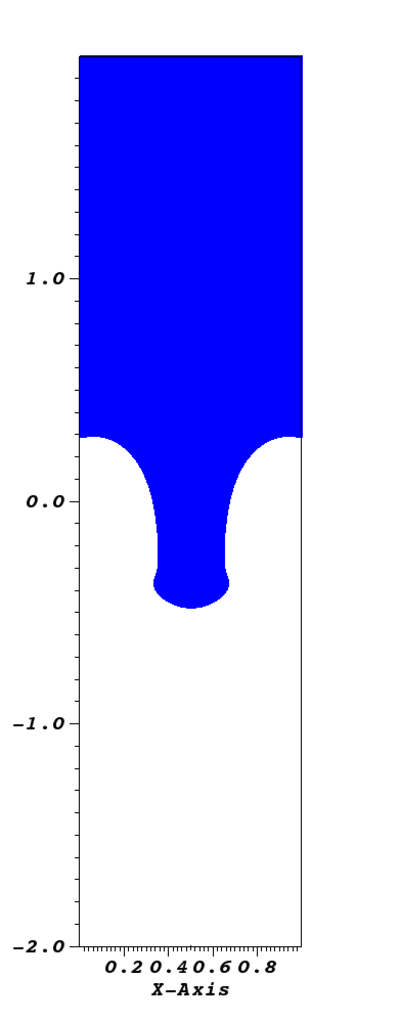}
			\label{subfig:rt_snap_3}
		} &
		
		\subfigure [$t = 1.6059$] {
			\includegraphics[width=\linewidth]{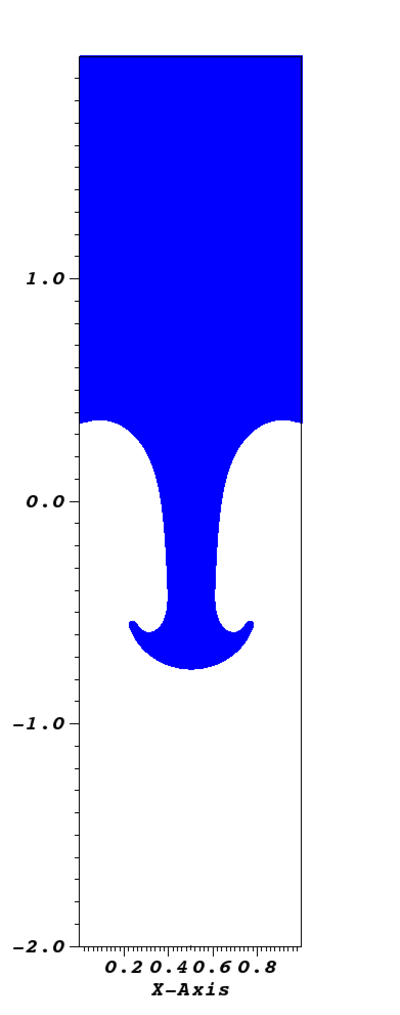}
			\label{subfig:rt_snap_4}
		} &
		\subfigure [$t = 2.0503$] {
			\includegraphics[width=\linewidth]{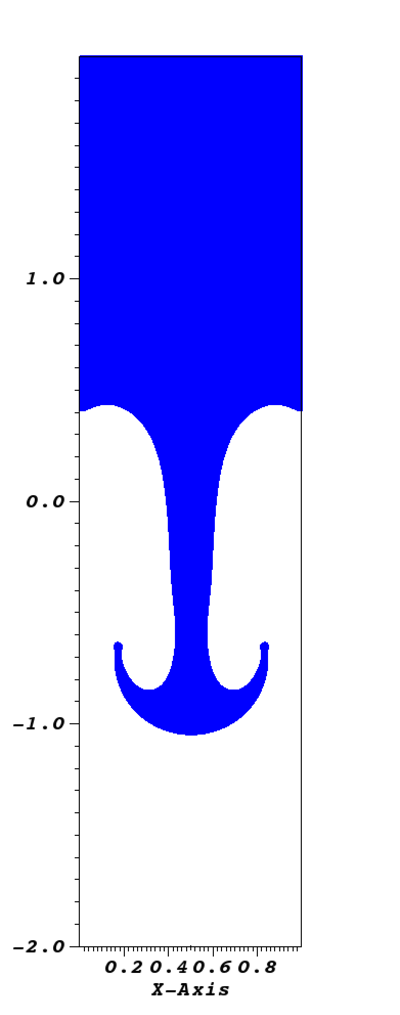}
			\label{subfig:rt_snap_5}
		} \\ 
		\subfigure [$t = 2.3533$] {
			\includegraphics[width=\linewidth]{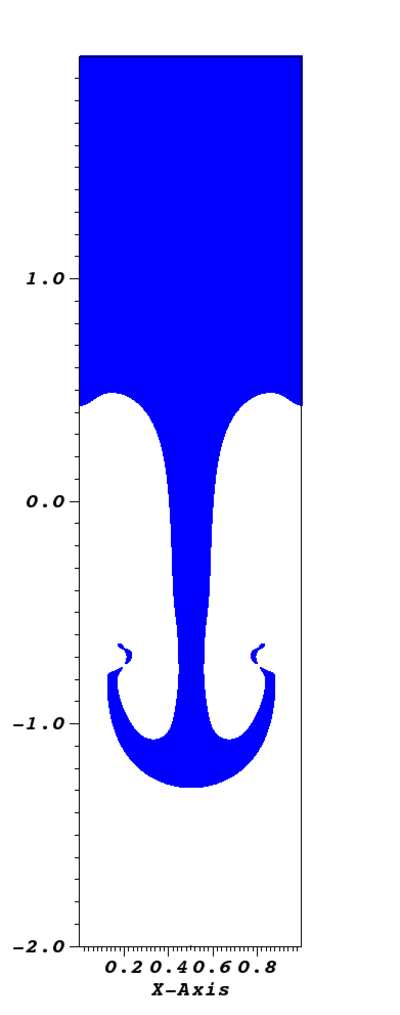}
			\label{subfig:rt_snap_6}
		} &
		
		\subfigure [$t = 2.5351$] {
			\includegraphics[width=\linewidth]{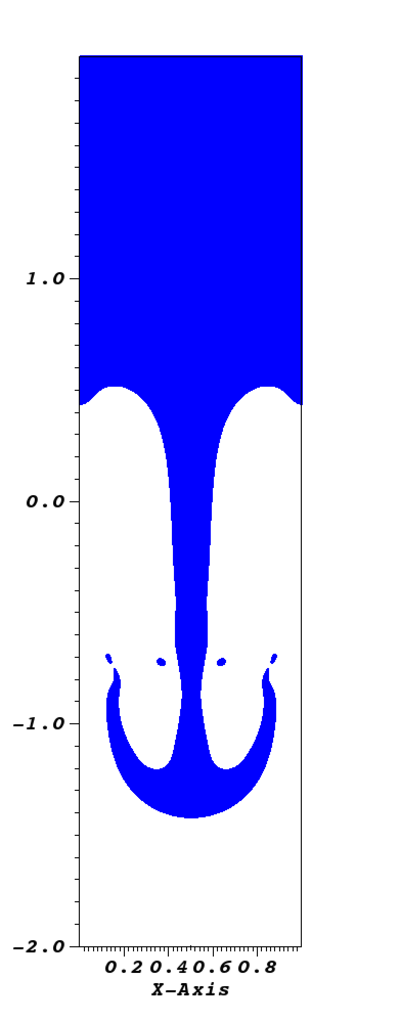}
			\label{subfig:rt_snap_7}
		} &
		\subfigure [$t = 2.9391$] {
			\includegraphics[width=\linewidth]{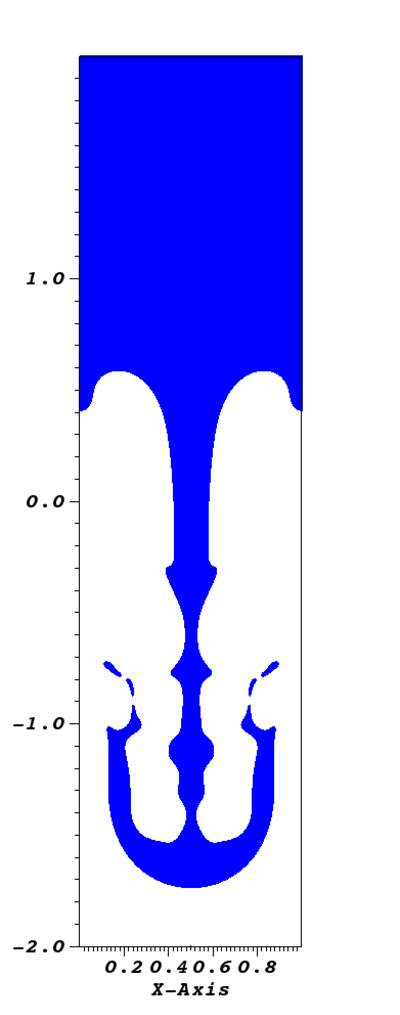}
			\label{subfig:rt_snap_8}
		} & 
		\subfigure [$t = 3.019$] {
			\includegraphics[width=\linewidth]{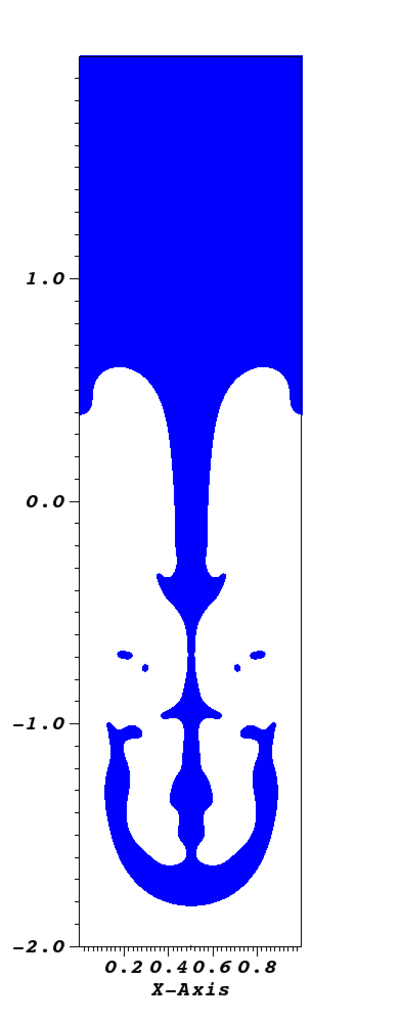}
			\label{subfig:rt_snap_9}
		} &
		\subfigure [$t = 3.2017$] {
			\includegraphics[width=\linewidth]{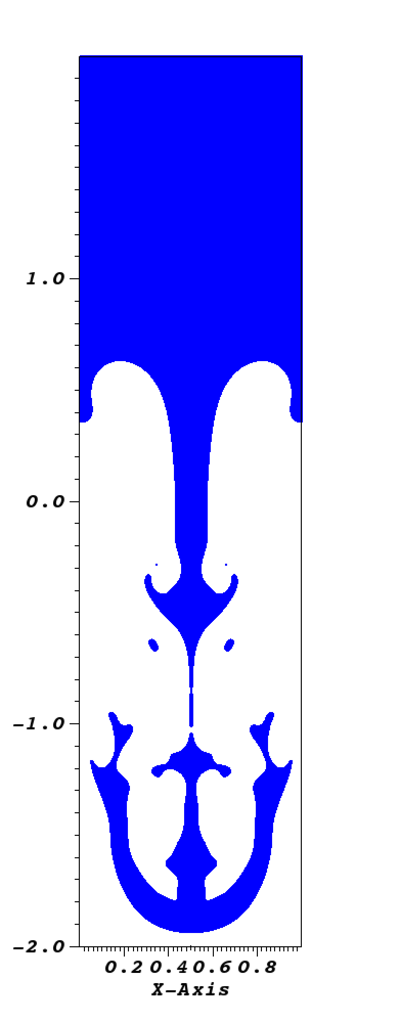}
			\label{subfig:rt_snap_10}
		}  
	\end{tabular}
	\caption{\textit{Rayleigh-Taylor instability in 2D} (\cref{subsec:rayleigh_taylor_2D}): Evolution of the interface as a function of time for $At = 0.82$ (density ratio of 0.1).}
	\label{fig:rt2d}
\end{figure}

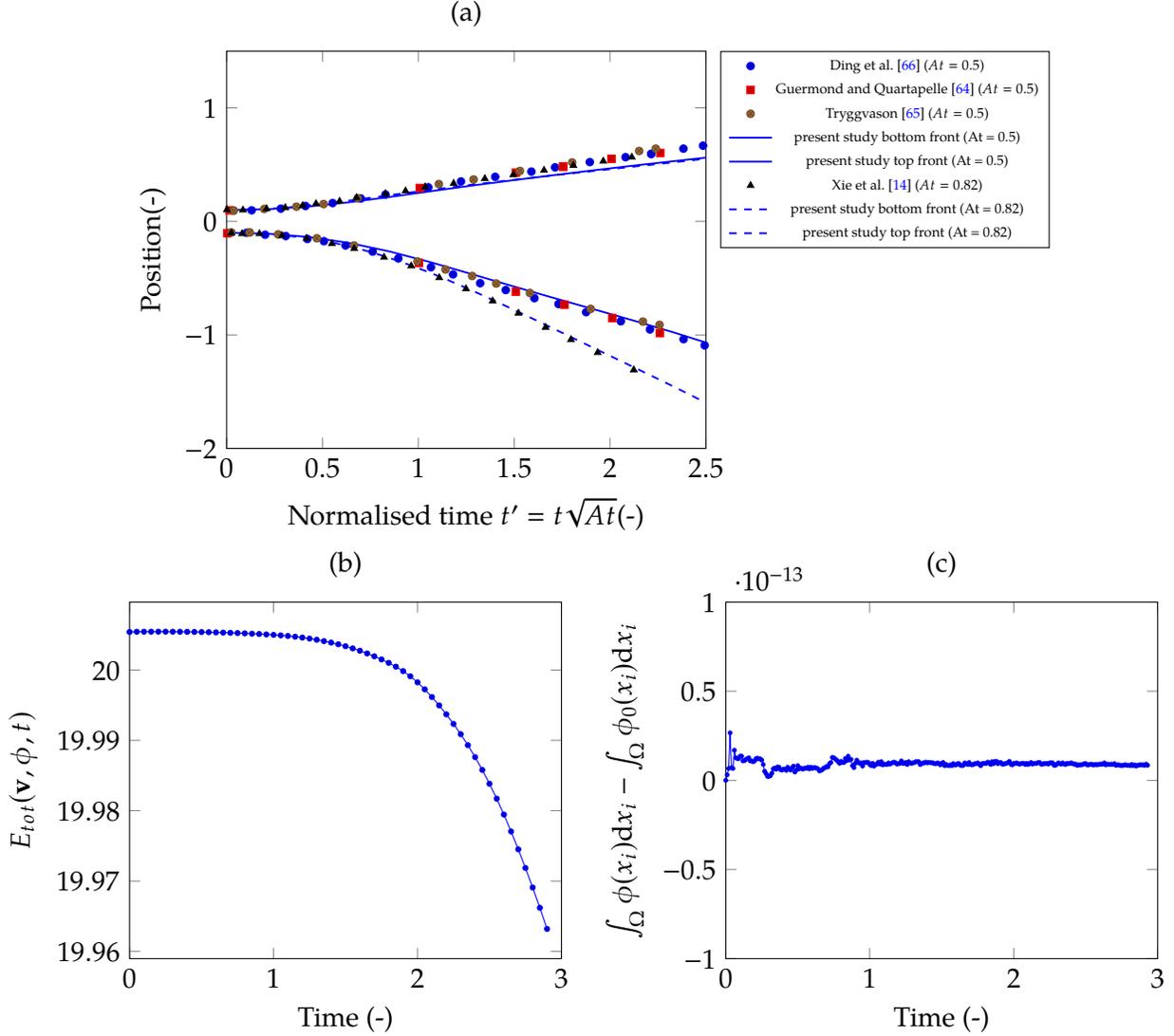
\begin{figure}[H]
	\centering	
	\begin{tikzpicture}[]
	\begin{axis}[width=0.5\linewidth,scaled y ticks=true,xlabel={Normalised time $t' = t \sqrt{At}$(-)},ylabel={Position(-)},legend style={nodes={scale=0.65, transform shape}}, xmin=0, xmax=2.5, ymin=-2, ymax=1.5, ytick distance=1.0,  xtick={0.0, 0.5, 1.0, 1.5, 2, 2.5}, title={(a)},
	legend style={nodes={scale=0.95, transform shape}, row sep=2.0pt},
	legend entries={\citet{Ding2007}~($At = 0.5$), \citet{Guermond2000}~($At = 0.5$), \citet{Tryggvason1988}~($At = 0.5$), present study bottom front (At = 0.5), present study top front (At = 0.5), \citet{Xie2015}~($At = 0.82$), present study bottom front (At = 0.82), present study top front (At = 0.82) },
	legend pos= outer north east,
	legend image post style={scale=1.0},
	legend style={font=\footnotesize}
	]
	\addplot +[only marks,mark size = 1.5pt] table [x={time},y={position},col sep=comma] {Ding_combined.csv};
	\addplot +[only marks,mark size = 1.5pt] table [x={time},y={position},col sep=comma] {Guermond_combined.csv};
	\addplot +[only marks,mark size = 1.5pt] table [x={time},y={position},col sep=comma] {Tryggvason_combined.csv};	
	\addplot [line width=0.25mm, color = blue]table [x={time},y={bottomFront},col sep=comma, each nth point=3, filter discard warning=false, unbounded coords=discard] {extentsRTinstability_At0dot5.csv};
	\addplot [line width=0.25mm, color = blue]table [x={time},y={topFront},col sep=comma, each nth point=3, filter discard warning=false, unbounded coords=discard] {extentsRTinstability_At0dot5.csv};
	\addplot [only marks,mark size = 1.5pt, mark=triangle*, color=black, each nth point=3, filter discard warning=false, unbounded coords=discard] table [x={time},y={position},col sep=comma] {xie_combined_At0dot82.csv};
	\addplot [line width=0.25mm, color = blue, dashed]table [x={time},y={bottomFront},col sep=comma, each nth point=3, filter discard warning=false, unbounded coords=discard] {extentsRTinstability_At0dot82.csv};
	\addplot [line width=0.25mm, color = blue, dashed]table [x={time},y={topFront},col sep=comma, each nth point=3, filter discard warning=false, unbounded coords=discard] {extentsRTinstability_At0dot82.csv};
	\end{axis}
	\end{tikzpicture}
	
	\begin{tikzpicture}
	\begin{axis}[width=0.46\linewidth,scaled y ticks=true,xlabel={Time (-)},ylabel={$E_{tot}(\vec{v},\phi,t)$},legend style={nodes={scale=0.65, transform shape}}, xmin=0, xmax=3, xtick={0,1,2,3},  title={(b)},
	]
	\addplot +[mark size = 1pt, each nth point=500, filter discard warning=false, unbounded coords=discard] table [x={time},y={TotalEnergy},col sep=comma] {Energy_data_At0dot5.csv};
	\end{axis}
	\end{tikzpicture}
	\hskip 5pt
	\begin{tikzpicture}
	\begin{axis}[width=0.46\linewidth,scaled y ticks=true,xlabel={Time (-)},ylabel={$\int_{\Omega} \phi (x_i)\mathrm{d}x_i - \int_{\Omega} \phi_{0} (x_i)\mathrm{d}x_i$},legend style={nodes={scale=0.65, transform shape}}, xmin=0, xmax=3.0, xtick={0,1,2,3,4}, title={(c)},
	ymin=-1e-13,ymax=1e-13
	]
	\addplot+[mark size = 0.75pt, each nth point=100, filter discard warning=false, unbounded coords=discard]table [x={time},y={TotalPhiNorm},col sep=comma] {Energy_data_At0dot5.csv};
	\end{axis}
	\end{tikzpicture}
	\caption{\textit{Rayleigh-Taylor instability 2D} (\cref{subsec:rayleigh_taylor_2D}): (a) Comparison of positions of top and bottom front of the interface with literature; (b) decay of the energy functional illustrating theorem \ref{th:energy_stability} for $At = 0.5$; (c) total mass conservation (integral of total $\phi$) for $At = 0.5$.}
	\label{fig:RT_2D_comparison}
\end{figure}

\subsection{3D simulations and comparison with experiments: Single rising bubble}
\label{subsec:single_rising_drop}
Now that we have shown that our numerical method compares well with other benchmark numerical experiments in the literature, we move to testing the method in comparison to experimental results.  We now present results for single rising bubble in 3D with the octree based meshes.  We compare these simulations with experimental data presented in the literature.  The non-dimensional setting is same as discussed in~\cref{subsec:single_rising_drop_2D}.  In all the numerical experiments for single bubble rise we keep the viscosity ratio to be 100.  We present numerical experiments with density ratios of 100, 1000, 10000, to show the robustness of the algorithm to large density ratios.  See \cref{fig:bubblerisesetup} for a schematic of the computational domain selected.  The boundary conditions are no-slip on all walls, and zero flux for both $\mu$ and $\phi$, which are identical to ones used for functions spaces in the proofs.  We use the biCGstab (bcgs) linear solver from the PETSc suite along with the Additive Schwarz (ASM) preconditioner for the linear solves in the Newton iterations (see \cref{subsec:newton_iter}).  The details of the actual command line arguments used are given in \ref{sec:app_linear_solve}.

The convergence criterion for all the 3D bubble rise cases are as follows.  For the Navier-Stokes (NS) block we use a relative tolerance of $10^{-6}$, and for the Cahn-Hilliard (CH) block the relative tolerance is set to $10^{-8}$.  
The tolerance for block iteration errors is set to $10^{-4}$ ($\text{block}_\text{tol}$ from \cref{fig:flowchart_block}).

From the numerical experiments we can predict the non-dimensional terminal rise velocity $u_T$ of the bubbles as the velocity of the centre of mass of the bubbles.  This allows us to calculate Reynolds number based on terminal velocity as $Re_T = Ar \cdot u_T$.  $Re_T$ is our first metric for comparison with experiments reported in \citet{Bhaga1981}.  The second metric to compare with experiments is chosen to be the terminal shape of the bubble.  To show the importance of energy stability, in the cases we present we use a fairly large time-step of $10^{-2}$.   

\begin{figure}[H]
\centering
\includegraphics[width=0.35\linewidth]{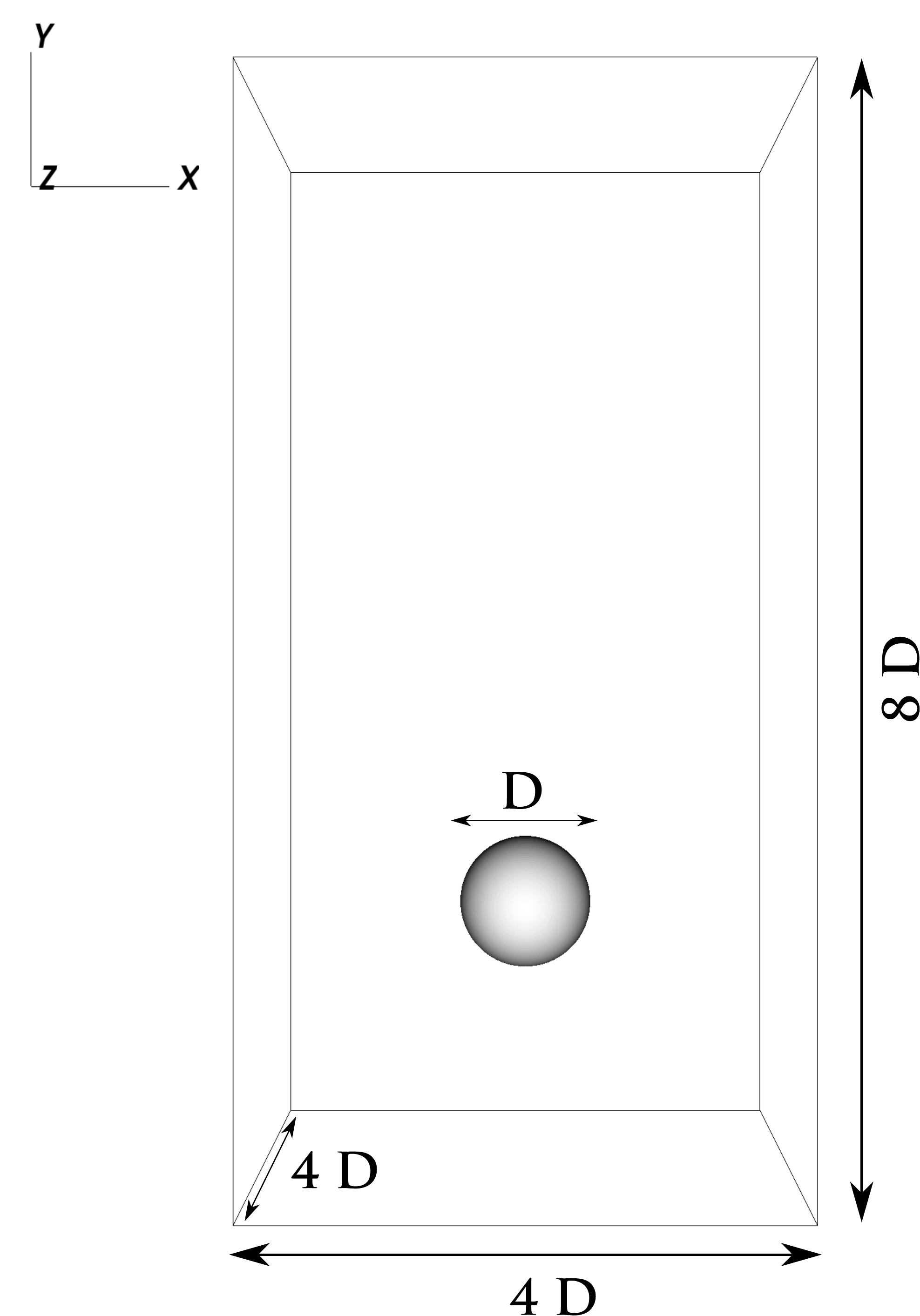}
\caption{Schematic of the computational domain used for 3D single bubble rise described in \cref{subsec:single_rising_drop}.}
\label{fig:bubblerisesetup}
\end{figure}

\subsubsection{Effect of Cn number}
\label{subsubsec:eff_cn_no}

As the model relies on selecting a computationally feasible thickness of the interface, it becomes important study the effect of $Cn$ (represents the non-dimensional thickness of the interface) on the performance of the model in comparison with the experiments.  The model approaches real physics in the limit of $Cn \rightarrow 0$, but as we decrease  $Cn$ we need to resolve smaller and smaller length scales. Therefore, decreasing Cahn number requires increasing mesh density, thereby making simulations more and more expensive.  One usually selects an 'optimum' $Cn_{opt}$, such that decreasing $Cn$ beyond this threshold, the quality of the solutions does not change (either measured by comparison with experiments or via lack of change of key quantities of interest).  To find this $Cn_{opt}$ number, we conduct three numerical experiments with Cn = 0.0125, 0.01, and 0.0075. We select the case of $Ar = 13.95$, $We = 116$. \cref{fig:effectOfCnNumber} shows the results from numerical experiments compared with the experiments for the respective $Cn$ numbers.  First of all, we can clearly see that the results show an excellent match with the experimental data, both in terms of shape of the bubble and the terminal Reynolds numbers.  An important observation is that there is a small difference between shape of the bubbles between case for $Cn = 0.0125$, and for the case of $Cn = 0.01$. But, there is no noticeable difference in the shape between the cases for $Cn = 0.01$ and $Cn = 0.0075$.  This indicates that an asymptotic behavior independent of $Cn$ number is reached and we set $Cn_{opt} = 0.01$. The rest of the numerical experiments presented in the paper for single bubble rising, we use a $Cn$ number of 0.01. The choice of $Cn_{opt}$ allows us to determine the adaptive meshing criterion for the case.  We maintain at least 6 elements with the size of $8D/2^{11}$ within the diffuse interface and a very coarse mesh everywhere else with element size of $8D/2^6$.  \cref{fig:mesh_adaption} shows the adaptivity of mesh as the air-water interface moves in the domain. 
 
\begin{figure}[H]
\centering
\includegraphics[width=0.5\linewidth]{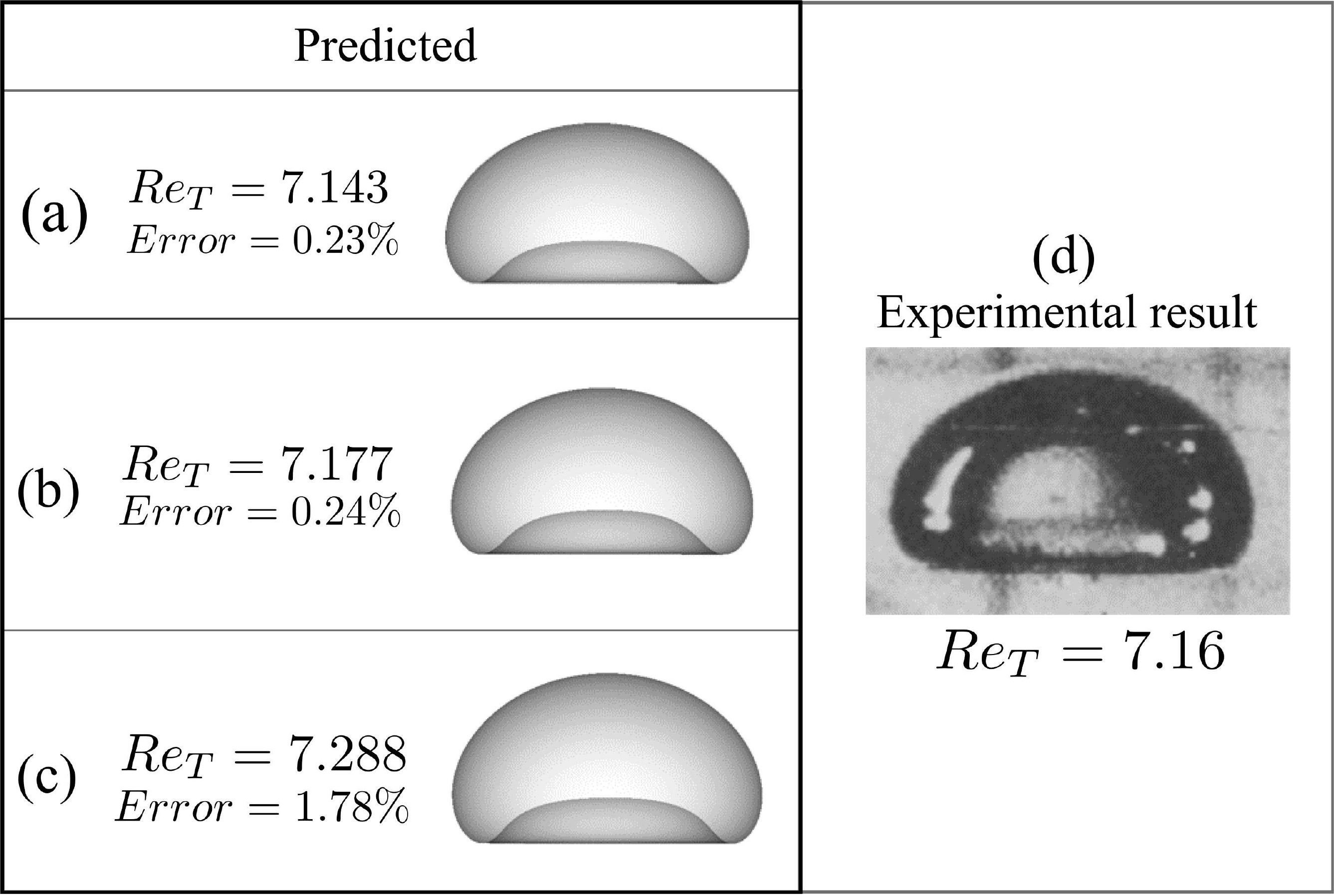}
\caption{\textit{Effect of Cn number} (\cref{subsubsec:eff_cn_no}): (a) Bubble shape for Cn = 0.0125; (b) bubble shape for Cn = 0.01; (c) bubble shape for Cn = 0.0075; (d) experimental result from~\citet{Bhaga1981}.}
\label{fig:effectOfCnNumber}
\end{figure}

\begin{figure}
	\centering
	\includegraphics[width=\linewidth]{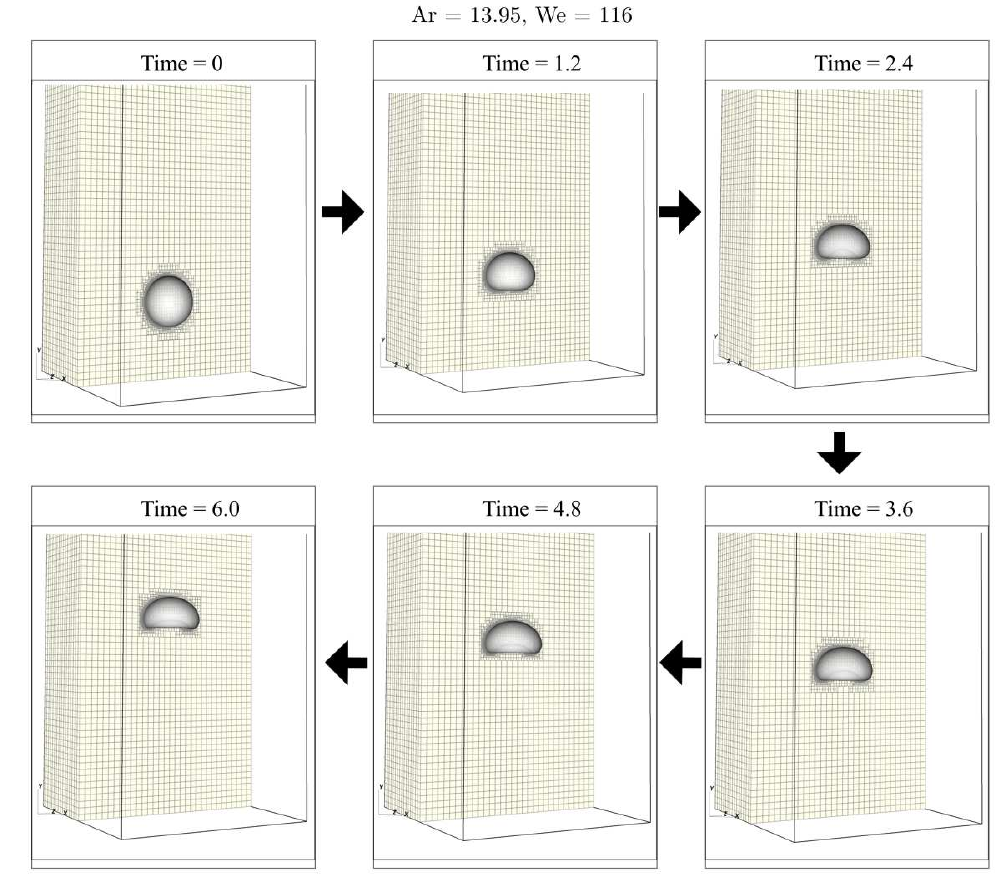}
	\caption{\textit{Evolution of mesh} (\cref{subsubsec:eff_cn_no}): Snapshots of the mesh at various time-points in the simulation.  Only half the mesh of the actual domain is shown in the figure to illustrate the refinement around the air-water interface.}
	\label{fig:mesh_adaption}
\end{figure}

\subsubsection{Effect of density ratio}
\label{subsubsec:eff_den_ratio}
We now investigate the effect of density ratio.  Typically with air-water system we see a density ratio $\rho_{+}/\rho_{-}$ (continuous to dispersed) of 1000.  We test the algorithm for three density ratios of $10^4$, $10^3$, and $10^2$.  \citet{Tryggvason2006} reported the effect of density ratios is primarily on the rise velocity of the bubble and the shape does not change after a threshold density ratio is high enough (\citet{Tryggvason2006} reported this threshold to be 50). Figure \ref{fig:effectOfDensityRatio} shows the independence of bubble shape for density three density ratios of $10^2$ (panel (a)), $10^3$ (panel (b)), and $10^4$ (panel (c)).  We observe that the shape predicted by the simulation for all the density ratios have no variation.  Therefore, we see the same independence of bubble shape on density ratio as reported by \citet{Tryggvason2006}.  We next investigate the effect of density ratio on the temporal variation of bubble rise velocity.  \cref{fig:effectOfDensityRatioRiseVel} shows the temporal evolution of rise velocity of the bubble for three density ratios.  It is clear that the difference between the curves is not very high, but the inset plots show that the rise velocity increases as the density ratio is increased. This behavior is reported in multiple studies in the literature \citep{Tryggvason2006,Hua2007,Balcazar2015}.  

\begin{figure}[H]
	\centering
	\includegraphics[width=0.5\linewidth]{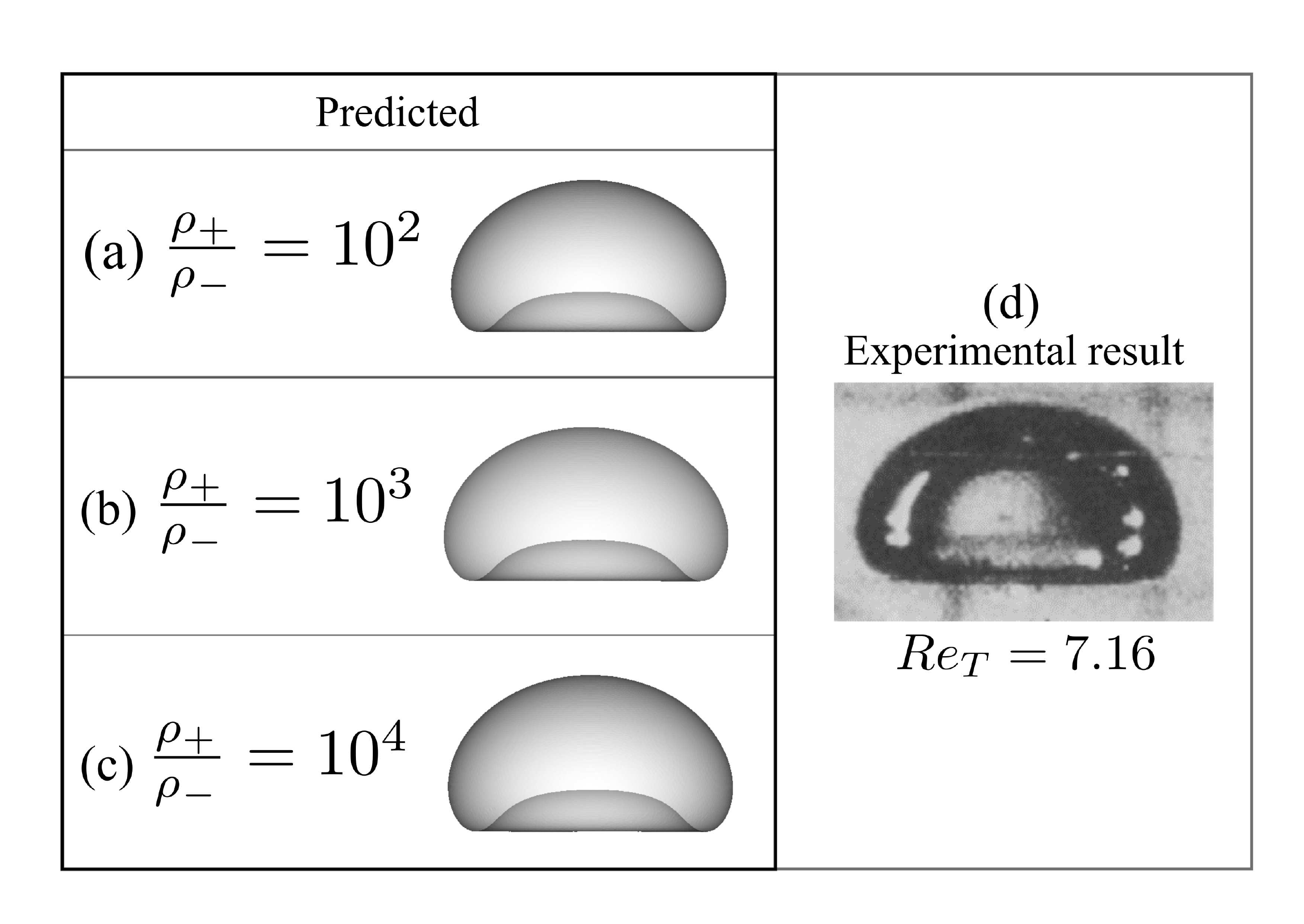}
	\caption{\textit{Effect of density ratio on bubble shape}
	(\cref{subsubsec:eff_den_ratio}): (a) Bubble shape for $\rho_{+}/\rho_{-} = 10^2$; (b) bubble shape for $\rho_{+}/\rho_{-} = 10^3$; (c) bubble shape for $\rho_{+}/\rho_{-} = 10^4$; (d) experimental result from \citet{Bhaga1981} (reproduced with permission from D. Bhaga, M. Weber, Bubbles in viscous liquids: shapes, wakes and velocities, Journal of fluid Mechanics 105 (1981) 61--85.)}
	\label{fig:effectOfDensityRatio}
\end{figure}


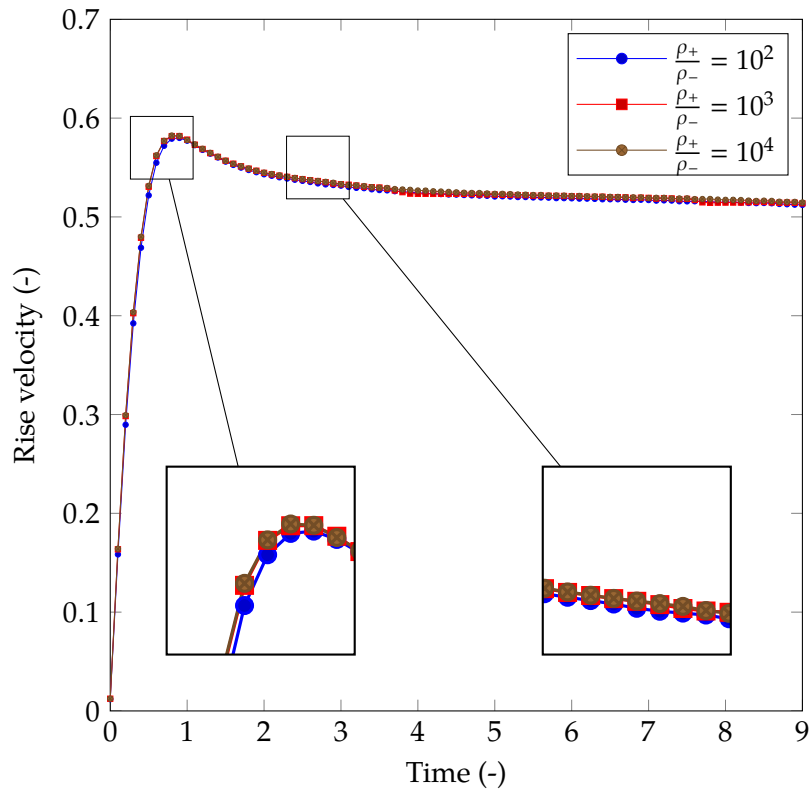
\begin{figure}[H]
	\centering
	\begin{tikzpicture}[spy using outlines={rectangle, magnification=3, size=2.5cm, connect spies}]
	\begin{axis}[width=0.65\linewidth, 
				height = 0.65\linewidth, 
				xmin=0, xmax=9, 
				ymin=0, ymax=0.7, 
				legend entries={${\frac{\rho_{+}}{\rho_{-}}} = 10^2$,${\frac{\rho_{+}}{\rho_{-}}} = 10^3$,${\frac{\rho_{+}}{\rho_{-}}} = 10^4$},
				legend style={nodes={scale=0.95, transform shape}, row sep=2.5pt},
				xlabel={Time (-)},ylabel={Rise velocity (-)},
				legend image post style={scale=2.0}
				]
	\addplot +[mark size = 1pt] table [x={time},y={densityRatio100},col sep=comma] {RT_RiseVel_data_Re14We116.csv};
	\addplot +[mark size = 1pt] table [x={time},y={densityRatio1000},col sep=comma] {RT_RiseVel_data_Re14We116.csv};
	\addplot +[mark size = 1pt] table [x={time},y={densityRatio10000},col sep=comma] {RT_RiseVel_data_Re14We116.csv};
	\coordinate (a) at (axis cs:0.67,0.57);
	\coordinate (b) at (axis cs:2.7,0.55);
	\end{axis}
	\spy [black] on (a)in node  at (2,2);
	\spy [black] on (b)in node  at (7,2);
	\end{tikzpicture}
	\caption{Effect of density ratio on rise velocity of the bubble
	(\cref{subsubsec:eff_den_ratio}).}
	\label{fig:effectOfDensityRatioRiseVel}
\end{figure}

\subsubsection{Comparison with other cases, energy stability}
\label{subsubsec:comp_summ}
We illustrate our numerical method with two other cases. We select a case where bubble deformation is not very high with $Ar = 6.54$ and $We = 116$, and another case with very high deformation (crowing effect) with $Ar = 30.83$ and $We = 339$. \Cref{fig:comparisonWithOtherCases} shows comparison of the numerical method with the aforementioned cases.  We see an excellent match between simulations and experiments with errors of less than $3\%$ in $Re_{T}$.  We emphasize the fact that the mesh is refined only near the interface of the bubble and it is quite coarse everywhere else in the domain, and we are using only linear elements.  This shows that the VMS based approximation accurately captures the evolution of the system. This is comparable to recent work by \citet{Yan2018} which uses higher order NURBS with levels sets with no adaptive meshing. 

We check whether the numerical method follows the theoretical energy stability proved in \cref{th:energy_stability}.  We present the evolution of the energy functional defined in \cref{eqn:energy_functional} for the case of $Ar = 6.54$ and $We = 116$.  \cref{fig:energy_stab} shows the decay of the total energy functional in accordance with the energy stability condition.  
Similar behavior of the energy functional is seen for all our bubble rise numerical experiments. 
\begin{figure}[H]
	\centering
	\includegraphics[width=0.65\linewidth]{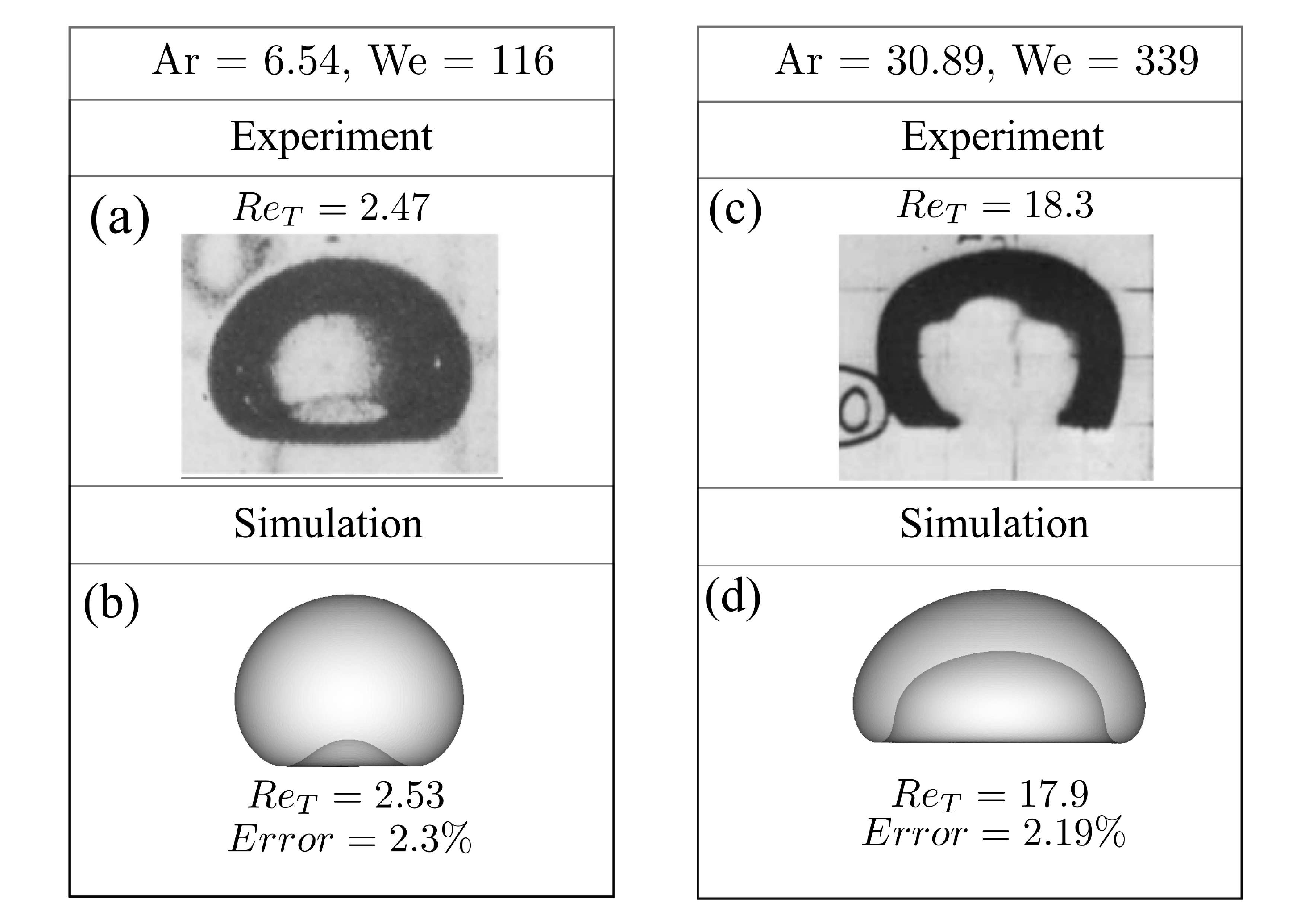}
	\caption{\textit{Comparison of the simulations with experiments}
	(\cref{subsubsec:comp_summ}): (a) Experimental terminal shape of the bubble with terminal velocity for $Ar = 6.54$, $We = 116$; (b) terminal bubble shape from simulations for the same conditions as panel (a); (c) experimental terminal shape of the bubble with terminal velocity for $Ar = 30.89$, $We = 339$; (d) terminal bubble shape from simulations for the same conditions as panel (c); (panel (a) and (c) reproduced with permission from D. Bhaga, M. Weber, Bubbles in viscous liquids: shapes, wakes and velocities, Journal of fluid Mechanics 105 (1981) 61--85.)}
	\label{fig:comparisonWithOtherCases}
\end{figure}


\begin{figure}[H]
	\centering
	\begin{tikzpicture}
	\begin{axis}[width=0.45\linewidth,scaled y ticks=true,xlabel={Time (-)},ylabel={$E_{tot}(v_i,\phi,t)$},legend style={nodes={scale=0.65, transform shape}}, xmin=0, xmax=10, ytick distance=0.5,  xtick={2,4,6,8,10}
	]
	\addplot +[mark size = 1pt, each nth point=20, filter discard warning=false, unbounded coords=discard] table [x={time},y={TotalEnergy},col sep=comma] {Energy_data_total_Re6_We116.csv};
	\end{axis}
	\end{tikzpicture}
	\caption{Decay of the energy functional illustrating theorem \ref{th:energy_stability} for the case of $Ar = 6.54$, $We = 116$
	(\cref{subsubsec:comp_summ}).} 
	\label{fig:energy_stab}
\end{figure}
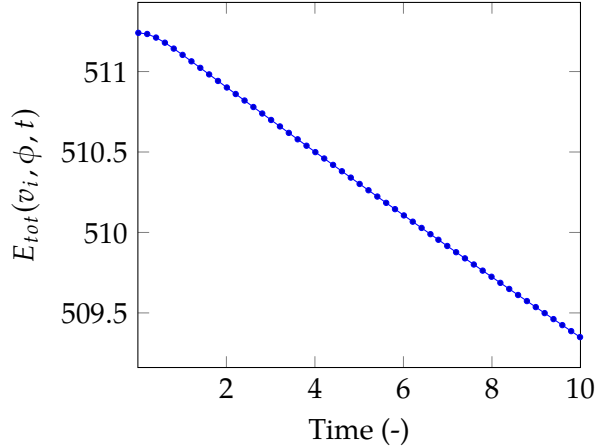

\subsection{3D simulations: Rayleigh-Taylor instability}
\label{subsec:rayleigh_taylor}

For the 3D simulations we choose an analytical initial condition for $\phi$ which governs the interface given by
\begin{align}
\phi(x_i) &= \tanh{\left(\frac{\left(h_0 - x_2\right) - g(x_i)}{\sqrt{2}Cn}\right)},\\ 
g(x_i) &= A \exp\left(-\left( \frac{ \left( x_1 - c_1 \right)^2 }{\lambda} + \frac{ \left( x_3 - c_3 \right)^2 }{\lambda} \right)\right).
\label{eq:initialConditionRT}
\end{align}
Here, $h_0$ is the location in the vertical direction for the interface, in this case chosen to be 2D from the bottom of the channel, $x_i$ is the position vector, $c_i$ is position of the centre of the Gaussian chosen to be $\left(0.5, 2.1, 0.5\right)$. $\lambda$ is the spread of the Gaussian, and $A$ is the amplitude of the Gaussian.  See \cref{fig:rayleighTaylorSetup} for a schematic of the computational domain selected along with the initial condition of the interface.   We use a $Cn$ number of 0.01 and the simulations resolve the large scales very well.  
We set the Reynolds number at 3000.  Further, the same choice of non-dimensionalising scales leads to a Weber number ($We = \rho_c g D^2/\sigma$).  We simulate two different initial conditions.  The We number is selected to be 1000.  Further, in the 3D simulations we choose the density ratio ($\rho_{+}/\rho_{-}$) of 0.33.  Similarly, $\nu_{+}/\nu_{-}$ is the viscosity ratio which is selected to be 1.0.

The simulations were performed using a time step of 0.0025. With the refinement near the interface being the finest at $4/2^{10}$ ensuring about 5 elements for resolving the diffuse interface, where as the refinement away from the interface was kept at $4/2^7$.  The boundary conditions we use are no-slip for velocity on all the walls and no flux conditions for $\phi$ and $\mu$.  We assume a 90 degree wetting angle for both the fluids.  An algebraic multigrid linear solver with additive Schwarz based smoothers is setup for the linear solves in the Newton iterations (see \cref{subsec:newton_iter}).  The details of the actual command line arguments used are given in \ref{sec:app_linear_solve}.

The convergence criterion for all the 3D Rayleigh-Taylor cases are as follows.  For the Navier-Stokes (NS) block we use a relative tolerance of $10^{-6}$, and for the Cahn-Hilliard (CH) block the relative tolerance is set to $10^{-8}$.  For the linear solves within each Newton iteration, for NS we use a relative tolerance of $10^{-7}$, and for CH relative tolerance is set to $10^{-7}$. The tolerance for block iteration errors is set to $10^{-4}$ ($\text{block}_\text{tol}$ from \cref{fig:flowchart_block}).

\subsubsection{Case 1 : $\lambda$ of 0.2}

The evolution of the interface along with the mesh adaption is shown in \cref{fig:mesh_adaption_rt}.  \Cref{fig:rt_comparison} shows a qualitative comparison of the interface shape with the shape previously reported for the same density ratio in \citet{Tryggvason1990}.  Although, the initial conditions for the interface in our case (inverted Gaussian) is different than the initial conditions used in \cite{Tryggvason1990} (two-dimensional harmonic wave), the nature of the instability evolving from both of the them is similar where a blob of heavy fluid on top penetrates into light fluid at the bottom, setting up interfacial instabilities.  It can be clearly seem from \cref{fig:rt_comparison} that the shapes at this fairly evolved times are quite similar to each other qualitatively. 
\begin{figure}[H]
	\centering
	\includegraphics[width=0.4\linewidth]{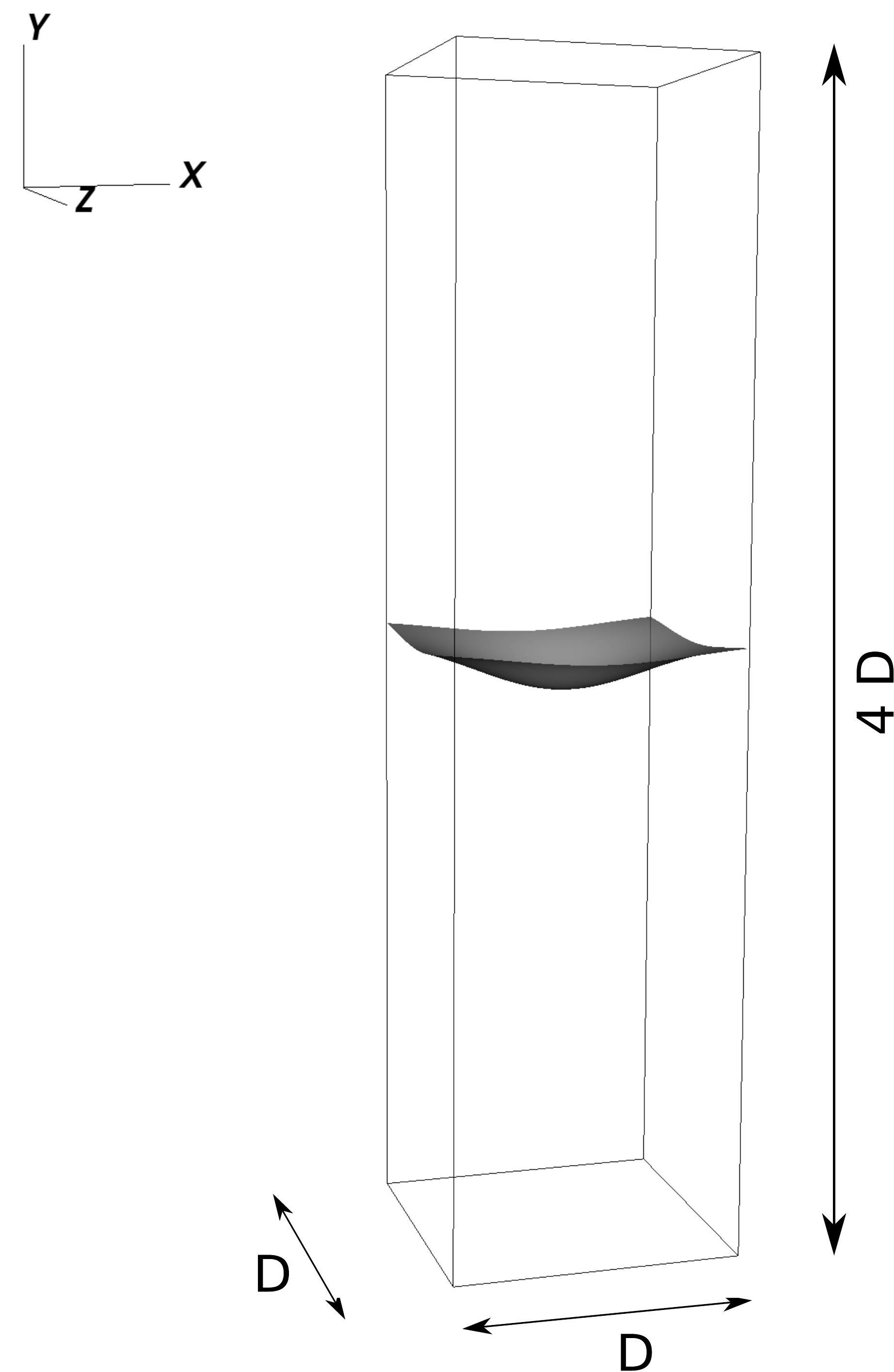}
	\caption{Schematic of the computational domain used for Rayleigh-Taylor instability (\cref{subsec:rayleigh_taylor}) with the iso-surface $\phi = 0$ showing the initial condition of the interface.}
	\label{fig:rayleighTaylorSetup}
\end{figure}

The initial perturbation of the interface (quite similar to that of \citep{Tryggvason1990}) is chosen here such that the front of the heavy fluid penetrates into the light fluid, making the interface buckle. It can be clearly seen in the evolution in \cref{fig:mesh_adaption_rt} that the sides of the interface clinging to the wall maintain a 90 degree angle as they rise to compensate for the motion of the centre front downwards.  \citet{Tryggvason1990} report that two counter rotating vortical structures are formed at the initial position of the interface propagate into the light fluid as vortices advance in with the blob. We also see a similar behavior in our simulations; \cref{fig:rt_streamline} show these two counter rotating vortices colored in blue.  As the fluid interface moves down in the centre, the lighter fluid is displaced and moves rapidly in near the walls going towards the corner setting off two counter rotating vortices near the wall.  These counter-rotating vortices are shown in red color in \cref{fig:rt_streamline}.  The same behavior was reported in \citet{Tryggvason1990}.  It is important to note that the simulation presented in this study is not highly resolved as $Cn$ here is 0.01 (which results in a mesh of about 3.5 million elements).  While we have performed simulations with higher resolution (14 million elements for $Cn = 0.0075$), we emphasize that a coarse resolution is able to resolve most of the physics reported in the literature.

\begin{figure}[H]
	\centering
	\includegraphics[width=0.85\linewidth]{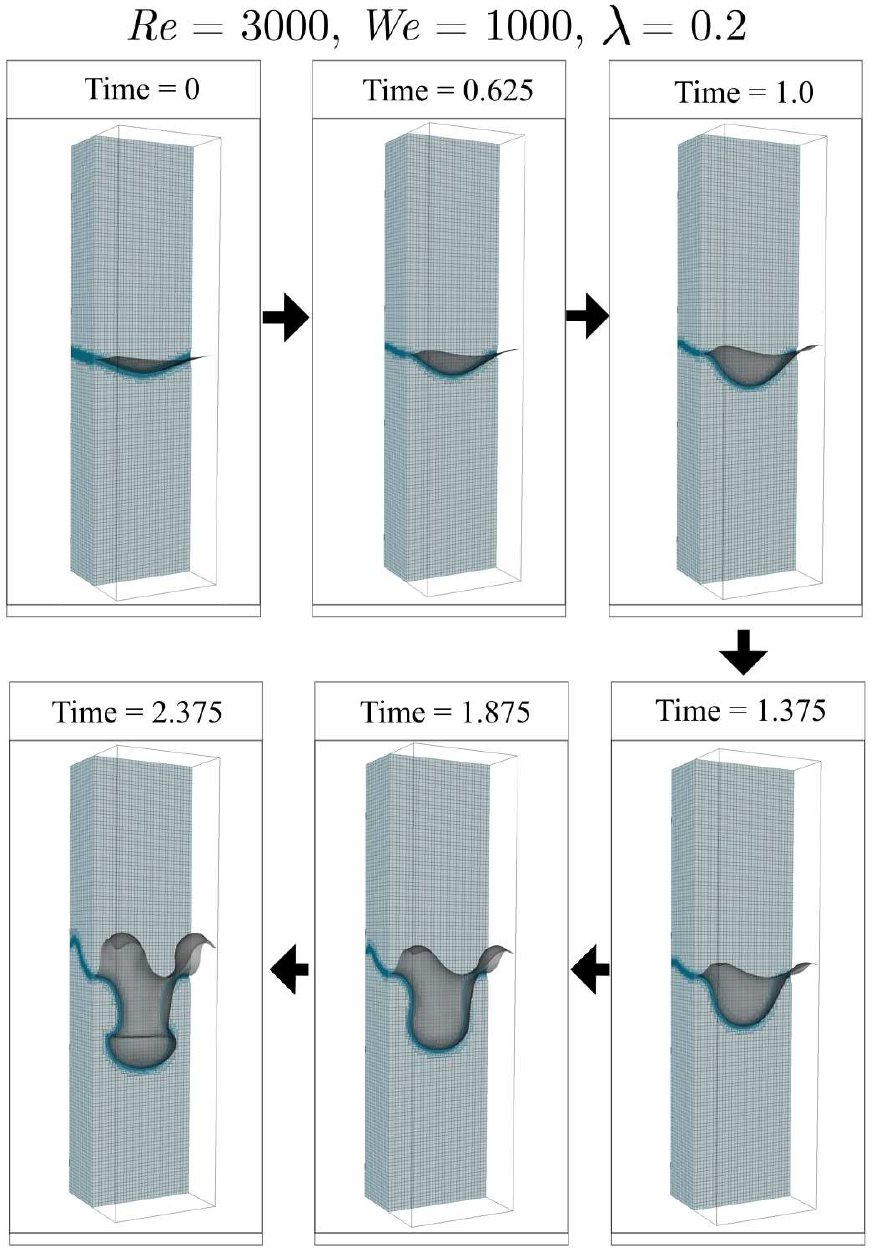}
	\caption{\textit{Evolution of mesh:} Snapshots of the mesh at various time-points in the simulation for Rayleigh-Taylor instability (\cref{subsec:rayleigh_taylor}) for $\lambda = 0.2$.  Only half the mesh of the actual domain is shown in the figure to illustrate the refinement around the interface of two fluids.}
	\label{fig:mesh_adaption_rt}
\end{figure}

\begin{figure}[H]
	\centering
	\includegraphics[width=0.75\linewidth]{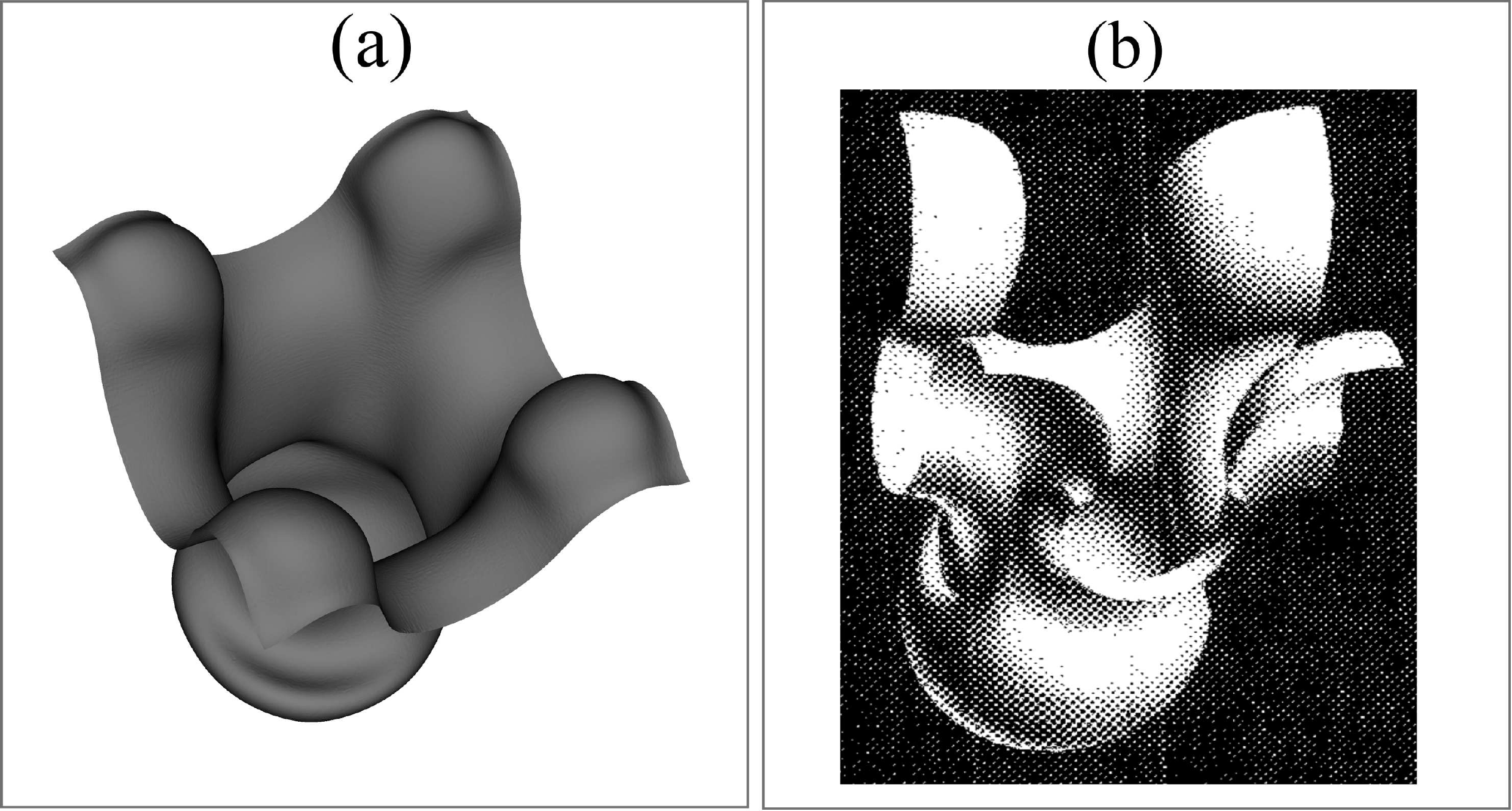}
	\caption{\textit{Qualitative comparison with previous literature} (\cref{subsec:rayleigh_taylor}): (a) Interface shape from current simulation at non-dimensional time 2.37; (b) interface shape reported in \citet{Tryggvason1990} at t = 3.0 (Reproduced from [G. Tryggvason, S. O. Unverdi, Computations of three-dimensional Rayleigh-Taylor instability,
		Physics of Fluids A: Fluid Dynamics 2 (5) (1990) 656--659
		], with the permission of AIP Publishing).}
	\label{fig:rt_comparison}
\end{figure}

\begin{figure}[H]
	\centering
	\includegraphics[width=0.5\linewidth]{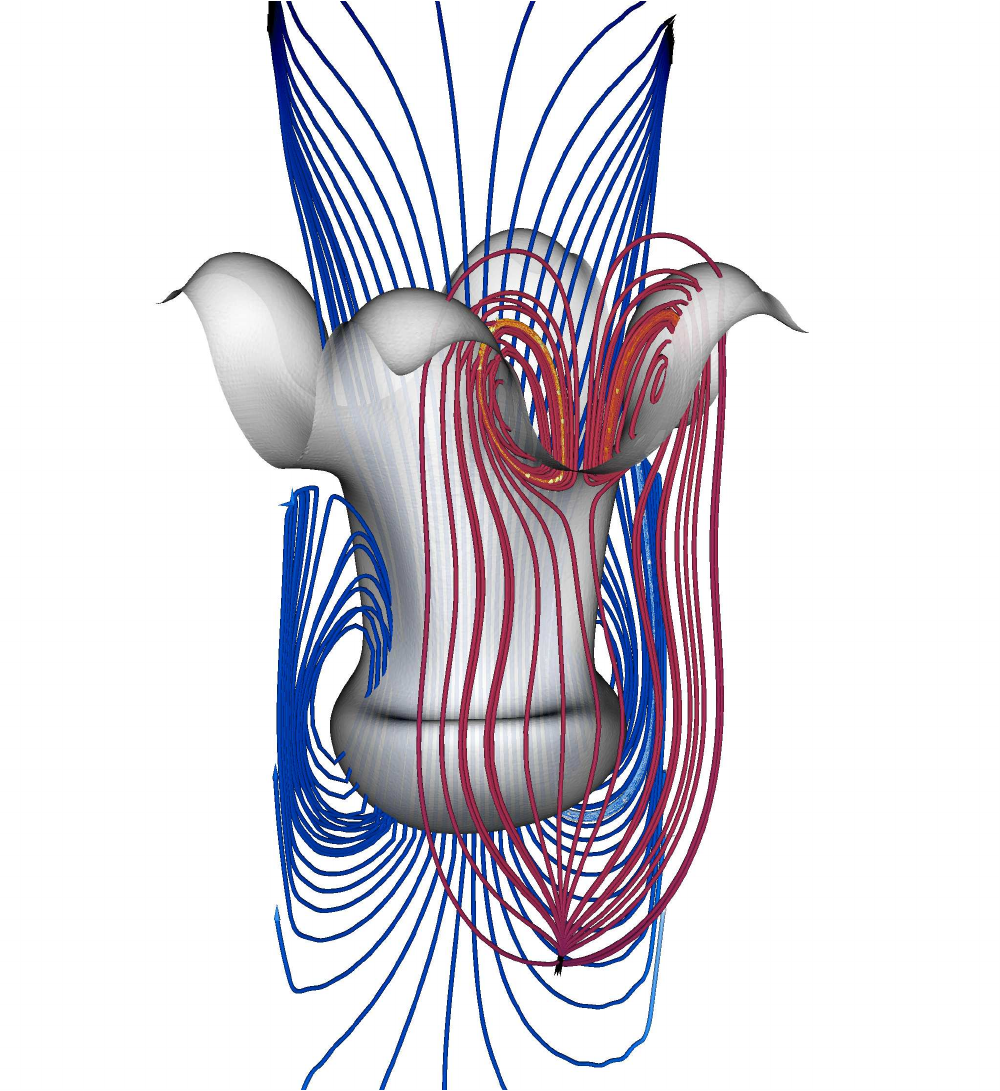}
	\caption{\textit{Streamlines drawn on top of the interface}
	(\cref{subsec:rayleigh_taylor}): Streamlines illustrating the vortical structures in Rayleigh-Taylor instability. The blue streamlines show the rollup of interface near the leading end of the interface.  The red-orange streamlines show the roll up of interface near the boundaries.}
	\label{fig:rt_streamline}
\end{figure}

\subsubsection{Case 2 : $\lambda$ of 0.08}

For this case we keep all the parameters the same as the case 1 except the $\lambda$ is decreased to 0.08.  The evolution of the interface is shown in \cref{fig:mesh_adaption_rt_c2}.  In this particular case we let the interface develop more to observe roll up and shedding at non-dimensional time t = 2.875.  A smaller $\lambda$ in the initial conditions allows for a much flatter initial profile of the interface at the wall, but a deeper penetration at the centre.  If we compare the evolution of the interface shape for the case of $\lambda = 0.2$ (\cref{fig:mesh_adaption_rt}) and for the case of $\lambda = 0.08$ (\cref{fig:mesh_adaption_rt_c2}), we observe that similar shapes are observed much sooner for the case of $\lambda = 0.08$.  For example, the shape of the interface at t = 2.375 in the case of $\lambda = 0.2$ is similar to the shape of interface at t = 1.875 near the center. 

Just like for the case of bubble rise, we check the behavior of the energy functional to observe whether energy stability is followed.  We present the evolution of the energy functional defined in \cref{eqn:energy_functional}.  \Cref{fig:RT_energy} shows the decay of the total energy functional in accordance with the energy stability condition. Similar behavior of the energy functional is also seen for the case of $\lambda = 0.2$.


\begin{figure}[H]
	\centering
	\includegraphics[width=0.75\linewidth]{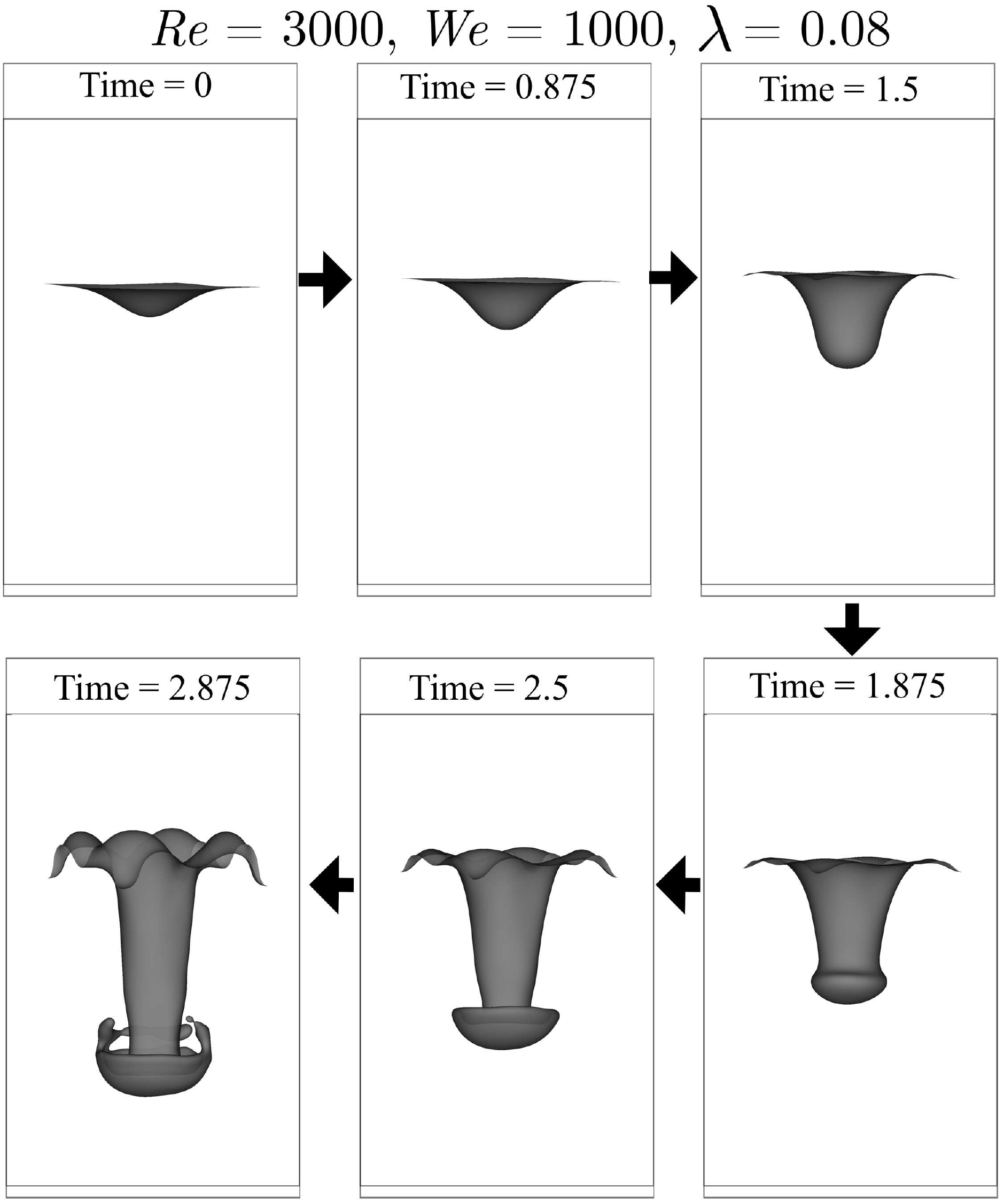}
	\caption{\textit{Evolution of the interface}
	(\cref{subsec:rayleigh_taylor}): Snapshots of the interface at various time-points in the simulation for Rayleigh-Taylor instability for $\lambda = 0.08$.}
	\label{fig:mesh_adaption_rt_c2}
\end{figure}

\begin{figure}[H]
	\centering
	\begin{tikzpicture}
	\begin{axis}[width=0.45\linewidth,scaled y ticks=true,xlabel={Time (-)},ylabel={$E_{tot}(\vec{v},\phi,t)$},legend style={nodes={scale=0.65, transform shape}}, xmin=0, xmax=3, xtick={0,1,2,3}
	]
	\addplot +[mark size = 1pt, each nth point=20, filter discard warning=false, unbounded coords=discard] table [x={time},y={TotalEnergy},col sep=comma] {RT_Energy_data_total.csv};
	\end{axis}
	\end{tikzpicture}
	\hskip 5pt
	\caption{Decay of the energy functional illustrating  \cref{th:energy_stability} for the case of $Re = 3000$ and $We = 1000$
	in the simulation of the 3D Rayleigh-Taylor instability
	(\cref{subsec:rayleigh_taylor}).
	}
	\label{fig:RT_energy}
\end{figure}
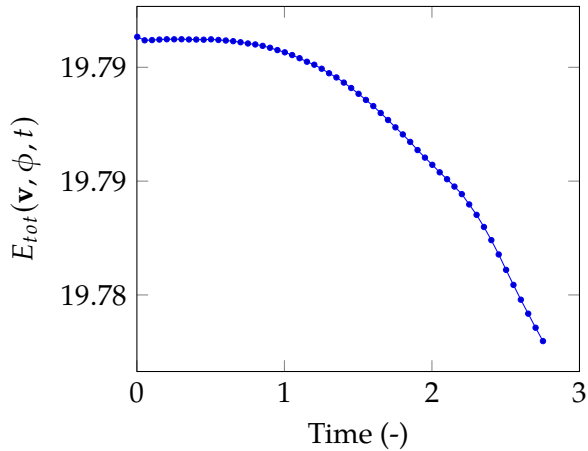

\section{Scaling of the numerical implementation}
\label{sec:scaling}
We perform scaling analysis to demonstrate scaling and parallelization of the framework.  All scaling tests were performed on TACC Stampede2 using the Knights Landing processors ($p = 136,\dots,17408$).  We used the bubble rise case in \cref{subsec:single_rising_drop} as a sample case for the scaling analysis with $Ar = 13.95$ and $We = 116$.  The bubble rise case for each scaling experiment is run for 5 time-steps so that any deviations from the long time behavior (timing of convergence in non-linear solves) in the initial time-steps does not dominate the timing.  We adaptively refine the mesh around the interface of the sphere five levels deeper than the rest of the background mesh.  The mesh is defined by a pair of minimum refinement $C$ and maximum refinement $R$, where the background mesh element size ranges from $8/2^C$ to $8/2^R$ at the interface.   We run this experiment on four background/interface refinement levels: 5/10, 6/11, 7/12, 8/13. Each refinement level has roughly seven to eight times more degrees of freedom to solve for than the previous level, with 5/10 having around 800,000 degrees of freedom and 8/13 reaching 138 million degrees of freedom.

We note that given specific $C$ and $R$ and the same initial conditions, the overall problem size in spite of mesh-refinement is consistent independent of the number of processes being used for the simulation. To this effect, we believe presenting performance for different $C/R$ combinations for different number of processes in the style of a strong scaling is appropriate.  For the same initial conditions, non-dimensional numbers, and a specific choice of refinement levels ($C$ and $R$), the problem is consistent independent of the number of processes being used, which allows us to use strong scaling type analysis.  Therefore, we vary the number of processes for each combination of $C$ and $R$ and present the timing information.  \Cref{fig:strong_scaling} shows the strong scaling analysis, and it can be seen that c code scales well, with continuing reductions in time-to-solve. Performing weak scaling for our case is harder because of mesh refinements, and subsequently the change in problem size. Therefore, we derive the weak-scalability from a set of strong scaling experiment.  We connect the points which roughly have the same number of elements per process to achieve this.  \Cref{fig:weak_scaling} shows weak scaling with curves being dashed lines for weak scaling, with the aforementioned approximation.  

From \cref{fig:strong_scaling,fig:weak_scaling} we observe that the numerical framework shows very good strong and weak scaling. 

\begin{figure}
	\centering
	\begin{tikzpicture}
	\begin{loglogaxis}[width=300pt,scaled y ticks=true,xlabel={Stampede2 KNL processors},ylabel={Time (s)},legend entries={C/R=5/10,C/R=6/11,C/R=7/12,C/R=8/13},legend style={nodes={scale=0.65, transform shape}}]
	\addplot table [x={ntasks},y={summary},col sep=comma, discard if not={refine_lvl_base}{5}] {results_moving.csv};
	\addplot table [x={ntasks},y={summary},col sep=comma, discard if not={refine_lvl_base}{6}] {results_moving.csv};
	\addplot +[mark=triangle*] table [x={ntasks},y={summary},col sep=comma, discard if not={refine_lvl_base}{7}] {results_moving.csv};
	\addplot +[mark=diamond*] table [x={ntasks},y={summary},col sep=comma, discard if not={refine_lvl_base}{8}] {results_moving.csv};
	\end{loglogaxis}
	\end{tikzpicture}
	\caption{Total time to solve five timesteps (including remeshing) for varying refinement configurations $C/R$ for a bubble rise in a channel.  The setup is identical to the case selected in section \ref{subsec:single_rising_drop} with the domain being the one presented in figure \ref{fig:bubblerisesetup}.} 
	\label{fig:strong_scaling}
\end{figure}
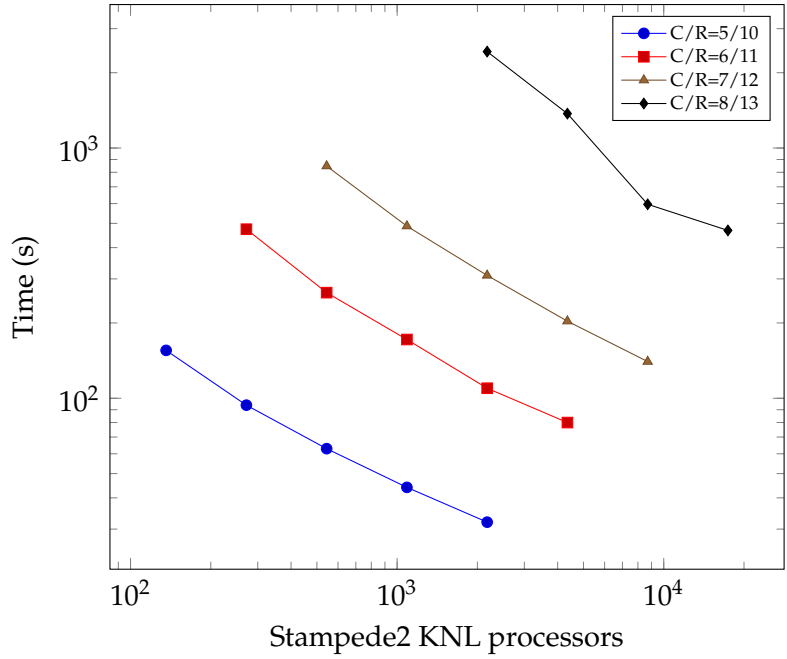 

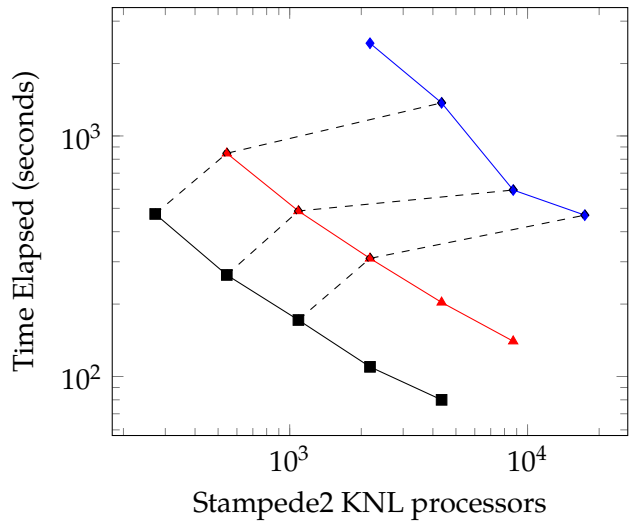
\begin{figure}
	\centering
	\begin{tikzpicture}
	\begin{loglogaxis}[
	xlabel=Stampede2 KNL processors,
	ylabel=Time Elapsed (seconds),
	legend style={at={(0.95,0.95)},anchor=north east}]
	\addplot[color=black,mark=square*] coordinates {
		(272,    473.8485147059)
		(544,    264.3035367647)
		(1088,    171.7272996324)
		(2176,    109.6786107537)
		(4352,    79.9761809283)
	};
	\addplot[color=red,mark=triangle*] coordinates {
		(544,    847.0702904412)
		(1088,    487.53003125)
		(2176,    309.7423203125)
		(4352,    203.2675539982)
		(8704,    140.1693323759)
	};
	\addplot[color=blue,mark=diamond*] coordinates {
		(2176,    2428.7664016543)
		(4352,    1372.6706962317)
		(8704,   596.0126691177)
		(17408,    468.5404142912)
	};

	\addplot[color=black,mark=diamond,dashed] coordinates {
		(272, 473.8485147059)
		(544, 847.0702904412)
		(4352, 1372.6706962317)
	};
	\addplot[color=black,mark=diamond,dashed] coordinates {
		(544, 264.3035367647)
		(1088,  487.53003125)
		(8704, 596.0126691177)
	};
	\addplot[color=black,mark=diamond,dashed] coordinates {
		(1088,    171.7272996324)
		(2176,    309.7423203125)
		(17408,    468.5404142912)
	};
	\end{loglogaxis}
	\end{tikzpicture} 
	\caption{\textit{Weak scalability approximated from multiple strong-scaling experiments}: We approximate the weak (dashed lines) scaling from the 
		strong (solid lines) scaling results for $r = (6/11,7/12,8/13)$ and $p$ up to $17408$ on Stampede2 Knights Landing processors.  Connecting the points which have same approximately same number of elements per process.
		\label{fig:weak_scaling} }
\end{figure}


\section{Conclusions and future work}

We have reported on a continuous Galerkin (cG) based framework to simulate two-phase flows with the thermodynamically consistent Cahn-Hilliard Navier-Stokes model.  We present rigorous proofs of energy stability for the implicit time scheme that we have selected.  We also present an existence result for our system, particularly studying the advective Cahn-Hilliard operator.  A variational multi-scale approach is used to model momentum equations and provide grad-div stabilisation for the proposed cG method.  The continuous model is discretized in space using a cG method with a massively parallel adaptive meshing framework called \dendro.  Extensive numerical experiments were carried out to test the accuracy of the numerical model.  The numerical model was validated against experimental datasets for an extreme density ratio (100 to 10000) and showed excellent agreement with the experimental results. We show that the model performs as good or better in comparison with front-tracking or level set models.  We demonstrated that the fully discretized numerical scheme also follows the energy stability and mass conservation proved for the semi-discrete form of the model.  Further, we used the model to simulate a turbulent case of Rayleigh-Taylor instability and the results presented show that the model resolves most of the physics reported in the literature using front tracking models.  We carried out extensive scaling tests of the numerical framework and show excellent weak and strong scaling till about 17000 MPI processes.  

We identify several avenues for future work. One critical aspect is related to practical applications of the CH-NS model, which require 3D simulations across long time horizons. As reported here, adaptive mesh refinement (AMR) provides an efficient strategy to model practical applications. AMR requires interpolating fields from a mesh at the previous time step to an updated mesh at the current time step. Care has to be taken to ensure that these interpolation schemes are also mass-conserving. In the current work, we see that the coarsening step (where eight daughter cells are replaced by one larger parent cell) could potentially not conserve mass. This suggests deployment of mass-conserving interpolation schemes, which is the focus of our subsequent work.

Another avenue of work is to extend the order of the basis function to consider higher order cG spaces. Higher order methods are useful on two fronts: (1) they improve the overall quality of the solutions and better enforce the solenoidality of mixture velocity; and (2) they also decrease the disparity in the largest and smallest scales for low Reynolds number applications where the disparity in length scales is only because the interface needs to be resolved with a high density of fine elements.  This will help speed up the framework for optimization applications targeted towards micro-fluidics.  We are currently working on developing a fully coupled solver instead of the block iteration approach presented in this paper.  The fully coupled approach will be faster because of less number of matrix assemblies.  In the current paper the mesh is only refined near the interface (interface scales), we are working on developing {\it a posteriori} estimates to refine the mesh based on both interface and velocity scales.  This will help resolve wakes and boundary layers much better in many applications involving turbulence.

\section{Acknowledgements}

The authors acknowledge XSEDE grant number TG-CTS110007 for computing time on TACC Stampede2, ALCF computing resources and Iowa State University computing resources. JAR was supported in part by NSF Grant DMS--1620128. BG, AL and MAK was funded in part by NSF grant 1935255, 1855902. HS was funded in part by NSF grant 1912930.

\bibliographystyle{elsarticle-num-names}
\bibliography{total_references,hari}

\begin{thebibliography}{71}
\providecommand{\natexlab}[1]{#1}
\providecommand{\url}[1]{\texttt{#1}}
\providecommand{\urlprefix}{URL }
\expandafter\ifx\csname urlstyle\endcsname\relax
  \providecommand{\doi}[1]{doi:\discretionary{}{}{}#1}\else
  \providecommand{\doi}[1]{doi:\discretionary{}{}{}\begingroup
  \urlstyle{rm}\url{#1}\endgroup}\fi
\providecommand{\bibinfo}[2]{#2}

\bibitem[{Notz and Basaran(2004)}]{Notz2004}
\bibinfo{author}{P.~K. Notz}, \bibinfo{author}{O.~A. Basaran},
  \bibinfo{title}{Dynamics and breakup of a contracting liquid filament},
  \bibinfo{journal}{Journal of Fluid Mechanics} \bibinfo{volume}{512}
  (\bibinfo{year}{2004}) \bibinfo{pages}{223--256}.

\bibitem[{Osher and Fedkiw(2003)}]{Osher2006}
\bibinfo{author}{S.~Osher}, \bibinfo{author}{R.~Fedkiw}, \bibinfo{title}{Level
  Set Methods and Dynamic Implicit Surfaces}, vol. \bibinfo{volume}{153} of
  \emph{\bibinfo{series}{Applied Mathematical Sciences}},
  \bibinfo{publisher}{Springer}, \bibinfo{year}{2003}.

\bibitem[{Unverdi and Tryggvason(1992)}]{Unverdi1992}
\bibinfo{author}{S.~O. Unverdi}, \bibinfo{author}{G.~Tryggvason},
  \bibinfo{title}{A front-tracking method for viscous, incompressible,
  multi-fluid flows}, \bibinfo{journal}{Journal of Computational Physics}
  \bibinfo{volume}{100}~(\bibinfo{number}{1}) (\bibinfo{year}{1992})
  \bibinfo{pages}{25--37}.

\bibitem[{Brackbill et~al.(1992)Brackbill, Kothe, and Zemach}]{Brackbill1992}
\bibinfo{author}{J.~U. Brackbill}, \bibinfo{author}{D.~B. Kothe},
  \bibinfo{author}{C.~Zemach}, \bibinfo{title}{A continuum method for modeling
  surface tension}, \bibinfo{journal}{Journal of Computational Physics}
  \bibinfo{volume}{100}~(\bibinfo{number}{2}) (\bibinfo{year}{1992})
  \bibinfo{pages}{335--354}.

\bibitem[{Prosperetti and Tryggvason(2009)}]{Prosperetti2009}
\bibinfo{author}{A.~Prosperetti}, \bibinfo{author}{G.~Tryggvason},
  \bibinfo{title}{Computational Methods for Multiphase Flow},
  \bibinfo{publisher}{Cambridge University Press}, \bibinfo{year}{2009}.

\bibitem[{Jacqmin(1996)}]{Jacqmin1996}
\bibinfo{author}{D.~Jacqmin}, \bibinfo{title}{An energy approach to the
  continuum surface tension method}, in: \bibinfo{booktitle}{34th AIAA
  Aerospace Sciences Meeting {\&} Exhibit}, vol. \bibinfo{volume}{AIAA96},
  \bibinfo{organization}{American Institute of Aeronautics and Astronautics},
  \bibinfo{address}{Reno, Nevada, USA}, \bibinfo{pages}{0858},
  \bibinfo{year}{1996}.

\bibitem[{Jacqmin(2000)}]{Jacqmin2000}
\bibinfo{author}{D.~Jacqmin}, \bibinfo{title}{Contact-line dynamics of a
  diffuse fluid interface}, \bibinfo{journal}{Journal of Fluid Mechanics}
  \bibinfo{volume}{402} (\bibinfo{year}{2000}) \bibinfo{pages}{57--88},
  \doi{\bibinfo{doi}{10.1017/S0022112099006874}}.

\bibitem[{Anderson et~al.(1998)Anderson, McFadden, and Wheeler}]{Anderson1998}
\bibinfo{author}{D.~M. Anderson}, \bibinfo{author}{G.~B. McFadden},
  \bibinfo{author}{A.~A. Wheeler}, \bibinfo{title}{{Diffuse-interface methods
  in fluid mechanics}}, \bibinfo{journal}{Annual Review of Fluid Mechanics}
  \bibinfo{volume}{30}~(\bibinfo{number}{1}) (\bibinfo{year}{1998})
  \bibinfo{pages}{139--165}, ISSN \bibinfo{issn}{0066-4189},
  \doi{\bibinfo{doi}{10.1146/annurev.fluid.30.1.139}}.

\bibitem[{Guo et~al.(2017)Guo, Lin, Lowengrub, and Wise}]{Guo2017}
\bibinfo{author}{Z.~Guo}, \bibinfo{author}{P.~Lin},
  \bibinfo{author}{J.~Lowengrub}, \bibinfo{author}{S.~M. Wise},
  \bibinfo{title}{{Mass conservative and energy stable finite difference
  methods for the quasi-incompressible Navier--Stokes--Cahn--Hilliard system:
  Primitive variable and projection-type schemes}}, \bibinfo{journal}{Computer
  Methods in Applied Mechanics and Engineering} \bibinfo{volume}{326}
  (\bibinfo{year}{2017}) \bibinfo{pages}{144--174}, ISSN
  \bibinfo{issn}{00457825}, \doi{\bibinfo{doi}{10.1016/j.cma.2017.08.011}}.

\bibitem[{Shokrpour~Roudbari et~al.(2018)Shokrpour~Roudbari,
  {\c{S}}im{\c{s}}ek, van Brummelen, and van~der Zee}]{Shokrpour2018}
\bibinfo{author}{M.~Shokrpour~Roudbari},
  \bibinfo{author}{G.~{\c{S}}im{\c{s}}ek}, \bibinfo{author}{E.~H. van
  Brummelen}, \bibinfo{author}{K.~G. van~der Zee},
  \bibinfo{title}{Diffuse-interface two-phase flow models with different
  densities: A new quasi-incompressible form and a linear energy-stable
  method}, \bibinfo{journal}{Mathematical Models and Methods in Applied
  Sciences} \bibinfo{volume}{28}~(\bibinfo{number}{04}) (\bibinfo{year}{2018})
  \bibinfo{pages}{733--770}, \doi{\bibinfo{doi}{10.1142/S0218202518500197}}.

\bibitem[{Hohenberg and Halperin(1977)}]{Hohenberg1977}
\bibinfo{author}{P.~Hohenberg}, \bibinfo{author}{B.~Halperin},
  \bibinfo{title}{{Theory of dynamic critical phenomena}},
  \bibinfo{journal}{Reviews of Modern Physics}
  \bibinfo{volume}{49}~(\bibinfo{number}{3}) (\bibinfo{year}{1977})
  \bibinfo{pages}{435--479}, ISSN \bibinfo{issn}{0034-6861},
  \doi{\bibinfo{doi}{10.1103/RevModPhys.49.435}}.

\bibitem[{Villanueva and Amberg(2006)}]{Villanueva2006}
\bibinfo{author}{W.~Villanueva}, \bibinfo{author}{G.~Amberg},
  \bibinfo{title}{Some generic capillary-driven flows},
  \bibinfo{journal}{International Journal of Multiphase Flow}
  \bibinfo{volume}{32}~(\bibinfo{number}{9}) (\bibinfo{year}{2006})
  \bibinfo{pages}{1072--1086}.

\bibitem[{Dong(2014)}]{Dong2014}
\bibinfo{author}{S.~Dong}, \bibinfo{title}{{An efficient algorithm for
  incompressible N-phase flows}}, \bibinfo{journal}{Journal of Computational
  Physics} \bibinfo{volume}{276} (\bibinfo{year}{2014})
  \bibinfo{pages}{691--728}, ISSN \bibinfo{issn}{10902716},
  \doi{\bibinfo{doi}{10.1016/j.jcp.2014.08.002}}.

\bibitem[{Xie et~al.(2015)Xie, Wodo, and Ganapathysubramanian}]{Xie2015}
\bibinfo{author}{Y.~Xie}, \bibinfo{author}{O.~Wodo},
  \bibinfo{author}{B.~Ganapathysubramanian}, \bibinfo{title}{{Incompressible
  two-phase flow: Diffuse interface approach for large density ratios, grid
  resolution study, and 3D patterned substrate wetting problem}},
  \bibinfo{journal}{Computers and Fluids} \bibinfo{volume}{141}
  (\bibinfo{year}{2015}) \bibinfo{pages}{223--234}, ISSN
  \bibinfo{issn}{00457930},
  \doi{\bibinfo{doi}{10.1016/j.compfluid.2016.04.011}}.

\bibitem[{Gurtin et~al.(1996)Gurtin, Polignone, and Vinals}]{Gurtin1996}
\bibinfo{author}{M.~E. Gurtin}, \bibinfo{author}{D.~Polignone},
  \bibinfo{author}{J.~Vinals}, \bibinfo{title}{Two-phase binary fluids and
  immiscible fluids described by an order parameter},
  \bibinfo{journal}{Mathematical Models and Methods in Applied Sciences}
  \bibinfo{volume}{6}~(\bibinfo{number}{06}) (\bibinfo{year}{1996})
  \bibinfo{pages}{815--831}.

\bibitem[{Abels et~al.(2012)Abels, Garcke, and Gr{\"u}n}]{Abels2012}
\bibinfo{author}{H.~Abels}, \bibinfo{author}{H.~Garcke},
  \bibinfo{author}{G.~Gr{\"u}n}, \bibinfo{title}{Thermodynamically consistent,
  frame indifferent diffuse interface models for incompressible two-phase flows
  with different densities}, \bibinfo{journal}{Mathematical Models and Methods
  in Applied Sciences} \bibinfo{volume}{22}~(\bibinfo{number}{03})
  (\bibinfo{year}{2012}) \bibinfo{pages}{1150013}.

\bibitem[{Volker(2016)}]{Volker2016}
\bibinfo{author}{J.~Volker}, \bibinfo{title}{Finite Element Methods for
  Incompressible Flow Problems}, vol.~\bibinfo{volume}{51} of
  \emph{\bibinfo{series}{Springer Series in Computational Mathematics Book}},
  \bibinfo{publisher}{Springer}, \bibinfo{year}{2016}.

\bibitem[{Burstedde et~al.(2011{\natexlab{a}})Burstedde, Wilcox, and
  Ghattas}]{Burstedde2011}
\bibinfo{author}{C.~Burstedde}, \bibinfo{author}{L.~C. Wilcox},
  \bibinfo{author}{O.~Ghattas}, \bibinfo{title}{p4est: Scalable algorithms for
  parallel adaptive mesh refinement on forests of octrees},
  \bibinfo{journal}{SIAM Journal on Scientific Computing}
  \bibinfo{volume}{33}~(\bibinfo{number}{3})
  (\bibinfo{year}{2011}{\natexlab{a}}) \bibinfo{pages}{1103--1133}.

\bibitem[{Sundar et~al.(2007{\natexlab{a}})Sundar, Sampath, Adavani,
  Davatzikos, and Biros}]{Sundar2007}
\bibinfo{author}{H.~Sundar}, \bibinfo{author}{R.~S. Sampath},
  \bibinfo{author}{S.~S. Adavani}, \bibinfo{author}{C.~Davatzikos},
  \bibinfo{author}{G.~Biros}, \bibinfo{title}{Low-constant parallel algorithms
  for finite element simulations using linear octrees}, in:
  \bibinfo{booktitle}{Proceedings of the 2007 ACM/IEEE conference on
  Supercomputing}, \bibinfo{organization}{ACM}, \bibinfo{pages}{25},
  \bibinfo{year}{2007}{\natexlab{a}}.

\bibitem[{Sundar et~al.(2008{\natexlab{a}})Sundar, Sampath, and
  Biros}]{Sundar2008}
\bibinfo{author}{H.~Sundar}, \bibinfo{author}{R.~S. Sampath},
  \bibinfo{author}{G.~Biros}, \bibinfo{title}{Bottom-up construction and 2: 1
  balance refinement of linear octrees in parallel}, \bibinfo{journal}{SIAM
  Journal on Scientific Computing} \bibinfo{volume}{30}~(\bibinfo{number}{5})
  (\bibinfo{year}{2008}{\natexlab{a}}) \bibinfo{pages}{2675--2708}.

\bibitem[{Kim et~al.(2004)Kim, Kang, and Lowengrub}]{Kim2004a}
\bibinfo{author}{J.~Kim}, \bibinfo{author}{K.~Kang},
  \bibinfo{author}{J.~Lowengrub}, \bibinfo{title}{{Conservative multigrid
  methods for Cahn--Hilliard fluids}}, \bibinfo{journal}{Journal of
  Computational Physics} \bibinfo{volume}{193}~(\bibinfo{number}{2})
  (\bibinfo{year}{2004}) \bibinfo{pages}{511--543}, ISSN
  \bibinfo{issn}{00219991}, \doi{\bibinfo{doi}{10.1016/j.jcp.2003.07.035}}.

\bibitem[{Feng(2006)}]{Feng2006}
\bibinfo{author}{X.~Feng}, \bibinfo{title}{{Fully Discrete Finite Element
  Approximations of the Navier--Stokes--Cahn-Hilliard Diffuse Interface Model
  for Two-Phase Fluid Flows}}, \bibinfo{journal}{SIAM Journal on Numerical
  Analysis} \bibinfo{volume}{44}~(\bibinfo{number}{3}) (\bibinfo{year}{2006})
  \bibinfo{pages}{1049--1072}, ISSN \bibinfo{issn}{0036-1429},
  \doi{\bibinfo{doi}{10.1137/050638333}}.

\bibitem[{Han and Wang(2015)}]{Han2015}
\bibinfo{author}{D.~Han}, \bibinfo{author}{X.~Wang}, \bibinfo{title}{{A second
  order in time, uniquely solvable, unconditionally stable numerical scheme for
  Cahn--Hilliard--Navier--Stokes equation}}, \bibinfo{journal}{Journal of
  Computational Physics} \bibinfo{volume}{290} (\bibinfo{year}{2015})
  \bibinfo{pages}{139--156}, ISSN \bibinfo{issn}{00219991},
  \doi{\bibinfo{doi}{10.1016/j.jcp.2015.02.046}}.

\bibitem[{Shen and Yang(2010{\natexlab{a}})}]{Shen2010a}
\bibinfo{author}{J.~Shen}, \bibinfo{author}{X.~Yang}, \bibinfo{title}{A
  Phase-Field Model and Its Numerical Approximation for Two-Phase
  Incompressible Flows with Different Densities and Viscosities},
  \bibinfo{journal}{SIAM Journal on Scientific Computing}
  \bibinfo{volume}{32}~(\bibinfo{number}{3})
  (\bibinfo{year}{2010}{\natexlab{a}}) \bibinfo{pages}{1159--1179}, ISSN
  \bibinfo{issn}{1064-8275}, \doi{\bibinfo{doi}{10.1137/09075860X}}.

\bibitem[{Shen and Yang(2010{\natexlab{b}})}]{Shen2010b}
\bibinfo{author}{J.~Shen}, \bibinfo{author}{X.~Yang}, \bibinfo{title}{{Energy
  stable schemes for Cahn-Hilliard phase-field model of two-phase
  incompressible flows}}, \bibinfo{journal}{Chinese Annals of Mathematics,
  Series B} \bibinfo{volume}{31}~(\bibinfo{number}{5})
  (\bibinfo{year}{2010}{\natexlab{b}}) \bibinfo{pages}{743--758}, ISSN
  \bibinfo{issn}{0252-9599}, \doi{\bibinfo{doi}{10.1007/s11401-010-0599-y}}.

\bibitem[{Chen and Shen(2016)}]{Chen2016}
\bibinfo{author}{Y.~Chen}, \bibinfo{author}{J.~Shen},
  \bibinfo{title}{{Efficient, adaptive energy stable schemes for the
  incompressible Cahn--Hilliard Navier--Stokes phase-field models}},
  \bibinfo{journal}{Journal of Computational Physics} \bibinfo{volume}{308}
  (\bibinfo{year}{2016}) \bibinfo{pages}{40--56}, ISSN
  \bibinfo{issn}{00219991}, \doi{\bibinfo{doi}{10.1016/j.jcp.2015.12.006}}.

\bibitem[{Shen et~al.(2018)Shen, Xu, and Yang}]{Shen2018}
\bibinfo{author}{J.~Shen}, \bibinfo{author}{J.~Xu}, \bibinfo{author}{J.~Yang},
  \bibinfo{title}{The scalar auxiliary variable (SAV) approach for gradient
  flows}, \bibinfo{journal}{Journal of Computational Physics}
  \bibinfo{volume}{353} (\bibinfo{year}{2018}) \bibinfo{pages}{407--416}.

\bibitem[{Zhu et~al.(2019{\natexlab{a}})Zhu, Chen, Yao, and Sun}]{Zhu2019}
\bibinfo{author}{G.~Zhu}, \bibinfo{author}{H.~Chen}, \bibinfo{author}{J.~Yao},
  \bibinfo{author}{S.~Sun}, \bibinfo{title}{Efficient energy-stable schemes for
  the hydrodynamics coupled phase-field model}, \bibinfo{journal}{Applied
  Mathematical Modelling} \bibinfo{volume}{70}
  (\bibinfo{year}{2019}{\natexlab{a}}) \bibinfo{pages}{82--108}.

\bibitem[{Hughes et~al.(2000)Hughes, Mazzei, and Jansen}]{Hughes2000}
\bibinfo{author}{T.~J. Hughes}, \bibinfo{author}{L.~Mazzei},
  \bibinfo{author}{K.~E. Jansen}, \bibinfo{title}{Large eddy simulation and the
  variational multiscale method}, \bibinfo{journal}{Computing and Visualization
  in Science} \bibinfo{volume}{3}~(\bibinfo{number}{1-2})
  (\bibinfo{year}{2000}) \bibinfo{pages}{47--59}.

\bibitem[{Ahmed et~al.(2017)Ahmed, Chac{\'o}n~Rebollo, John, and
  Rubino}]{Ahmed2017}
\bibinfo{author}{N.~Ahmed}, \bibinfo{author}{T.~Chac{\'o}n~Rebollo},
  \bibinfo{author}{V.~John}, \bibinfo{author}{S.~Rubino}, \bibinfo{title}{A
  Review of Variational Multiscale Methods for the Simulation of Turbulent
  Incompressible Flows}, \bibinfo{journal}{Archives of Computational Methods in
  Engineering} \bibinfo{volume}{24}~(\bibinfo{number}{1})
  (\bibinfo{year}{2017}) \bibinfo{pages}{115--164}, ISSN
  \bibinfo{issn}{1886-1784}, \doi{\bibinfo{doi}{10.1007/s11831-015-9161-0}}.

\bibitem[{Bazilevs et~al.(2007)Bazilevs, Calo, Cottrell, Hughes, Reali, and
  Scovazzi}]{Bazilevs2007}
\bibinfo{author}{Y.~Bazilevs}, \bibinfo{author}{V.~Calo},
  \bibinfo{author}{J.~Cottrell}, \bibinfo{author}{T.~Hughes},
  \bibinfo{author}{A.~Reali}, \bibinfo{author}{G.~Scovazzi},
  \bibinfo{title}{Variational multiscale residual-based turbulence modeling for
  large eddy simulation of incompressible flows}, \bibinfo{journal}{Computer
  Methods in Applied Mechanics and Engineering}
  \bibinfo{volume}{197}~(\bibinfo{number}{1-4}) (\bibinfo{year}{2007})
  \bibinfo{pages}{173--201}.

\bibitem[{Coupez and Hachem(2013)}]{Coupez2013}
\bibinfo{author}{T.~Coupez}, \bibinfo{author}{E.~Hachem},
  \bibinfo{title}{{Solution of high-Reynolds incompressible flow with
  stabilized finite element and adaptive anisotropic meshing}},
  \bibinfo{journal}{Computer Methods in Applied Mechanics and Engineering}
  \bibinfo{volume}{267} (\bibinfo{year}{2013}) \bibinfo{pages}{65--85}, ISSN
  \bibinfo{issn}{00457825}, \doi{\bibinfo{doi}{10.1016/j.cma.2013.08.004}}.

\bibitem[{Hachem et~al.(2013)Hachem, Feghali, Codina, and Coupez}]{Hachem2013}
\bibinfo{author}{E.~Hachem}, \bibinfo{author}{S.~Feghali},
  \bibinfo{author}{R.~Codina}, \bibinfo{author}{T.~Coupez},
  \bibinfo{title}{{Anisotropic adaptive meshing and monolithic Variational
  Multiscale method for fluid-structure interaction}},
  \bibinfo{journal}{Computers and Structures} \bibinfo{volume}{122}
  (\bibinfo{year}{2013}) \bibinfo{pages}{88--100}, ISSN
  \bibinfo{issn}{00457949},
  \doi{\bibinfo{doi}{10.1016/j.compstruc.2012.12.004}}.

\bibitem[{Hachem et~al.(2016)Hachem, Khalloufi, Bruchon, Valette, and
  Mesri}]{Hachem2016}
\bibinfo{author}{E.~Hachem}, \bibinfo{author}{M.~Khalloufi},
  \bibinfo{author}{J.~Bruchon}, \bibinfo{author}{R.~Valette},
  \bibinfo{author}{Y.~Mesri}, \bibinfo{title}{{Unified adaptive Variational
  MultiScale method for two phase compressible-incompressible flows}},
  \bibinfo{journal}{Computer Methods in Applied Mechanics and Engineering}
  \bibinfo{volume}{308} (\bibinfo{year}{2016}) \bibinfo{pages}{238--255}, ISSN
  \bibinfo{issn}{00457825}, \doi{\bibinfo{doi}{10.1016/j.cma.2016.05.022}}.

\bibitem[{Khokhlov(1998)}]{Khokhlov1998}
\bibinfo{author}{A.~M. Khokhlov}, \bibinfo{title}{Fully threaded tree
  algorithms for adaptive refinement fluid dynamics simulations},
  \bibinfo{journal}{Journal of Computational Physics}
  \bibinfo{volume}{143}~(\bibinfo{number}{2}) (\bibinfo{year}{1998})
  \bibinfo{pages}{519--543}.

\bibitem[{Zhu et~al.(2019{\natexlab{b}})Zhu, Kou, Yao, Wu, Yao, and
  Sun}]{Zhu2019therm}
\bibinfo{author}{G.~Zhu}, \bibinfo{author}{J.~Kou}, \bibinfo{author}{B.~Yao},
  \bibinfo{author}{Y.-s. Wu}, \bibinfo{author}{J.~Yao},
  \bibinfo{author}{S.~Sun}, \bibinfo{title}{Thermodynamically consistent
  modelling of two-phase flows with moving contact line and soluble
  surfactants}, \bibinfo{journal}{Journal of Fluid Mechanics}
  \bibinfo{volume}{879} (\bibinfo{year}{2019}{\natexlab{b}})
  \bibinfo{pages}{327--359}.

\bibitem[{Zhu et~al.(2020)Zhu, Kou, Yao, Li, and Sun}]{Zhu2020}
\bibinfo{author}{G.~Zhu}, \bibinfo{author}{J.~Kou}, \bibinfo{author}{J.~Yao},
  \bibinfo{author}{A.~Li}, \bibinfo{author}{S.~Sun}, \bibinfo{title}{A
  phase-field moving contact line model with soluble surfactants},
  \bibinfo{journal}{Journal of Computational Physics} \bibinfo{volume}{405}
  (\bibinfo{year}{2020}) \bibinfo{pages}{109170}.

\bibitem[{Magaletti et~al.(2013)Magaletti, Picano, Chinappi, Marino, and
  Casciola}]{Magaletti2013}
\bibinfo{author}{F.~Magaletti}, \bibinfo{author}{F.~Picano},
  \bibinfo{author}{M.~Chinappi}, \bibinfo{author}{L.~Marino},
  \bibinfo{author}{C.~M. Casciola}, \bibinfo{title}{The sharp-interface limit
  of the {C}ahn--{H}illiard/{N}avier--{S}tokes model for binary fluids},
  \bibinfo{journal}{Journal of Fluid Mechanics} \bibinfo{volume}{714}
  (\bibinfo{year}{2013}) \bibinfo{pages}{95--126},
  \doi{\bibinfo{doi}{10.1017/jfm.2012.461}}.

\bibitem[{Yue et~al.(2010)Yue, Zhou, and Feng}]{Yue2010}
\bibinfo{author}{P.~Yue}, \bibinfo{author}{C.~Zhou}, \bibinfo{author}{J.~Feng},
  \bibinfo{title}{{Sharp-interface limit of the Cahn--Hilliard model for moving
  contact lines}}, \bibinfo{journal}{Journal of Fluid Mechanics}
  \bibinfo{volume}{645} (\bibinfo{year}{2010}) \bibinfo{pages}{279--294}, ISSN
  \bibinfo{issn}{0022-1120}, \doi{\bibinfo{doi}{10.1017/S0022112009992679}}.

\bibitem[{Shen and Yang(2010{\natexlab{c}})}]{Shen2010}
\bibinfo{author}{J.~Shen}, \bibinfo{author}{X.~Yang}, \bibinfo{title}{Numerical
  approximations of {A}llen-{C}ahn and {C}ahn-{H}illiard equations},
  \bibinfo{journal}{Discrete Contin. Dyn. Syst}
  \bibinfo{volume}{28}~(\bibinfo{number}{4})
  (\bibinfo{year}{2010}{\natexlab{c}}) \bibinfo{pages}{1669--1691}.

\bibitem[{Ciarlet(2013)}]{Ciarlet2013}
\bibinfo{author}{P.~G. Ciarlet}, \bibinfo{title}{Linear and Nonlinear
  Functional Analysis with Applications}, \bibinfo{publisher}{Society for
  Industrial and Applied Mathematics}, \bibinfo{address}{Philadelphia, PA,
  USA}, ISBN \bibinfo{isbn}{1611972582, 9781611972580}, \bibinfo{year}{2013}.

\bibitem[{Zeidler(1985)}]{Zeidler1985IIb}
\bibinfo{author}{E.~Zeidler}, \bibinfo{title}{Functional Analysis and its
  Applications}, \bibinfo{journal}{Part II/B (Nonlinear Monotone Operators)
  Springer, Berlin} .

\bibitem[{Liu(2011)}]{Liu2011}
\bibinfo{author}{W.~Liu}, \bibinfo{title}{Existence and uniqueness of solutions
  to nonlinear evolution equations with locally monotone operators},
  \bibinfo{journal}{Nonlinear Analysis: Theory, Methods \& Applications}
  \bibinfo{volume}{74}~(\bibinfo{number}{18}) (\bibinfo{year}{2011})
  \bibinfo{pages}{7543--7561}.

\bibitem[{Oden and Reddy(2011)}]{Oden2012}
\bibinfo{author}{J.~T. Oden}, \bibinfo{author}{J.~N. Reddy}, \bibinfo{title}{An
  Introduction to the Mathematical Theory of Finite Elements},
  \bibinfo{publisher}{Dover}, \bibinfo{year}{2011}.

\bibitem[{Brooks and Hughes(1982)}]{article:BrooksHughes1982}
\bibinfo{author}{A.~Brooks}, \bibinfo{author}{T.~Hughes},
  \bibinfo{title}{{Streamline upwind/Petrov-Galerkin formulations for
  convection dominated flows with particular emphasis on the incompressible
  Navier-Stokes equations}}, \bibinfo{journal}{Computer Methods in Applied
  Mechanics and Engineering} \bibinfo{volume}{32} (\bibinfo{year}{1982})
  \bibinfo{pages}{199--259}.

\bibitem[{Tezduyar et~al.(1992)Tezduyar, Mittal, Ray, and
  Shih}]{article:TezMitRayShi92}
\bibinfo{author}{T.~Tezduyar}, \bibinfo{author}{S.~Mittal},
  \bibinfo{author}{S.~Ray}, \bibinfo{author}{R.~Shih},
  \bibinfo{title}{Incompressible flow computations with stabilized bilinear and
  linear equal-order-interpolation velocity-pressure elements},
  \bibinfo{journal}{Computer Methods in Applied Mechanics and Engineering}
  \bibinfo{volume}{95} (\bibinfo{year}{1992}) \bibinfo{pages}{221--242}.

\bibitem[{Hughes et~al.(2018)Hughes, Scovazzi, and
  Franca}]{hughes2018multiscale}
\bibinfo{author}{T.~J. Hughes}, \bibinfo{author}{G.~Scovazzi},
  \bibinfo{author}{L.~P. Franca}, \bibinfo{title}{Multiscale and stabilized
  methods}, \bibinfo{journal}{Encyclopedia of Computational Mechanics Second
  Edition}  (\bibinfo{year}{2018}) \bibinfo{pages}{1--64}.

\bibitem[{Hughes(1995)}]{Hughes1995}
\bibinfo{author}{T.~J. Hughes}, \bibinfo{title}{Multiscale phenomena: {G}reen's
  functions, the {D}irichlet-to-{N}eumann formulation, subgrid scale models,
  bubbles and the origins of stabilized methods}, \bibinfo{journal}{Computer
  Methods in Applied Mechanics and Engineering}
  \bibinfo{volume}{127}~(\bibinfo{number}{1-4}) (\bibinfo{year}{1995})
  \bibinfo{pages}{387--401}.

\bibitem[{Balay et~al.(1997)Balay, Gropp, McInnes, and Smith}]{petsc-efficient}
\bibinfo{author}{S.~Balay}, \bibinfo{author}{W.~D. Gropp},
  \bibinfo{author}{L.~C. McInnes}, \bibinfo{author}{B.~F. Smith},
  \bibinfo{title}{Efficient Management of Parallelism in Object Oriented
  Numerical Software Libraries}, in: \bibinfo{editor}{E.~Arge},
  \bibinfo{editor}{A.~M. Bruaset}, \bibinfo{editor}{H.~P. Langtangen} (Eds.),
  \bibinfo{booktitle}{Modern Software Tools in Scientific Computing},
  \bibinfo{publisher}{Birkh{\"{a}}user Press}, \bibinfo{pages}{163--202},
  \bibinfo{year}{1997}.

\bibitem[{Balay et~al.(2019{\natexlab{a}})Balay, Abhyankar, Adams, Brown,
  Brune, Buschelman, Dalcin, Dener, Eijkhout, Gropp, Karpeyev, Kaushik,
  Knepley, May, McInnes, Mills, Munson, Rupp, Sanan, Smith, Zampini, Zhang, and
  Zhang}]{petsc-web-page}
\bibinfo{author}{S.~Balay}, \bibinfo{author}{S.~Abhyankar},
  \bibinfo{author}{M.~F. Adams}, \bibinfo{author}{J.~Brown},
  \bibinfo{author}{P.~Brune}, \bibinfo{author}{K.~Buschelman},
  \bibinfo{author}{L.~Dalcin}, \bibinfo{author}{A.~Dener},
  \bibinfo{author}{V.~Eijkhout}, \bibinfo{author}{W.~D. Gropp},
  \bibinfo{author}{D.~Karpeyev}, \bibinfo{author}{D.~Kaushik},
  \bibinfo{author}{M.~G. Knepley}, \bibinfo{author}{D.~A. May},
  \bibinfo{author}{L.~C. McInnes}, \bibinfo{author}{R.~T. Mills},
  \bibinfo{author}{T.~Munson}, \bibinfo{author}{K.~Rupp},
  \bibinfo{author}{P.~Sanan}, \bibinfo{author}{B.~F. Smith},
  \bibinfo{author}{S.~Zampini}, \bibinfo{author}{H.~Zhang},
  \bibinfo{author}{H.~Zhang}, \bibinfo{title}{{PETS}c {W}eb page},
  \bibinfo{howpublished}{\url{https://www.mcs.anl.gov/petsc}},
  \bibinfo{year}{2019}{\natexlab{a}}.

\bibitem[{Balay et~al.(2019{\natexlab{b}})Balay, Abhyankar, Adams, Brown,
  Brune, Buschelman, Dalcin, Dener, Eijkhout, Gropp, Karpeyev, Kaushik,
  Knepley, May, McInnes, Mills, Munson, Rupp, Sanan, Smith, Zampini, Zhang, and
  Zhang}]{petsc-user-ref}
\bibinfo{author}{S.~Balay}, \bibinfo{author}{S.~Abhyankar},
  \bibinfo{author}{M.~F. Adams}, \bibinfo{author}{J.~Brown},
  \bibinfo{author}{P.~Brune}, \bibinfo{author}{K.~Buschelman},
  \bibinfo{author}{L.~Dalcin}, \bibinfo{author}{A.~Dener},
  \bibinfo{author}{V.~Eijkhout}, \bibinfo{author}{W.~D. Gropp},
  \bibinfo{author}{D.~Karpeyev}, \bibinfo{author}{D.~Kaushik},
  \bibinfo{author}{M.~G. Knepley}, \bibinfo{author}{D.~A. May},
  \bibinfo{author}{L.~C. McInnes}, \bibinfo{author}{R.~T. Mills},
  \bibinfo{author}{T.~Munson}, \bibinfo{author}{K.~Rupp},
  \bibinfo{author}{P.~Sanan}, \bibinfo{author}{B.~F. Smith},
  \bibinfo{author}{S.~Zampini}, \bibinfo{author}{H.~Zhang},
  \bibinfo{author}{H.~Zhang}, \bibinfo{title}{{PETS}c Users Manual},
  \bibinfo{type}{Tech. Rep.} \bibinfo{number}{ANL-95/11 - Revision 3.11},
  \bibinfo{institution}{Argonne National Laboratory},
  \bibinfo{year}{2019}{\natexlab{b}}.

\bibitem[{Sundar et~al.(2008{\natexlab{b}})Sundar, Sampath, and
  Biros}]{SundarSampathBiros08}
\bibinfo{author}{H.~Sundar}, \bibinfo{author}{R.~Sampath},
  \bibinfo{author}{G.~Biros}, \bibinfo{title}{Bottom-up construction and 2:1
  balance refinement of linear octrees in parallel}, \bibinfo{journal}{SIAM
  Journal on Scientific Computing} \bibinfo{volume}{30}~(\bibinfo{number}{5})
  (\bibinfo{year}{2008}{\natexlab{b}}) \bibinfo{pages}{2675--2708},
  \doi{\bibinfo{doi}{10.1137/070681727}}.

\bibitem[{Sundar et~al.(2007{\natexlab{b}})Sundar, Sampath, Adavani,
  Davatzikos, and Biros}]{SundarSampathAdavaniEtAl07}
\bibinfo{author}{H.~Sundar}, \bibinfo{author}{R.~S. Sampath},
  \bibinfo{author}{S.~S. Adavani}, \bibinfo{author}{C.~Davatzikos},
  \bibinfo{author}{G.~Biros}, \bibinfo{title}{Low-constant Parallel Algorithms
  for Finite Element Simulations using Linear Octrees}, in:
  \bibinfo{booktitle}{SC'07: Proceedings of the International Conference for
  High Performance Computing, Networking, Storage, and Analysis},
  \bibinfo{publisher}{ACM/IEEE}, \bibinfo{pages}{1--12},
  \bibinfo{year}{2007}{\natexlab{b}}.

\bibitem[{Fernando et~al.(2017)Fernando, Duplyakin, and
  Sundar}]{FernandoDuplyakinSundar17}
\bibinfo{author}{M.~Fernando}, \bibinfo{author}{D.~Duplyakin},
  \bibinfo{author}{H.~Sundar}, \bibinfo{title}{Machine and application aware
  partitioning for adaptive mesh refinement applications}, in:
  \bibinfo{booktitle}{Proceedings of the 26th International Symposium on
  High-Performance Parallel and Distributed Computing},
  \bibinfo{pages}{231--242}, \bibinfo{year}{2017}.

\bibitem[{Bern et~al.(1999)Bern, Eppstein, and Teng}]{bern1999parallel}
\bibinfo{author}{M.~Bern}, \bibinfo{author}{D.~Eppstein},
  \bibinfo{author}{S.-H. Teng}, \bibinfo{title}{Parallel construction of
  quadtrees and quality triangulations}, \bibinfo{journal}{International
  Journal of Computational Geometry \& Applications}
  \bibinfo{volume}{9}~(\bibinfo{number}{6}) (\bibinfo{year}{1999})
  \bibinfo{pages}{517--532}.

\bibitem[{Burstedde et~al.(2011{\natexlab{b}})Burstedde, Wilcox, and
  Ghattas}]{BursteddeWilcoxGhattas11}
\bibinfo{author}{C.~Burstedde}, \bibinfo{author}{L.~C. Wilcox},
  \bibinfo{author}{O.~Ghattas}, \bibinfo{title}{{\texttt{p4est}}: Scalable
  Algorithms for Parallel Adaptive Mesh Refinement on Forests of Octrees},
  \bibinfo{journal}{SIAM Journal on Scientific Computing}
  \bibinfo{volume}{33}~(\bibinfo{number}{3})
  (\bibinfo{year}{2011}{\natexlab{b}}) \bibinfo{pages}{1103--1133},
  \doi{\bibinfo{doi}{10.1137/100791634}}.

\bibitem[{Knuth(1973)}]{Knuth1973Art}
\bibinfo{author}{D.~E. Knuth}, \bibinfo{title}{The Art of Computer Programming.
  Volume {III}. Searching and Sorting}, The Art of Computer Programming,
  \bibinfo{publisher}{Addison-Wesley}, ISBN \bibinfo{isbn}{020103803},
  \urlprefix\url{http://www.worldcat.org/isbn/020103803}, \bibinfo{year}{1973}.

\bibitem[{Sundar et~al.(2012)Sundar, Biros, Burstedde, Rudi, Ghattas, and
  Stadler}]{SundarBirosBurstedde12}
\bibinfo{author}{H.~Sundar}, \bibinfo{author}{G.~Biros},
  \bibinfo{author}{C.~Burstedde}, \bibinfo{author}{J.~Rudi},
  \bibinfo{author}{O.~Ghattas}, \bibinfo{author}{G.~Stadler},
  \bibinfo{title}{Parallel Geometric-algebraic Multigrid on Unstructured
  Forests of Octrees}, in: \bibinfo{booktitle}{Proceedings of the International
  Conference on High Performance Computing, Networking, Storage and Analysis},
  SC '12, \bibinfo{publisher}{IEEE Computer Society Press},
  \bibinfo{address}{Los Alamitos, CA, USA}, ISBN
  \bibinfo{isbn}{978-1-4673-0804-5}, \bibinfo{pages}{43:1--43:11},
  \urlprefix\url{http://dl.acm.org/citation.cfm?id=2388996.2389055},
  \bibinfo{year}{2012}.

\bibitem[{Hysing et~al.(2009)Hysing, Turek, Kuzmin, Parolini, Burman, Ganesan,
  and Tobiska}]{Hysing2009}
\bibinfo{author}{S.-R. Hysing}, \bibinfo{author}{S.~Turek},
  \bibinfo{author}{D.~Kuzmin}, \bibinfo{author}{N.~Parolini},
  \bibinfo{author}{E.~Burman}, \bibinfo{author}{S.~Ganesan},
  \bibinfo{author}{L.~Tobiska}, \bibinfo{title}{Quantitative benchmark
  computations of two-dimensional bubble dynamics},
  \bibinfo{journal}{International Journal for Numerical Methods in Fluids}
  \bibinfo{volume}{60}~(\bibinfo{number}{11}) (\bibinfo{year}{2009})
  \bibinfo{pages}{1259--1288}.

\bibitem[{Aland and Voigt(2012)}]{Aland2012}
\bibinfo{author}{S.~Aland}, \bibinfo{author}{A.~Voigt},
  \bibinfo{title}{{Benchmark computations of diffuse interface models for
  two-dimensional bubble dynamics}}, \bibinfo{journal}{International Journal
  for Numerical Methods in Fluids} \bibinfo{volume}{69}~(\bibinfo{number}{3})
  (\bibinfo{year}{2012}) \bibinfo{pages}{747--761},
  \doi{\bibinfo{doi}{10.1002/fld.2611}},
  \urlprefix\url{http://doi.wiley.com/10.1002/fld.2611}.

\bibitem[{Yuan et~al.(2017)Yuan, Chen, Shu, Wang, Niu, and Shu}]{Yuan2017}
\bibinfo{author}{H.~Z. Yuan}, \bibinfo{author}{Z.~Chen},
  \bibinfo{author}{C.~Shu}, \bibinfo{author}{Y.~Wang}, \bibinfo{author}{X.~D.
  Niu}, \bibinfo{author}{S.~Shu}, \bibinfo{title}{{A free energy-based surface
  tension force model for simulation of multiphase flows by level-set method}},
  \bibinfo{journal}{Journal of Computational Physics} \bibinfo{volume}{345}
  (\bibinfo{year}{2017}) \bibinfo{pages}{404--426}, ISSN
  \bibinfo{issn}{10902716}, \doi{\bibinfo{doi}{10.1016/j.jcp.2017.05.020}},
  \urlprefix\url{http://dx.doi.org/10.1016/j.jcp.2017.05.020}.

\bibitem[{Tryggvason and Unverdi(1990)}]{Tryggvason1990}
\bibinfo{author}{G.~Tryggvason}, \bibinfo{author}{S.~O. Unverdi},
  \bibinfo{title}{Computations of three-dimensional {R}ayleigh--{T}aylor
  instability}, \bibinfo{journal}{Physics of Fluids A: Fluid Dynamics}
  \bibinfo{volume}{2}~(\bibinfo{number}{5}) (\bibinfo{year}{1990})
  \bibinfo{pages}{656--659}.

\bibitem[{Li et~al.(1996)Li, Jin, and Glimm}]{Li1996}
\bibinfo{author}{X.~Li}, \bibinfo{author}{B.~Jin}, \bibinfo{author}{J.~Glimm},
  \bibinfo{title}{Numerical study for the three-dimensional
  {R}ayleigh--{T}aylor instability through the {TVD/AC} scheme and parallel
  computation}, \bibinfo{journal}{Journal of Computational Physics}
  \bibinfo{volume}{126}~(\bibinfo{number}{2}) (\bibinfo{year}{1996})
  \bibinfo{pages}{343--355}.

\bibitem[{Guermond and Quartapelle(2000)}]{Guermond2000}
\bibinfo{author}{J.-L. Guermond}, \bibinfo{author}{L.~Quartapelle},
  \bibinfo{title}{A projection FEM for variable density incompressible flows},
  \bibinfo{journal}{Journal of Computational Physics}
  \bibinfo{volume}{165}~(\bibinfo{number}{1}) (\bibinfo{year}{2000})
  \bibinfo{pages}{167--188}.

\bibitem[{Tryggvason(1988)}]{Tryggvason1988}
\bibinfo{author}{G.~Tryggvason}, \bibinfo{title}{Numerical simulations of the
  Rayleigh-Taylor instability}, \bibinfo{journal}{Journal of Computational
  Physics} \bibinfo{volume}{75}~(\bibinfo{number}{2}) (\bibinfo{year}{1988})
  \bibinfo{pages}{253--282}.

\bibitem[{Ding et~al.(2007)Ding, Spelt, and Shu}]{Ding2007}
\bibinfo{author}{H.~Ding}, \bibinfo{author}{P.~D. Spelt},
  \bibinfo{author}{C.~Shu}, \bibinfo{title}{Diffuse interface model for
  incompressible two-phase flows with large density ratios},
  \bibinfo{journal}{Journal of Computational Physics}
  \bibinfo{volume}{226}~(\bibinfo{number}{2}) (\bibinfo{year}{2007})
  \bibinfo{pages}{2078--2095}.

\bibitem[{Bhaga and Weber(1981)}]{Bhaga1981}
\bibinfo{author}{D.~Bhaga}, \bibinfo{author}{M.~Weber}, \bibinfo{title}{Bubbles
  in viscous liquids: shapes, wakes and velocities}, \bibinfo{journal}{Journal
  of fluid Mechanics} \bibinfo{volume}{105} (\bibinfo{year}{1981})
  \bibinfo{pages}{61--85}.

\bibitem[{Tryggvason et~al.(2006)Tryggvason, Lu, Biswas, and
  Esmaeeli}]{Tryggvason2006}
\bibinfo{author}{G.~Tryggvason}, \bibinfo{author}{J.~Lu},
  \bibinfo{author}{S.~Biswas}, \bibinfo{author}{A.~Esmaeeli},
  \bibinfo{title}{Direct Numerical Simulations of Bubbly Flows}, in:
  \bibinfo{editor}{S.~Balachandar}, \bibinfo{editor}{A.~Prosperetti} (Eds.),
  \bibinfo{booktitle}{Proceedings of the IUTAMSymposium on Computational
  Multiphase Flow}, \bibinfo{publisher}{Springer Berlin Heidelberg},
  \bibinfo{pages}{273--281}, \bibinfo{year}{2006}.

\bibitem[{Hua and Lou(2007)}]{Hua2007}
\bibinfo{author}{J.~Hua}, \bibinfo{author}{J.~Lou}, \bibinfo{title}{Numerical
  simulation of bubble rising in viscous liquid}, \bibinfo{journal}{Journal of
  Computational Physics} \bibinfo{volume}{222}~(\bibinfo{number}{2})
  (\bibinfo{year}{2007}) \bibinfo{pages}{769--795}.

\bibitem[{Balc{\'a}zar et~al.(2015)Balc{\'a}zar, Lehmkuhl, Jofre, and
  Oliva}]{Balcazar2015}
\bibinfo{author}{N.~Balc{\'a}zar}, \bibinfo{author}{O.~Lehmkuhl},
  \bibinfo{author}{L.~Jofre}, \bibinfo{author}{A.~Oliva},
  \bibinfo{title}{Level-set simulations of buoyancy-driven motion of single and
  multiple bubbles}, \bibinfo{journal}{International Journal of Heat and Fluid
  Flow} \bibinfo{volume}{56} (\bibinfo{year}{2015}) \bibinfo{pages}{91--107}.

\bibitem[{Yan et~al.(2018)Yan, Lin, Bazilevs, and Wagner}]{Yan2018}
\bibinfo{author}{J.~Yan}, \bibinfo{author}{S.~Lin},
  \bibinfo{author}{Y.~Bazilevs}, \bibinfo{author}{G.~J. Wagner},
  \bibinfo{title}{Isogeometric analysis of multi-phase flows with surface
  tension and with application to dynamics of rising bubbles},
  \bibinfo{journal}{Computers \& Fluids} .

\end{thebibliography}


\appendix
\section{Proofs of some elementary propositions}
\label{sec:appendix1}
\begin{proposition} The following identity holds:
		\begin{equation}
		\pd{\tilde{\phi}^{k}}{x_j}\left(\pd{}{x_j}\left({\pd{\tilde{\phi}^{k}}{x_i}}\right)\right) = \frac{1}{2}\pd{}{x_j}\left(  \pd{\tilde{\phi}^{k}}{x_i} \pd{\tilde{\phi}^{k}}{x_i} \right)
		\end{equation}
			$\forall \;\; \widetilde{\phi}^k$, $\in H^1(\Omega)$, where $\phi^k, \phi^{k+1}, \mu^{k},\mu^{k+1}, \vec{v}^k, \vec{v}^{k+1}$ solves \cref{eqn:nav_stokes_var_semi_disc} -- \cref{eqn:phi_eqn_var_semi_disc}.
\label{prop:forcing_aside}
\end{proposition}
\begin{proof}
	We just need to recall a vector identity to prove this.
	Recall the vector identities
	\begin{align}
	\frac{1}{2}\pd{\left(A_i A_i\right)}{x_j} &= \left(A_i\pd{}{x_i}\right)A_j + \epsilon_{ijk}A_j\left(\epsilon_{klm}\pd{A_m}{x_l}\right),\label{eqn:vec_id_cross}\\
	\epsilon_{ijk}\pd{}{x_j}\left( \pd{f}{x_k}\right) &= 0, \label{eqn:cross_zero_scalar}
	\end{align}
	where $f$ is a scalar function and $\epsilon$ is the Levi-Civita symbol.
	In our case $A_j = \pd{\tilde{\phi}^{k}}{x_j}$, which causes the cross product term in \cref{eqn:vec_id_cross} to be zero from \cref{eqn:cross_zero_scalar}, and which leads to the desired result.
\end{proof}

\begin{lemma} 
	The variational evolution term from the Cahn-Hilliard contribution in the momentum equation, 
	\cref{eqn:nav_stokes_var_semi_disc}, can be written as follows:  
	\begin{equation}
	\begin{split}
	\left(\rho\left(\phi^{k+1}\right)\left(v^{k+1}_i - v^k_i\right), \tvi^k\right) & = 
	\quad \frac{1}{2}\int_{\Omega} 
	\left\{ \rho\left(\phi^{k+1}\right)\norm{{\vec{v}}^{k+1}}^2 - \rho\left(\phi^{k}\right)\norm{\vec{v}^{k}}^2 \right\}  \, d\vec{x}
	\end{split}
	\label{eqn:evolution_term_lemma}
	\end{equation}
	$\forall \;\; \tphi^k$, $\phi^{k}$, $\phi^{k+1} \in  H^1(\Omega)$, and $\forall \;\; {\vec{v}}^k, {\vec{v}}^{k+1} \in \vec{H}_{0}^1(\Omega)$, where  ${\vec{v}}^k, {\vec{v}}^{k+1}, \phi^k, \phi^{k+1}$ satisfy \cref{eqn:nav_stokes_var_semi_disc} -- \cref{eqn:phi_eqn_var_semi_disc},
	and
	\begin{equation}
	\begin{split}
	\norm{\vec{v}}^2 &:= \sum_i \abs{v_i}^2.
	\end{split}
	\end{equation}
	\label{lem:evol_estimate}
\end{lemma}
\begin{proof}
We start with the left-hand side and expand into two terms:
	\begin{align}
	\begin{split}
	& \left(\rho\left(\phi^{k+1}\right)\left(v^{k+1}_i - v^k_i\right), \tvi^k\right) = 
	\frac{1}{2}\int_{\Omega}\left\{ \rho\left(\phi^{k+1}\right)\norm{{\vec{v}}^{k+1}}^2 - \rho\left(\phi^{k+1}\right)\norm{\vec{v}^{k}}^2 \right\}  \, d\vec{x},
	\end{split}
	\label{eqn:evol_inner_prod}
	\end{align}
	where we made use of the fact that $\tvi^k = (v_i^{k+1}+v_i^k)/2$.  
	Next we add and subtract the same term:
	\begin{equation*}
	\label{eqn:evol_inner_prod2}
	\begin{split}
	 \left(\rho\left(\phi^{k+1}\right)\left(v^{k+1}_i - v^k_i\right), \tvi^k\right) &=   
	\frac{1}{2}\int_{\Omega}\left\{ \rho\left(\phi^{k+1}\right)\norm{{\vec{v}}^{k+1}}^2 - \rho\left(\phi^{k+1}\right)\norm{\vec{v}^{k}}^2 \right\}  \, d\vec{x}\\
	&+ \frac{1}{2}\int_{\Omega}\left\{ \rho\left(\phi^{k}\right)\norm{{\vec{v}}^{k}}^2 - \rho\left(\phi^{k}\right)\norm{\vec{v}^{k}}^2 \right\}  \, d\vec{x} \\
	&=  \frac{1}{2}\int_{\Omega}\left\{ \rho\left(\phi^{k+1}\right)\norm{{\vec{v}}^{k+1}}^2  -  \rho\left(\phi^{k}\right)\norm{\vec{v}^{k}}^2 \right\}  \, d\vec{x}\\
	&- \frac{1}{2}\int_{\Omega}\left\{ \rho\left(\phi^{k+1}\right) - \rho\left(\phi^{k}\right) \right\} \norm{\vec{v}^{k}}^2 \, d\vec{x}.
\end{split}
	\end{equation*}
Using \cref{eqn:disc_consv_semi} we can evaluate the last term in the above equation as follows:  
	\begin{align}
	\begin{split}
	- \frac{1}{2}\int_{\Omega}\left\{ \rho\left(\phi^{k+1}\right) - \rho\left(\phi^{k}\right) \right\} \norm{\vec{v}^{k}}^2 \, d\vec{x} &=   
	\frac{1}{2}\int_{\Omega} \pd{}{x_i} \left\{ \rho\left(\tphi^{k} \right)\tvi^{k} + \frac{1}{Pe}\tJi^{k}  \right\}  \norm{\vec{v}^{k}}^2 \, d\vec{x}\\
	&= - \frac{1}{2}\int_{\Omega} \left\{ \rho\left(\tphi^{k} \right)\tvi^{k} + \frac{1}{Pe}\tJi^{k}  \right\} \pd{}{x_i}\norm{\vec{v}^{k}}^2 \, d\vec{x}\\
	&= - \frac{1}{2}\int_{\Omega} \left\{ \rho\left(\tphi^{k} \right)\tvi^{k} + \frac{1}{Pe}\tJi^{k}  \right\} 2 v_i \pd{v_i^{k}}{x_i} \, d\vec{x}\\
	&= 0,
	\end{split}
	\label{eqn:evol_inner_prod3}
	\end{align}
	where we used the solenoidality of the mixture velocity (see~\cref{eqn:solenoidality}).  This gives the desired result
	\cref{eqn:evolution_term_lemma}.
\end{proof}

\section{Details of solver selection for the numerical experiments}
\label{sec:app_linear_solve}
For the cases presented in \cref{subsec:single_rising_drop} we use the BiCGStab linear solver (a Krylov space solver) with additive Schwarz-based preconditioning.  For better reproduction, the command line options we provide {\sc petsc} are given below which include some commands used for printing some norms as well.
\begin{lstlisting}
-ns_ksp_type bcgs
-ns_pc_type asm
-ch_ksp_type bcgs
-ch_pc_type asm
-ns_snes_monitor 
-ns_snes_converged_reason 
-ns_ksp_converged_reason
-ch_snes_monitor
-ch_snes_converged_reason
-ch_ksp_converged_reason
\end{lstlisting}
Here the prefix \codeword{-ch} is for applying the option to the Cahn-Hilliard solver, and \codeword{-ns} for the momentum solver respectively.

For the Rayleigh-Taylor instability case (\cref{subsec:rayleigh_taylor}), which was significantly more expensive, we used an algebraic multigrid (AMG) linear solver with an additive Schwarz method (ASM) as a smoother. To improve the readability we separate the options for momentum and Cahn-Hilliard equations in two separate structures to input them into {\sc petsc} 
and the options are shown below.
\begin{lstlisting}
solver_options_ns = {
	snes_atol = 1e-4
	snes_rtol = 1e-6
	snes_stol = 1e-5
	snes_max_it = 40
	ksp_rtol = 1e-5
	ksp_diagonal_scale = True
	ksp_diagonal_scale_fix = True

	#multigrid

		#solver selection
		ksp_type = "fgmres"
		pc_type = "gamg"
		pc_gamg_asm_use_agg = True
		mg_levels_ksp_type = "bcgs"
		mg_levels_pc_type = "asm"
		mg_levels_sub_pc_type = "lu"
	  	
	#performance options
		mattransposematmult_via = "matmatmult"
		pc_gamg_reuse_interpolation = "True"
		mg_levels_ksp_max_it = 20
};

solver_options_ch = {
	snes_atol = 1e-12
	snes_rtol = 1e-8
	snes_stol = 1e-10
	snes_max_it = 20
  
	# multigrid
		ksp_type = "fgmres"
		pc_type = "gamg"
		pc_gamg_asm_use_agg = True
		mg_levels_ksp_type = "bcgs"
		mg_levels_pc_type = "asm"
	#performance options
		mattransposematmult_via = "matmatmult"
		pc_gamg_reuse_interpolation = "True"
		mg_levels_ksp_max_it = 4
};
\end{lstlisting}

The linear systems we handle are fairly ill-conditioned, therefore, the smoothers we need to use are fairly expensive.  The ASM/LU based smoother is more expensive compared to other smoothers like block Jacobi, however ASM/LU is more robust (better convergence).  This setup works very well with a relatively constant number of Krylov iterations as the number of processes are increased in the massively parallel setting. The scaling results we present use the same setup of solvers,  but there is substantial room for improvement in this area of the code where fieldsplit preconditioners using Schur complement can be used as smoothers to improve speed of the AMG solver.

\end{document}